\documentclass{article}
\usepackage[paper=a4paper,left=30mm,right=30mm,top=25mm,bottom=25mm]{geometry}
\usepackage{amsmath,amsthm,amssymb,mathtools}
\usepackage{enumitem}
\usepackage{booktabs}
\usepackage{graphicx}
\usepackage{float}
\usepackage{placeins}
\usepackage{csquotes}
\usepackage[numbers]{natbib}
\usepackage{hyperref}
\usepackage{color}
\usepackage{comment}

\newtheorem{theorem}{Theorem}[section]
\newtheorem{lemma}[theorem]{Lemma}
\newtheorem{corollary}[theorem]{Corollary}
\newtheorem{proposition}[theorem]{Proposition}
\newtheorem{definition}[theorem]{Definition}
\newtheorem{remark}[theorem]{Remark}
\newtheorem{example}[theorem]{Example}

\newcommand{\CC}{\mathcal{C}}
\newcommand{\de}{\mathrm{\,d}}
\newcommand{\lam}{\lambda}
\newcommand{\brho}{\overline\varrho}
\newcommand{\Var}{\operatorname{Var}}

\newcommand{\E}{\mathbb{E}}

\newcommand{\N}{\mathbb{N}}

\newcommand{\R}{\mathbb{R}}

\newcommand{\cM}{\mathcal{M}}

\newcommand{\cS}{\mathcal{S}}
\newcommand{\cU}{\mathcal{U}}

\newcommand{\eqd}{\stackrel{\mathrm{d}}=}

\definecolor{light-gray}{gray}{0.95}
\definecolor{darkblue}{rgb}{0,0,.5}
\definecolor{foxred}{rgb}{0.7, 0.11, 0.11}

\title{
The exact Spearman rho--footrule region via optimal transport with applications to finite rankings, mixability, and Chatterjee's rank correlation

}
\author{%
	Jonathan Ansari\thanks{Department of Mathematics, Paris Lodron Universit\"at Salzburg,
	Hellbrunner Stra\ss e 34, 5020 Salzburg, Austria.}
	\and
	Marcus Rockel\thanks{Department of Quantitative Finance, Albert-Ludwigs-Universit\"at Freiburg,
	Rempartstra\ss e 16, 79098 Freiburg, Germany.}%
}
\date{\today}

\begin{document}
\maketitle

\begin{abstract}
We solve the open problem of determining the maximal value of Spearman's rho when Spearman's footrule is prescribed, thereby completing the exact attainable region of these two quantities.
To prove this result, we reformulate the underlying copula optimization problem as an optimal transport problem with a linear moment constraint and construct the unique optimal coupling through a matching feasible dual potential and its contact set.
Equivalently, this coupling minimizes the variance of \(|U-V|\) among all couplings of \(U,V\sim \cU(0,1)\) with prescribed mean
\(\E|U-V|\).
Our main result admits several applications: 
First, in the context of finite rankings, we obtain an improved Cauchy--Schwarz inequality between Spearman's footrule distance and the associated quadratic rank difference.
Second, in the framework of generalized mixability, we characterize the
attainable constant values of \(|U'+V'|\), for \(U',V'\sim \cU(-\tfrac 1 2, \tfrac 1 2)\), and determine the minimal quadratic deviation for a given mean.
Third, we derive explicit bounds relating Chatterjee's rank correlation
\(\xi(X,Y)\), which can detect complex functional dependence of \(Y\) on \(X\), to the copula correlation ratio---a rank-based fraction of explained variance---by exploiting their conditional i.i.d.\ representations in terms of Spearman's footrule and Spearman's rho.

\end{abstract}


\section{Introduction}\label{sec:intro}

Measures of association quantify the dependence structure between two random variables \(X\) and \(Y\).
Classical rank correlations such as Spearman's rho and Kendall's tau attain values in \([-1,1]\) and
describe the degree of positive or negative dependence. In
particular, they are designed to capture the strength of
    \emph{monotone} association; see e.g. \cite[Section 2.4]{durante2016principles} for an overview.
A different notion of dependence has received considerable attention in
recent years. Dependence measures such as Chatterjee's rank correlation \cite{chatterjee2021new}
quantify the degree of \emph{functional dependence} of \(Y\) on \(X\).
They take values in \([0,1]\), where the value \(0\) characterizes
independence and the value \(1\) characterizes perfect functional
dependence, meaning that \(Y\) is almost surely a measurable function of
\(X\). Importantly, this function need not be monotone, but can be arbitrarily complex. Such measures
therefore detect forms of dependence that may be invisible to classical
rank correlations; see \cite{chatterjee2024survey} for a review.

Generally, measures of association capture distinct aspects of dependence and
therefore provide complementary information about the underlying
dependence structure.
A natural question is to what extent their values constrain one another
and how they can be interpreted in terms of each other. Given two measures of association \(\kappa_1\) and \(\kappa_2\), one may consider their
\emph{exact attainable region}
\begin{align*}
    \Omega_{\kappa_1,\kappa_2}
    :=
    \bigl\{
        (\kappa_1(C),\kappa_2(C)):C\in\CC
    \bigr\},
\end{align*}
where \(\CC\) denotes the class of bivariate copulas. We generally assume that \(X\) and \(Y\) have a continuous distribution function, so that by Sklar's theorem (see Eq. \eqref{the_Sklar}) their dependence structure is fully captured by a unique copula \(C=C_{X,Y}\). The region \(\Omega_{\kappa_1,\kappa_2}\)
describes all pairs of values that can occur simultaneously and thus
provides a sharp quantitative comparison of the two measures.
Following the solution of the classical Kendall--Spearman problem by
\citet{schreyer2017exact}, exact regions have been determined for
several pairs of measures of association; see e.g.
\cite{kokolbukovsek2021spearman,kokolbukovsek2022exact,
kokolbukovsek2023exact,ansari2026exact}.

In this paper, we solve the open problem of maximizing Spearman's rho when the value of Spearman's footrule is prescribed. While Spearman's rho is a well-established rank correlation, Spearman's footrule is less familiar, but no less relevant; see \cite{genest2010spearman} and the applications below. 
Both admit a representation in terms of the underlying copula \(C = C_{X,Y}\) via
\begin{equation}\label{eq:defs}
     \varrho(C)
    \coloneqq
    12\int_{[0,1]^2}C(u,v)\,\de\lam_2(u,v)-3,
    \qquad
    \phi(C)
    \coloneqq
    6\int_0^1 C(t,t)\,\de t-2;
\end{equation}
see \cite{durante2016principles,nelsen2006introduction} for general background.
Further, they are affine in $C$ and continuous with respect to uniform convergence, so the region $\Omega_{\varrho,\phi}=\{(\varrho(C),\phi(C)):C\in\CC\}$ is convex and compact.
While \(\varrho\) attains values from \(-1\) to \(1\), \(\phi\) ranges from \(-\tfrac 1 2\) to \(1\).
\citet{kokolbukovsek2024exact} determined the $\varrho$-minimal boundary \(\underline{\varrho}(x):= \tfrac{2\sqrt3}{9}(1+2x)^{3/2}-1\) of $\Omega_{\varrho,\phi}$ exactly and established the upper bound \(u(x) \coloneqq 1-\tfrac{2}{3}(1-x)^2\), so that
\begin{equation}\label{eq:ksbound}
	\underline{\varrho}(\phi(C)) \leq \varrho(C) \le u(\phi(C)) \quad \text{for all } C\in \CC.
\end{equation}
The lower bound \(\underline{\varrho}(x)\) is attained for all \(x\in[-\tfrac12,1]\) by the Bertino copulas \(B_x\) in \eqref{def:lower-bertino}.
In contrast, the upper bound \(u(x)\) is attained at the countably many points $x_N\coloneqq 1-\tfrac{3}{2N}$, $N\ge1$, by equidistant even shuffles of min \cite{mikusinski1992shuffles}, as well as at $x=1$, and it is not globally sharp.
\citet{tschimpke2025revisiting} gave alternative proofs of both bounds and constructed the attainable curve in \eqref{curve_s}, which, however, is below the optimal, yet unknown curve
\begin{align}\label{defmaxprob}
    \overline{\varrho}(x) := \max\{ \varrho(C) \mid C\in \CC,\, \phi(C) = x\}, \quad \text{} x\in [-\tfrac 1 2,1].
\end{align}

In our main result, Theorem~\ref{thm:main}, we solve the optimization problem \eqref{defmaxprob} and derive a closed-form expression for $\overline{\varrho}$. Further, for every $x\in[-\tfrac12,1]$, we explicitly construct a copula $C_x$ such that $\phi(C_x)=x$ and $\varrho(C_x)=\overline{\varrho}(x)$ and prove that it is unique.
Key is the moment representation
\begin{align}\label{eq_rep_phi_rho}
    \varrho(C)
    =
    1-6\E|U-V|^2
    \qquad \text{and} \qquad
    \phi(C)
    =
    1-3\E|U-V| \qquad \text{for } (U,V)\sim C; 
\end{align}
see Lemma~\ref{lem:moments}. This allows us to solve the copula maximization problem \eqref{defmaxprob} with tools from optimal transport theory: we minimize \(\E |U-V|^2\) over all couplings of \(U,V\sim \cU(0,1)\) under the linear constraint \(\E|U-V|=m\); see the paragraph after Theorem~\ref{thm:main} for a sketch of the proof that we carry out in Sections \ref{sec:attain} and \ref{sec:dual}.

The moment representation of \(\varrho\) and \(\phi\) in \eqref{eq_rep_phi_rho} and the closed-form expression of \(\brho\) in  \eqref{eq:boundary} admit several interesting consequences and applications.
First, we obtain an improved Cauchy--Schwarz inequality for finite rankings. For a permutation
\(\pi\) of \(\{1,\ldots,n\}\), let
\begin{align*}
    D_\pi
    :=
    \sum_{i=1}^n |i-\pi(i)| \qquad \text{and} \qquad
    S_\pi
    :=
    \sum_{i=1}^n (i-\pi(i))^2
\end{align*}
denote the Spearman footrule distance and the corresponding quadratic rank differences, respectively; see e.g. \cite{clemencon2013ranking,diaconis1977spearman}. Then, by Theorem \ref{cor:permutation-costs}, the term \(V_{\operatorname{min}}(m)\geq 0\) in \eqref{def_Vmin} sharpens the elementary
Cauchy--Schwarz inequality \(D_\pi^2/n\leq S_\pi\) in \cite[p.~268]{diaconis1977spearman} to
\[
    \frac{D_\pi^2}{n} + n^3 V_{\operatorname{min}}(m) \leq S_\pi \qquad \text{for } m = \frac{D_\pi}{n^2}.
\]
We refer to Section \ref{sec_21} for details.

A second application arises in the context of mixability. The standard notion of mixability asks whether a sum of identically distributed random variables can be constant; see e.g. \cite{wang2015current}. Generalized mixability in \cite{bignozzi2015studying} extends this concept and studies supermodular functions of random vectors. We consider the supermodular function \(\psi(u,v) = |u+v|\) and obtain that 
\begin{align}\label{eq_gen_mix}
    \psi(U',V') = m \text{ a.s.} \quad \text{for } U',V'\sim \cU(-\tfrac 1 2, \tfrac 1 2) \quad \Longleftrightarrow\quad m\in \mathfrak{C} := \{0\}\cup
        \left\{\frac{1}{2N}\colon N\in\N\right\}.
\end{align}
More generally, we determine, for a prescribed value \(m\in [0,\tfrac 1 2]\), the minimal quadratic deviation of \(\psi(U',V')\) to \(m\) over all couplings of \(U',V'\sim \cU(-\frac 1 2,\frac 1 2)\); see Theorem \ref{cor:distance-mixability}.
Interestingly, there are countably many centers \(m\) for the \(\cU(-\tfrac 1 2, \tfrac 1 2)\)-distribution with respect to the mixing function \(\psi\), whereas standard mixability for an integrable distribution admits at most one center.

A third application of Theorem \ref{thm:main} arises for Chatterjee's rank correlation \(\xi\) in \eqref{def_chattxi}, recently introduced in \cite{chatterjee2021new}. As mentioned above, it ranges from \(0\) to \(1\) and can detect arbitrary functional relationships of \(Y\) on \(X\). However, there are still open questions, for example, on the interpretation of its values. A natural question is how much variance of \(Y\) can be explained by \(X\) when the value of \(\xi(X,Y)\) is known. To give a first answer to this question, we consider the copula correlation ratio \(\eta(X,Y)\) in \eqref{def_cop_cor_ratio}, a rank-transformed fraction of explained variance \cite{shih2021copula}. Interestingly, \(\xi\) and \(\eta\) admit representations in terms of Spearman's footrule and rho via
\begin{align}
    \xi(C) = \phi(C\ast C) \qquad \text{and} \qquad \eta(C) = \varrho(C\ast C);
\end{align}
see \eqref{eq_xi_DSS} and \eqref{eq_rep_ccorr}, where \(C\ast C\) denotes the Markov product in \eqref{def_MK_product}.
Applying the exact \((\varrho,\phi)\)-region in Corollary \ref{cor:main} and Cauchy--Schwarz inequality, we obtain the bounds 
\begin{equation}\label{eq:xi-rank-correlation-ratio}
    \max\!\left\{
        0,\underline{\varrho}\bigl(\xi(X,Y)\bigr)
    \right\}
    \leq
    \eta(X,Y)
    \leq
    \min\!\left\{
        \overline{\varrho}\bigl(\xi(X,Y)\bigr),
        2\xi(X,Y)
    \right\},
\end{equation}
see Theorem \ref{cor:xi-rank-sobol}.
These bounds are quite tight due to Proposition \ref{prop:xi-rank-sobol-inner} where we constructed an inner enclosure of the \((\xi,\eta)\)-region. For an illustration of the \(\xi\)-\(\eta\)-bounds, we refer to Figure \ref{fig:xi-rank-sobol-bounds}.

\subsection{Main result}

For \(N\in \N\), recall \(x_N = 1 - \frac{3}{2N}\) and define 
\begin{equation}\label{eq:intervals}
	I_N \coloneqq [x_N,x_{N+1}) = \left[1-\frac{3}{2N},1-\frac{3}{2N+2}\right).
\end{equation}
These intervals cover $[-\tfrac12,1)$ and are pairwise disjoint.
The following theorem, visualized in Figure~\ref{fig:region}, is our main result.

\begin{theorem}[Closed-form expression for \(\brho\)]\label{thm:main}
	Let $x\in I_N$, put
	\(
	m\coloneqq\tfrac{1-x}{3}\in[\tfrac{1}{2N+2},\tfrac{1}{2N}],
	\)
	let $\ell\in\{\tfrac{1}{2N},\tfrac{1}{2N+2}\}$ be an endpoint nearest to $m$ (if the two endpoints are equally near, either choice is allowed; the two values of \eqref{eq:boundary} coincide), and set
	\(
	\Delta\coloneqq|m-\ell|.
	\)
    Then we have
	\begin{equation}\label{eq:boundary}
		\brho(x) = 1 - 6\ell(2m-\ell) - \frac{4\,\Delta^{3/2}}{\sqrt{N(N+1)}}
		= u(x) - \Bigl(\frac{4\,\Delta^{3/2}}{\sqrt{N(N+1)}} - 6\Delta^2\Bigr).
	\end{equation}
    In particular, the symmetric copula \(C_x\) in \eqref{def_C_x} satisfies \(\varrho(C_x) = \brho(x)\) and \(\phi(C_x) = x\).
    Moreover, it is the unique copula with these two properties.
    At the remaining endpoint \(x=1\), we have \(\brho(1)=1\), uniquely attained by the upper Fr\'{e}chet copula \(C_1(u,v)=\min\{u,v\}\).
\end{theorem}

To prove Theorem \ref{thm:main}, we solve the copula maximization problem \eqref{defmaxprob} by translating it into an optimal transport problem and then applying standard tools from OT theory.
To briefly explain our ideas, let us recall the moment representation \(\varrho\) and \(\phi\) in \eqref{eq_rep_phi_rho}. 
Hence, the \(\varrho\)-maximal boundary of \(\Omega_{\varrho,\phi}\) corresponds to minimizing \(\E|U-V|^2\) under all random vectors \((U,V)\) with \(U,V\sim \cU(0,1)\) at a prescribed value \(\E|U-V|\).
For \(m\in[0,\tfrac12]\), this is the OT problem with a linear constraint
\begin{align}\label{eq:OT_primal}
    \mathsf{P}(m)\coloneqq\min_{\pi\in \Pi}\left\{ \int_{[0,1]^2} (b-a)^2 \de \pi(a,b)  \colon \int_{[0,1]^2} |b-a| \de \pi(a,b) = m\right\},
\end{align}
where \(\Pi\) denotes the set of couplings of two \(\cU(0,1)\) distributions, i.e., the set of distributions on \([0,1]^2\) with uniform marginals.
Recall that the distribution function associated with a coupling \(\pi\in \Pi\) is a (bivariate) copula.
Conversely, the probability measure induced by a bivariate copula is an element of \(\Pi\).
Hence, for \(m = \frac{1-x}{3}\), the OT problem \eqref{eq:OT_primal} is a reformulation of the copula problem \eqref{defmaxprob}.
To solve the \emph{primal problem} \eqref{eq:OT_primal}, we consider its dual formulation
\begin{align}\label{eq:OT_dual}
    \sup_{\varphi \oplus \psi \leq c_\theta} \int_0^1 \varphi(a) \de a + \int_0^1 \psi(b) \de b ,
\end{align}
where the optimization is over all continuous functions \(\varphi,\psi\colon [0,1]\to \R\) and constants \(\theta\in \R\) such that
\[
    \varphi\oplus\psi(a,b)\coloneqq\varphi(a)+\psi(b)
    \leq (b-a)^2-\theta\bigl(|b-a|-m\bigr)
    \eqqcolon c_\theta(a,b)
\]
for all \(a,b\in[0,1]\).
Kantorovich duality with linear constraints yields equality of the primal and dual values; see \cite[Theorem~2.1]{zaev2015monge}.

Since \(\varrho\) and \(\phi\) are symmetric, the transport cost \(c_\theta\) is symmetric.
Since, additionally, the optimization in \eqref{eq:OT_primal} is over distributions having identical marginals, the dual problem \eqref{eq:OT_dual} can be rewritten as
\begin{align}\label{eq:OT_dual2}
    \mathsf{D}(m)\coloneqq
    \sup_{\substack{f\in C([0,1]),\ \theta\in\R\\ f\oplus f\leq c_\theta}}
    2\int_0^1 f(u)\de u;
\end{align}
see Proposition~\ref{prop_KR_duality}.
Here \(C([0,1])\) denotes the space of continuous real-valued functions on \([0,1]\).
To determine the optimal value \(\mathsf{P}(m)\), we provide a feasible coupling \(\pi_x\) and a feasible potential \((f_x,\theta)\) such that \(\pi_x\) is concentrated on the contact set of \((f_x,\theta)\).
In this case, \((f_x,\theta)\) maximizes the dual problem and \(\pi_x\) minimizes the primal problem, so that \(\mathsf{P}(m)=\mathsf{D}(m)\); see Corollary \ref{cor_feas}.
Then the copula associated with \(\pi_x\) solves the maximization problem \eqref{defmaxprob}.

As a consequence of Theorem \ref{thm:main}, we can determine the exact region of Spearman's rho and Spearman's footrule as follows.

\begin{corollary}[The exact region $\Omega_{\varrho,\phi}$]\label{cor:main}
    For \(\brho\) in \eqref{eq:boundary} and \(\underline{\varrho}(x) = \tfrac{2\sqrt3}{9}(1+2x)^{3/2}-1\), we have
    \begin{align}\label{eq_exactregion}
	    \Omega_{\varrho,\phi} = \bigl\{(\varrho,x) : x\in[-\tfrac12,1],\ \underline{\varrho}(x)  \le \varrho\le \brho(x)\bigr\}.
	\end{align}
\end{corollary}

Beyond completing the exact \((\varrho,\phi)\)-region, Theorem~\ref{thm:main} admits a direct probabilistic interpretation in terms of the first two moments of the absolute rank difference.
Before developing this interpretation, we briefly relate Theorem~\ref{thm:main} to the previously known bounds.

\begin{figure}[t!]
	\centering
	\includegraphics[width=0.49\textwidth]{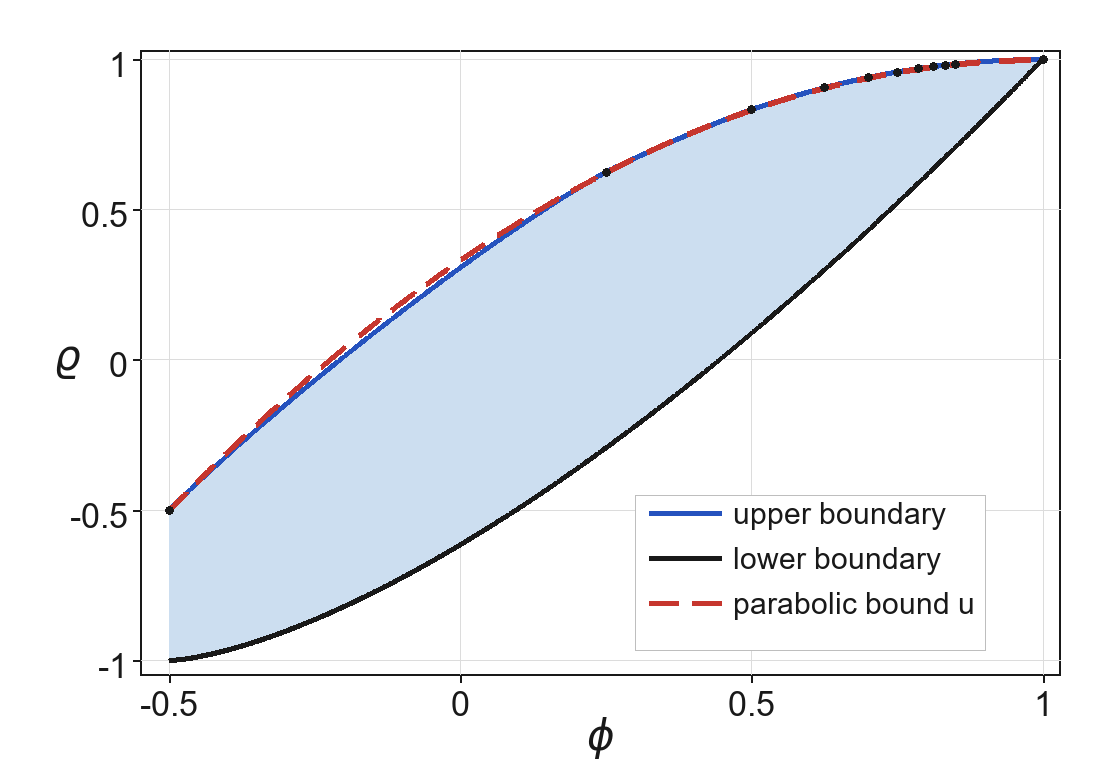}\hfill
	\includegraphics[width=0.49\textwidth]{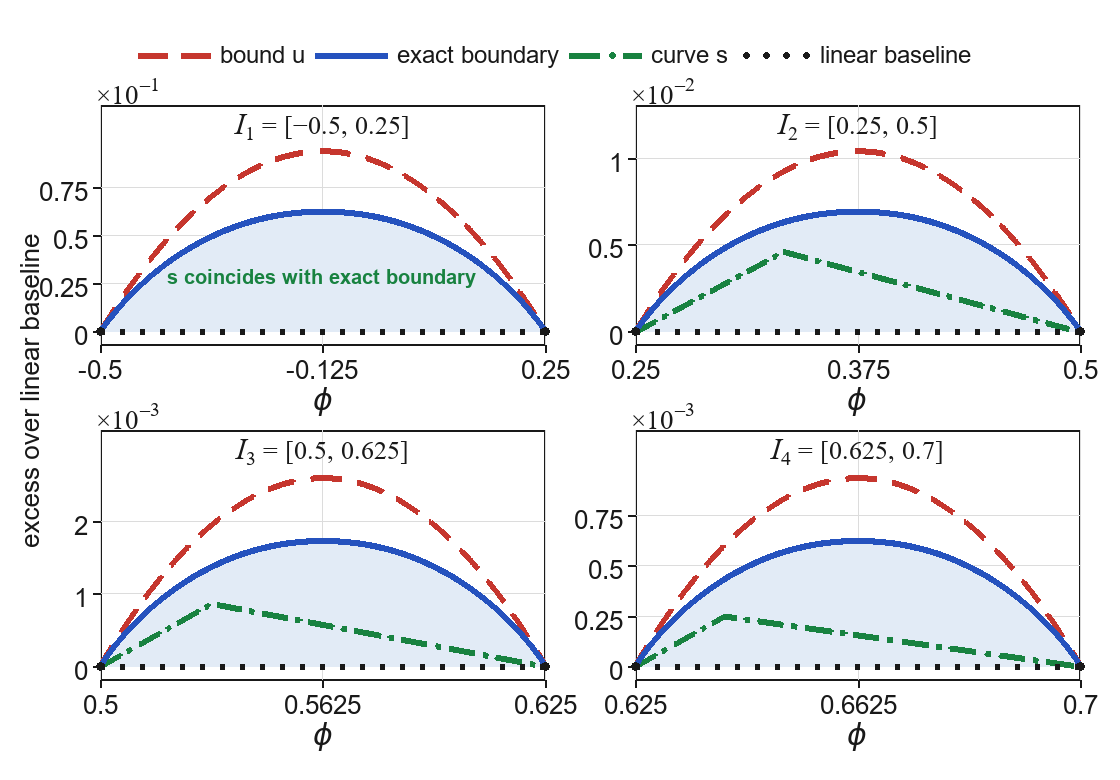}
	\caption{Left: the exact region \(\Omega_{\varrho,\phi}\) in \eqref{eq_exactregion}, bounded above by \(\brho\) from \eqref{eq:boundary} and below by the sharp boundary of \citet{kokolbukovsek2024exact}.
	The dashed curve is \(u\) from \eqref{eq:ksbound}, and the black dots mark its contact points with \(\brho\).
	Right: on the first four intervals \(I_1,\dots,I_4\) defined in \eqref{eq:intervals}, the vertical excesses of \(u\), \(\brho\), and the attainable curve \(s\) of \citet{tschimpke2025revisiting} are plotted relative to the piecewise-linear baseline obtained by joining consecutive black contact points in the left panel.
	The vertical scale is adapted to each interval.}
	\label{fig:region}
\end{figure}

\begin{remark}
    As motivated in the introduction, measures of association provide
information about the underlying dependence structure. In particular,
the extremal values \(\varrho(C)=\pm1\) determine the copula \(C\) as the upper/lower Fr\'{e}chet copula, respectively. The uniqueness statement in
Theorem~\ref{thm:main} shows that an analogous phenomenon occurs along
the upper boundary of the \((\varrho,\phi)\)-region: if
\(    \phi(C)=x\) and \(\varrho(C)=\overline{\varrho}(x),
\)
then we have
\(C=C_x.\)
Thus, every point on the upper boundary uniquely determines the
corresponding copula \(C_x\).
Conversely, if \((\varrho(C),\phi(C))\) is separated from the boundary
of \(\Omega_{\varrho,\phi}\), then \(C\) cannot be close to either
\(C_x\) or \(B_x\), where \(x=\phi(C)\), with respect to any metric
of weak convergence.
\end{remark}

\begin{remark}[Relation to earlier results]\label{rem:previous-curves}
As mentioned above, \citet{kokolbukovsek2024exact} determined the sharp lower boundary \(\underline{\varrho}\) exactly and proved the universal upper bound $u$ that coincides, due to \eqref{eq:boundary}, with \(\overline{\varrho}\) at the points $\{x_N\}_{N\geq 1}$ and at $1$.
They also constructed an attainable curve \(r\), which coincides on the interval \([-\tfrac 1 2, -\tfrac 1 8]\) with \(\overline{\varrho}\) and the function \(s\) below.
\citet{tschimpke2025revisiting} improved the attainable curve \(r\) by the function
\begin{align}\label{curve_s}
    s(x)=
	\begin{cases}
		2x+\tfrac12-\tfrac{\sqrt3}{9}(1+2x)^{3/2},
			&x\in[-\tfrac12,-\tfrac18),\\[2mm]
		x+\tfrac38-\tfrac{\sqrt6}{36}(1-4x)^{3/2},
			&x\in[-\tfrac18,\tfrac14),\\[2mm]
		u(x_N)+\dfrac{w_N-u(x_N)}{z_N-x_N}(x-x_N),
			&x\in[x_N,z_N),\quad N\ge2,\\[3mm]
		w_N+\dfrac{u(x_{N+1})-w_N}{x_{N+1}-z_N}(x-z_N),
			&x\in[z_N,x_{N+1}),\quad N\ge2,\\[3mm]
		1,	&x=1.
	\end{cases}
\end{align}
where \(z_N\coloneqq (2N^2+N-4)/(2(N+1)^2)\) and \(w_N\coloneqq (2N^5+6N^4+3N^3-7N^2-3N+1)/(2N^2(N+1)^3).\) They conjecture that \(s\) is not optimal at  any \(x\in(-\tfrac18,1)\setminus\{x_N:N\ge2\};\) see \cite[Paragraph after Thm.~5.3]{tschimpke2025revisiting}.
Theorem~\ref{thm:main} gives the precise resolution:
\[
	\brho(x)=s(x)
	\quad\text{for }x\in[-\tfrac12,\tfrac14]
		\cup\{x_N:N\ge2\}\cup\{1\},
	\qquad
	\brho(x)>s(x)
	\quad\text{for }x\in\bigcup_{N\ge2}(x_N,x_{N+1}).
\]
Consequently, the function \(s\) is also optimal on $(-\tfrac18,\tfrac14)$.
\end{remark}

\subsection{Discussion of Theorem \ref{thm:main}}\label{sec:int_main_thm}

For a discussion of the \(\brho\)-formula in \eqref{eq:boundary}, let us first recall that the universal upper bound \(u\) is a direct consequence of the Cauchy--Schwarz inequality.
Therefore, define the absolute difference
\begin{align}\label{def_Z}
    Z:= |U - V|,\quad \text{for } (U,V)\sim C.
\end{align}
Then, using the representation of \(\varrho\) and \(\phi\) in \eqref{eq_rep_phi_rho} and applying the Cauchy--Schwarz inequality, gives
\begin{align}\label{eq_CS}
   \frac{1-\varrho(C)}{6}
   = \E[Z^2]
   \geq \left(\E[Z]\right)^2
   = \left(\frac{1-\phi(C)}{3}\right)^2.
\end{align}
Solving the above inequality for \(\varrho(C)\) yields \(\varrho(C) \leq 1 - \frac 2 3 (1-\phi(C))^2\).
This implies, setting \(x = \phi(C)\), the expression
\begin{align}
    u(x) = 1 - \frac 2 3 (1-x)^2.
\end{align}
\begin{remark}
    The elementary Cauchy--Schwarz argument above, applied directly to the population versions of \(\varrho\) and \(\phi\), provides a considerably shorter proof of the upper bound \(u\) established in \cite[Theorem 11]{kokolbukovsek2024exact} and \cite[Theorem 3.4]{tschimpke2025revisiting}.
\end{remark}

The precise sense in which Theorem~\ref{thm:main} improves Cauchy--Schwarz is worth isolating.
For an arbitrary square-integrable nonnegative random variable \(\widetilde Z\) with mean \(m=\E\widetilde Z\), Cauchy--Schwarz inequality (equivalently, nonnegativity of the variance) gives only
\[
    \E\widetilde Z^2\geq m^2.
\]
This inequality cannot be improved from the value of \(m\) and the boundedness condition \(0\leq\widetilde Z\leq1\) alone, since the constant random variable \(\widetilde Z=m\) attains equality.
Our variable \(Z\) from \eqref{def_Z} additionally satisfies \(U,V\sim\cU(0,1)\).
Thus \(Z\) is not an arbitrary bounded random variable: it is the absolute difference induced by a coupling with two prescribed uniform marginals.
No independence, symmetry, or particular form of dependence between \(U\) and \(V\) is assumed.
It is precisely this uniform-marginal, or bistochastic, constraint that yields the improvement.
The following result gives the sharp strengthening of Cauchy--Schwarz.

\begin{proposition}[Improvement of Cauchy--Schwarz inequality under uniform marginals]
\label{prop:sharp-CS}~\\
For \(U,V\sim\cU(0,1)\), set \(m\coloneqq\E|U-V|\).
Then the following hold true:
\begin{enumerate}[label = (\roman*)]
    \item For \(m>0\), choose \(N\geq1\) such that
    \(
        m\in \left(\frac1{2N+2},\frac1{2N}\right].
    \)
    Let \(\ell\in\{\frac1{2N+2},\frac1{2N}\}\) be an endpoint nearest to \(m\), put \(\Delta_m:=|m-\ell|\), and define
    \begin{align}\label{def_Vmin}
        V_{\operatorname{min}}(m)
        \coloneqq
        \frac{2\Delta_m^{3/2}}{3\sqrt{N(N+1)}}-\Delta_m^2, \qquad \text{and} \quad V_{\operatorname{min}}(0):= 0.
    \end{align}
    Then
    \begin{equation}\label{eq:interpretation-sharp-cs}
        \E (U-V)^2
        \geq
        m^2+V_{\operatorname{min}}(m).
    \end{equation}

    \item The bound in \eqref{eq:interpretation-sharp-cs} is sharp for every \(m\).
    For \(m>0\), it is attained by the copula \(C_x\) from Definition~\ref{def_Cx} with \(x=1-3m\), and for \(m=0\) by the comonotonicity copula \(M(u,v)=\min\{u,v\}\).

    \item The correction term satisfies \(V_{\operatorname{min}}(m)\geq 0\) and
    \begin{align}\label{def_centersetC}
        V_{\operatorname{min}}(m)=0
        \quad\Longleftrightarrow\quad
        m\in
        \mathfrak C
        \coloneqq
        \{0\}\cup
        \left\{\frac{1}{2N}\colon N\in\N\right\}.
    \end{align}
\end{enumerate}
\end{proposition}
As a consequence of the above result, the ordinary Cauchy--Schwarz inequality
\begin{align}\label{eq:CSimp}
    \E(U-V)^2\geq \bigl(\E|U-V|\bigr)^2
\end{align}
is strict whenever \(m=\E|U-V|\notin\mathfrak C\).
Equality in \eqref{eq:CSimp} is equivalent to the constant-displacement condition
\begin{align}\label{eq:Z_const}
    |U-V|=m \quad \text{almost surely}.
\end{align}
For an illustration of the nontrivial equality cases in \eqref{eq:CSimp}, we refer to Figure \ref{fig:copulas}, where \(m=\tfrac14\) and \(m=\tfrac16\) occur as endpoint cases of the copula family constructed in Section~\ref{sec:attain}.
In terms of the concordance coefficients \(\varrho\) and \(\phi\), the values in Proposition~\ref{prop:sharp-CS} are precisely the points at which the ordinary Cauchy--Schwarz bound \(u\) touches the sharp boundary \(\brho\).
We return to the constant-displacement condition \eqref{eq:Z_const} from the perspective of generalized mixability in Section~\ref{sec:mixability}.

Combining the upper and lower boundaries of the exact \((\varrho,\phi)\)-region, the following results determines the possible variances of \(|U-V|\) when its mean is prescribed but the dependence structure of \(U,V\sim \cU(0,1)\) is unspecified.
Recall that \(C_x\) in \eqref{def_C_x} is \(\varrho\)-maximal under the constraint \(\phi(C_x)=x\), whereas the Bertino copula \(B_x\) in \eqref{def:lower-bertino} is \(\varrho\)-minimal under the same constraint.

\begin{corollary}[Exact mean--variance region for \(|U-V|\)]
\label{cor:rank-gap-variance}
For \(U,V\sim\cU(0,1)\) and \(m=\E|U-V|\), let \(V_{\operatorname{min}}(m)\) be given by \eqref{def_Vmin} and define
\begin{align}\label{def_Vmax}
    V_{\operatorname{max}}(m)
    \coloneqq
    \frac{1-(1-2m)^{3/2}}{3}-m^2.
\end{align}
Then
\begin{align}\label{eq_cor:rank-gap-variance2}
    V_{\operatorname{min}}(m)
    \leq
    \Var(|U-V|)
    \leq
    V_{\operatorname{max}}(m).
\end{align}
The lower bound is attained for \((U,V)\sim C_x\), \(x=1-3m\), and the upper bound is attained for \((U,V)\sim B_x\).
\end{corollary}

\begin{remark}
\label{rem:mean-variance-consequences}
\begin{enumerate}[label = (\alph*)]
    \item Every value between the two bounds in \eqref{eq_cor:rank-gap-variance2} is attainable.
    Indeed, both \(C_x\) and \(B_x\) satisfy \(\phi(C_x) = \phi(B_x)=x=1-3m\), and hence have the same mean absolute difference \(m\).
    Since both \(\phi\) and \(\varrho\) are affine in the copula, convex mixtures of \(C_x\) and \(B_x\) attain every intermediate variance.
    For \(m=0\), both bounds are zero, that is \(V_{\operatorname{min}}(0)=V_{\operatorname{max}}(0)=0\).

    \item Corollary~\ref{cor:rank-gap-variance} provides an equivalent mean--variance interpretation of Theorem~\ref{thm:main}: for every prescribed mean \(m=\E|U-V|\), it determines exactly the smallest and largest possible variance of the absolute difference \(|U-V|\) among all couplings of two uniform random variables.
    In particular, the lower boundary \(V_{\operatorname{min}}(m)\) is precisely the sharp Cauchy--Schwarz correction from Proposition~\ref{prop:sharp-CS}.

    \item If the mean \(m\) is not prescribed, optimizing the two bounds in \eqref{eq_cor:rank-gap-variance2} over \(m\in[0,\tfrac12]\) yields the universal variance bounds
    \begin{align}\label{eq_cor:rank-gap-variance}
        0
        \leq
        \Var(|U-V|)
        \leq
        \frac{5(3-\sqrt5)}{24}.
    \end{align}
    The lower bound is attained whenever \(|U-V|\) is vanishes almost surely, i.e., \(U=V\) almost surely.
    The upper bound is attained for \((U,V)\sim B_{x_0}\), where \(x_0\coloneqq\frac{7-3\sqrt5}{4}\).
\end{enumerate}
\end{remark}

\subsection{Organization of the paper}

The rest of the paper is organized as follows.
Section~\ref{sec:appl_main_thm} develops consequences and applications of
Theorem~\ref{thm:main} to finite rankings, generalized mixability, and the
relation between Chatterjee's rank correlation and the copula correlation ratio.
Section~\ref{sec:prelim} collects notation and reviews the optimal transport
concepts used throughout the paper, in particular Kantorovich duality with
linear constraints.
Section~\ref{sec:attain} constructs the candidate copula \(C_x\) in
\eqref{def_C_x}, verifies the moment constraint of the associated coupling,
and calculates its rank correlations in Proposition~\ref{prop_C_x}.
Section~\ref{sec:dual} introduces the dual potential in
Definition~\ref{def:dual-potential}, proves equality on the support of the
candidate coupling in Lemma~\ref{lem:equality-on-support}, establishes global
dual feasibility in Lemma~\ref{lem:global-dual-feasibility}, and proves
optimality and uniqueness in Proposition~\ref{prop:rho-phi-Cx}.
Finally, the proofs of the results stated in Sections~\ref{sec:intro} and
\ref{sec:appl_main_thm} are collected in Sections~\ref{sec:proof1} and
\ref{sec:proof2}, respectively.

\section{Consequences and applications}\label{sec:appl_main_thm}

The closed-form expression of \(\brho\) in Theorem~\ref{thm:main} leads to several consequences beyond the comparison of Spearman's rho and Spearman's footrule.
We first translate the exact mean--variance region in \eqref{eq_cor:rank-gap-variance2} into inequalities relating absolute and quadratic errors of finite rankings.
We then interpret the equality cases of the sharp Cauchy--Schwarz inequality under uniform marginals in terms of generalized mixability and quantify the minimal dispersion when exact mixability is impossible.
Finally, we apply the exact \((\varrho,\phi)\)-region to the comparison of Chatterjee's rank correlation with the copula correlation ratio, two recently studied measures of directed dependence.

\subsection{Inequalities for finite rankings}\label{sec_21}

We first translate the variance bounds in Corollary~\ref{cor:rank-gap-variance} into inequalities for two familiar distances between finite rankings.
Let \(\mathfrak S_n\) denote the set of permutations on \(\{1,\ldots,n\}\) and let \(\pi\in\mathfrak S_n\).
We compare the ranking \(\pi\) with the identity ranking.
Two popular unnormalized distances are
\begin{equation}\label{eq:finite-rank-distances}
    D_\pi\coloneqq\sum_{i=1}^n|i-\pi(i)|
    \qquad\text{and}\qquad
    S_\pi\coloneqq\sum_{i=1}^n(i-\pi(i))^2,
\end{equation}
which measure absolute and quadratic rank error, respectively; see, e.g., \cite{clemencon2013ranking,diaconis1977spearman}.
In particular, \(|i-\pi(i)|\) is the absolute \emph{rank displacement} of item \(i\).

To apply the copula bounds, we embed the finite permutation into a coupling with continuous uniform marginals.
Let \(I\) be uniformly distributed on \(\{1,\ldots,n\}\), let \(W\sim\cU(0,1)\) be independent of \(I\), and define
\begin{align}\label{eq:permutation-coupling}
    U_\pi
    \coloneqq
    \frac{I-1+W}{n},
    \qquad
    V_\pi
    \coloneqq
    \frac{\pi(I)-1+W}{n}.
\end{align}
Then \(U_\pi,V_\pi\sim\cU(0,1)\).
Indeed, conditionally on \(I=i\), the variable \(U_\pi\) is uniform on the \(i\)-th interval of the regular partition of \([0,1]\), whereas \(V_\pi\) is uniform on the \(\pi(i)\)-th interval.
Since \(\pi\) is a permutation, averaging over \(I\) yields uniform marginals.
Moreover,
\begin{align*}
    U_\pi-V_\pi
    =
    \frac{I-\pi(I)}{n},
\end{align*}
so the coupling preserves the normalized rank displacement exactly.
The corresponding copula is a shuffle of \(M\).

Consequently, the mean absolute and mean squared normalized rank displacements are
\begin{equation}\label{eq:permutation-costs}
    \begin{aligned}
        m_\pi
        &\coloneqq
        \E|U_\pi-V_\pi|
        =
        \frac{D_\pi}{n^2}, \qquad 
        q_\pi
        \coloneqq
        \E(U_\pi-V_\pi)^2
        =
        \frac{S_\pi}{n^3}.
    \end{aligned}
\end{equation}
Here one factor \(1/n\) comes from averaging over the \(n\) items, while normalizing the rank displacement by \(n\) contributes one further factor to \(m_\pi\) and two further factors to \(q_\pi\).
The next result is a discretized version of Corollary~\ref{cor:rank-gap-variance}.

\begin{theorem}[Finite-permutation inequalities]
\label{cor:permutation-costs}
Every \(\pi\in\mathfrak S_n\) satisfies
\begin{equation}\label{eq:permutation-envelope}
    m_\pi^2+V_{\operatorname{min}}(m_\pi)
    \leq
    q_\pi
    \leq
    m_\pi^2+V_{\operatorname{max}}(m_\pi)
    =
    \frac{1-(1-2m_\pi)^{3/2}}{3},
\end{equation}
where \(V_{\operatorname{min}}\) and \(V_{\operatorname{max}}\) are given in \eqref{def_Vmin} and \eqref{def_Vmax}, respectively.
\end{theorem}

\begin{remark}
    The lower bound in \eqref{eq:permutation-envelope} improves the Cauchy--Schwarz inequality \(S_\pi\geq D_\pi^2/n\) to
    \begin{equation}\label{eq:improved-cauchy-schwarz}
        S_\pi
        \geq
        \frac{D_\pi^2}{n}
        +
        n^3V_{\operatorname{min}}(m_\pi).
    \end{equation}
    The correction is strictly positive whenever \(m_\pi\notin\mathfrak C\).
    For \(m_\pi=1/(2N)\), equality in the ordinary Cauchy--Schwarz inequality is attainable by a permutation of size \(n\) if and only if \(n/(2N)\) is an integer.
    Equivalently, there exists a permutation \(\pi\in\mathfrak S_n\) satisfying
\[
    |i-\pi(i)|=\frac{n}{2N},
    \qquad i=1,\ldots,n.
\]
Consequently, the bounds in \eqref{eq:permutation-envelope} are sharp for the continuum problem of copulas and asymptotically sharp for finite rankings, but need not be optimal for each fixed \(n\).
\end{remark}

\subsection{Generalized mixability and minimal dispersion}\label{sec:mixability}

Classical complete mixability asks whether random variables with prescribed marginal distributions can be coupled such that their sum is constant almost surely; see, e.g., \cite{wang2015current}.
For integrable marginal distributions, the corresponding center is necessarily unique, since
\begin{align}\label{eq_mixability}
    X_1+\cdots+X_n=c \quad \text{almost surely}
\end{align}
implies \(c=\sum_{i=1}^n \E X_i.\) Bignozzi and Puccetti \cite{bignozzi2015studying} introduced the more general notions of \(\Psi\)-complete and \(\Psi\)-joint mixability, replacing the sum in \eqref{eq_mixability} by a measurable aggregation function \(\Psi\colon\R^n\to\R\).
In particular, a distribution \(F\) is \(\Psi\)-completely mixable with index \(n\) and center \(c\) if there exist \(X_1,\ldots,X_n\sim F\) such that
\begin{align*}
    \Psi(X_1,\ldots,X_n)=c
    \quad\text{almost surely}.
\end{align*}
Unlike for the classical sum, the center of a nonlinear aggregation function need not be determined by the marginal means and, in particular, several centers may be possible.

A particularly important class considered in \cite{bignozzi2015studying} is given by supermodular aggregation functions.
Recall that a function \(\Psi\colon\R^n\to\R\) is called \emph{supermodular} if
\begin{align*}
    \Psi(x)+\Psi(y)
    \leq
    \Psi(x\wedge y)+\Psi(x\vee y),
    \qquad x,y\in\R^n,
\end{align*}
where \(x\wedge y\) and \(x\vee y\) denote the componentwise minimum and maximum, respectively.
For twice continuously differentiable functions, supermodularity is equivalent to \(\partial_{ij}\Psi\geq0\) for all \(i\neq j\).

Our main result, Theorem \ref{thm:main}, admits a natural interpretation in this framework.
To make the connection to supermodular aggregation functions explicit, let \(U':=U-\frac12\) and \(V':=\frac12-V.\) Then \(U',V'\) are uniform on \((-1/2,1/2)\) and
\begin{align}\label{eq:centered-distance}
    |U-V|=|U'+V'|.
\end{align}
Consequently, for the supermodular function \(\Psi(u,v):=|u+v|\), the constant-distance condition \(|U-V|=m\) almost surely is equivalent to \(\Psi(U',V')=m\) almost surely.
Thus the constant-distance problem can be viewed as a \(\Psi\)-complete mixability problem for the centered uniform distribution.

The equality cases of the Cauchy--Schwarz bound that are attainable under the uniform-marginal constraint admit the following reformulation in terms of generalized complete mixability.
Recall the set \(\mathfrak{C}\) in \eqref{def_centersetC}.

\begin{lemma}[\(\Psi\)-complete mixability of the centered uniform distribution]
\label{lem:psi-mixability-centers}
Let \(\Psi(u,v)=|u+v|\).
There exist \(U',V'\sim\cU(-\tfrac12,\tfrac12)\) such that \(\Psi(U',V')=m\) almost surely if and only if \(m\in \mathfrak C\).
\end{lemma}

\noindent
Hence, the distribution \(\cU(-\tfrac12,\tfrac12)\) is \(\Psi\)-completely mixable with index \(2\) precisely for the countable set of centers \(\mathfrak C\).

Indeed, generalized mixability asks whether the aggregate \(\Psi(U',V')\) can be made constant.
If exact constancy is impossible, it is natural to ask how close one can get to it while prescribing its mean.
We therefore define the minimal variance
\begin{align}\label{eq:def-minimal-mixability-dispersion}
    \mathsf V(m)
    :=
    \min\left\{
        \Var\bigl(\Psi(U',V')\bigr):
        U',V'\sim\cU\left(-\frac12,\frac12\right),\
        \E\Psi(U',V')=m
    \right\}, \qquad m\in \left[0,\frac 1 2\right].
\end{align}
Exact \(\Psi\)-complete mixability with center \(m\) is equivalent to \(\mathsf V(m)=0\).
If \(\mathsf V(m)>0\), then exact mixability with center \(m\) is impossible, and \(\mathsf V(m)\) quantifies the smallest quadratic dispersion that an aggregate with prescribed mean \(m\) can attain.
In this sense, \(\mathsf V(m)\) provides a quantitative relaxation of exact \(\Psi\)-complete mixability.

As a consequence of Proposition~\ref{prop:sharp-CS}, we determine the minimal dispersion exactly as follows.

\begin{theorem}[Minimal dispersion and generalized mixability]
\label{cor:distance-mixability}
For \(\Psi(u,v)=|u+v|\), we have
\begin{align*}
    \mathsf V(m)=V_{\operatorname{min}}(m),
\end{align*}
where \(V_{\operatorname{min}}(m)\) is given by \eqref{def_Vmin}.
In particular, \(\mathsf V(m)=0\) if and only if \(m\in\mathfrak C\).
\end{theorem}

\begin{remark}
\begin{enumerate}[label = (\alph*)]
    \item Supermodularity provides a natural link between generalized mixability and dependence optimization.
    Among all couplings with prescribed marginals, for expectations of supermodular functions, the comonotonic coupling is maximal, while  the countermonotonic coupling is minimal; see, e.g., \cite[Theorems~3.9.8 and~3.9.15]{muller2002comparison}.
    Extremal expectations of supermodular functions can also be approximated numerically by rearrangement methods \cite{puccetti2015computation}.
    In our setting, \(\Psi(u,v)=|u+v|\) is supermodular, and the countermonotonic coupling \(V'=-U'\) yields the center \(0\).
    Unlike for classical complete mixability, however, centers of generalized \(\Psi\)-mixability need not be unique; see, e.g., \cite[Example~11]{bignozzi2015studying} for an example with two distinct centers.
    Lemma~\ref{lem:psi-mixability-centers} shows an even richer phenomenon: for the centered uniform distribution, the set of centers is the countably infinite set \(\mathfrak C\).

    \item Theorem~\ref{cor:distance-mixability} refines the binary question of exact \(\Psi\)-complete mixability into a quantitative one.
    The function \(m\mapsto\mathsf V(m)\) gives the minimal variance of the aggregate \(\Psi(U',V')\) under the constraint \(\E\Psi(U',V')=m\).
    Its zeros are precisely the exact mixability centers \(m\in\mathfrak C\), whereas \(\mathsf V(m)>0\) otherwise.
    Thus Proposition~\ref{prop:sharp-CS} determines how closely exact \(\Psi\)-complete mixability can be approximated at every prescribed mean aggregate.
\end{enumerate}
\end{remark}

\subsection{Chatterjee's rank correlation and the copula correlation ratio}

As an application of Theorem \ref{thm:main} and the exact \((\varrho,\phi)\)-region in Corollary \ref{cor:main}, we now determine bounds for two measures of \emph{directed} dependence recently studied in the statistics literature.
To be precise, let us first recall that classical measures of association such as Pearson correlation, Kendall's tau or Spearman's rho and footrule quantify the degree of positive or negative (linear) dependence between two random variables \(X\) and \(Y\).
However, they fail to detect non-linear and non-monotone relationships, respectively.
For example, all these quantities vanish for \(X\) standard normal and \(Y = X^2\).

Motivated by the seminal papers \cite{dette2013copula,chatterjee2021new,azadkia2021simple}, in the last decade many works have focused on dependence measures \(\kappa\) satisfying the following axioms:
\begin{enumerate}[label = (\Roman*)]
    \item \label{axiom1} \(\kappa(X,Y) \in [0,1]\),
    \item \label{axiom2} \(\kappa(X,Y) = 0\) if and only if \(X\) and \(Y\) are independent,
    \item \label{axiom3} \(\kappa(X,Y) = 1\) if and only if \(Y\) perfectly depends on \(X\), i.e., there exists a measurable function \(f\) such that \(Y = f(X)\) almost surely.
\end{enumerate}
Note that the functional relation in \ref{axiom3} is not assumed to be increasing or decreasing.
Hence, in contrast to the classical measures of association, suitable dependence measures \(\kappa\) can detect complicated and complex functional dependencies.

The certainly most prominent such dependence measure is Chatterjee's rank correlation \(\xi\) \cite{chatterjee2021new} whose population version is given by
\begin{align}\label{def_chattxi}
    \xi(X,Y) = \frac{\int \Var\bigl(P(Y\geq y \mid X)\bigr) \de P^Y(y)}{\int \Var(1_{\{Y\geq y\}}) \de P^Y(y)};
\end{align}
see \cite{chatterjee2024survey} for a recent survey and \cite{ansari2025direct,ansari2026quantifying,gamboa2022global,huang2022kernel,strothmann2024rearranged,wiesel2022measuring} for several extensions and related constructions.
Small/Large variability of the conditional survival probability in the numerator of \eqref{def_chattxi} indicates low/strong dependence of \(Y\) on \(X\).
The extreme cases of independence and perfect dependence are obtained, where \(P(Y\geq y\mid X)\) is constant or coincides with the indicator function in the denominator for all \(y\).
To better understand the behavior of \(\xi\) and to interpret its values, it is important to compare it with related measures of association.
For instance, attainable sets and inequalities with respect to Spearman's rho, Spearman's footrule, and Kendall's tau have been studied in \cite{ansari2026exact,rockel2026exact,rockel2026kendall}.

In the rest of this section, we focus on comparing \(\xi\) with the (rank-transformed) fraction of explained variance.
This comparison is particularly natural, since \(\xi\) measures the strength of functional dependence of \(Y\) on \(X\), while the fraction of explained variance quantifies how much of the variance in \(Y\) can be accounted for by \(X\).
To explain the details, let us assume for simplicity that \(X\) and \(Y\) have a continuous distribution function.
Then \(\xi\) depends only on the copula \(C \) of \((X,Y)\), and it reduces to the \emph{Dette-Siburg-Stoimenov measure}
\begin{align}\label{eq_xi_DSS}
    \xi(C) = 6 \int_0^1\int_0^1 (\partial_1 C(u,v))^2 \de u \de v - 2;
\end{align}
see \cite{dette2013copula}.
Here, \(\partial_1 C\) denotes the partial derivative of \(C\) with respect to the first component, which exists outside a Lebesgue null set \cite{nelsen2006introduction}.
Interestingly, the functional \(\xi\) in \eqref{eq_xi_DSS} admits a representation via Spearman's footrule through
\begin{align}\label{eq_xi_phi}
    \xi(C) = \phi(C\ast C);
\end{align}
see \cite{fuchs2024quantifying}.
Here, \(D\ast E\) denotes (a version of) the Markov product of two bivariate copulas \(D,E\in \CC\) defined by
\begin{align}\label{def_MK_product}
     D\ast E\,(u,v):=  \int_0^1 \partial_1 D(t,u) \partial_1 E(t,v) \de t, \quad (u,v)\in [0,1]^2.
\end{align}
It is well known that the Markov product is a copula that models conditional independence \cite[Chapter~5]{durante2016principles}.
In particular, if \(Y'\) is a conditionally independent copy of \(Y\) given \(X\), then we have
\begin{align}\label{eq_uus}
    (U,U')\sim C\ast C, \qquad \text{where}\quad U:= F_Y(Y) \quad \text{and} \quad U':=F_Y(Y').
\end{align}
Recall that \(\phi(C) = 6 \int_0^1 C(t,t) \de t - 2\).
Hence, by \eqref{eq_xi_phi}, Chatterjee's rank correlation evaluates the Markov product only on its diagonal.
While Chatterjee's rank correlation naturally complements Pearson's correlation and Spearman's and Kendall's rank correlation, several questions concerning the interpretation of its population value are open.
Since \(\xi\) measures the strength of functional dependence of \(Y\) on \(X\), a natural question is how much of the variance of the (rank-transformed) response \(Y\) can be explained by \(X\) when the value \(\xi(Y,X)\) is given?

We therefore consider the copula correlation ratio \cite{shih2021copula,sungur2005note} defined by
\begin{align}\label{def_cop_cor_ratio}
    \eta(C):=\eta(X,Y) := \frac{\Var\bigl(\E[F_Y(Y)\mid X]\bigr)}{\Var(F_Y(Y))}.
\end{align}
It quantifies the strength of regression dependence and coincides with the fraction of explained variance or the first-order Sobol' index of the rank-transformed response \(U=F_Y(Y)\) with respect to \(X\); see \cite{gamboa2022global}.
The copula correlation ratio satisfies Axioms \ref{axiom1} and \ref{axiom3}, but does not characterize independence.
Instead it satisfies
\begin{enumerate}[label = (\Roman*'), start = 2]
    \item \(\eta(X,Y) = 0\) if and only if \(\Var(\E[F_Y(Y)| X]) = 0\);
\end{enumerate}
see \cite[Theorem 2.2]{ansari2026quantifying}.
The latter is, in particular, fulfilled if \(X\) and \(Y\) are independent (but not vice versa).

While Chatterjee's rank correlation can be expressed as Spearman's footrule of the Markov product (see \eqref{eq_xi_phi}), the copula correlation ratio admits a representation through Spearman's rho of the Markov product via
\begin{align}\label{eq_rep_ccorr}
    \eta(C) = \frac{\Var\bigl(\E[U\mid X]\bigr)}{\Var(U)} = \frac{\E [U U'] - \E U \E U'}{\sqrt{\Var(U)} \sqrt{\Var(U')}} = \varrho(U,U') = \varrho(C\ast C),
\end{align}
where \(U\) and \(U'\) are given as in \eqref{eq_uus}.

As a consequence of the \((\varrho,\phi)\)-bounds in Theorem \ref{thm:main}, we determine the following bounds for the copula correlation ratio via Chatterjee's rank correlation.
Vice versa, this gives bounds for \(\xi\) in terms of \(\eta\); see Figure \ref{fig:xi-rank-sobol-bounds} for an illustration.

\begin{theorem}[\(\xi\)-\(\eta\)-bounds]\label{cor:xi-rank-sobol}
Assume that \((X,Y)\) has a continuous distribution function.
Then we have
\begin{equation}\label{eq:xi-rank-correlation-ratio}
    \max\!\left\{
        0,\underline{\varrho}\bigl(\xi(X,Y)\bigr)
    \right\}
    \leq
    \eta(X,Y)
    \leq
    \min\!\left\{
        \overline{\varrho}\bigl(\xi(X,Y)\bigr),
        2\xi(X,Y)
    \right\},
\end{equation}
where \(\underline{\varrho}\) and \(\brho\) are the bounds for the exact \((\varrho,\phi)\)-region in Corollary \ref{cor:main}.
\end{theorem}

\begin{figure}[t]
    \centering
    \includegraphics[width=0.72\textwidth]{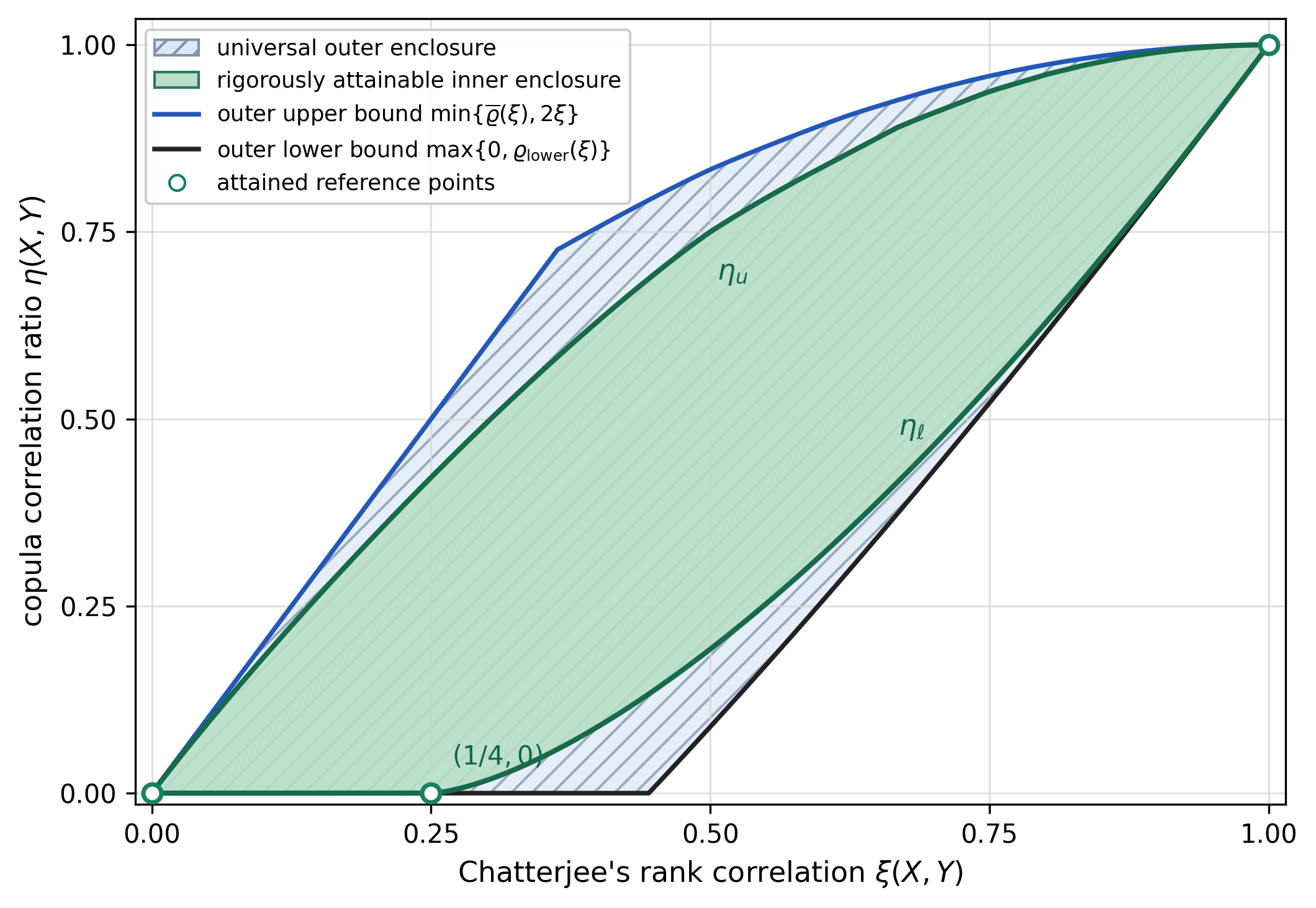}
    \caption{Outer and inner enclosures for \(\Omega_{\xi,\eta} =\left\{(\xi(C),\eta(C)), \, C\in \CC\right\}\).
    The hatched blue region is the universal outer enclosure from Theorem~\ref{cor:xi-rank-sobol}.
    The solid green region is the rigorously attainable inner enclosure from Proposition~\ref{prop:xi-rank-sobol-inner}; every point in it is realized by a conditional-i.i.d. model.
    The two unshaded strips between the inner and outer boundaries remain unresolved.}
    \label{fig:xi-rank-sobol-bounds}
\end{figure}

\begin{remark}
\begin{enumerate}[label = (\alph*)]
    \item The two quantities \(\xi\) and \(\eta\) capture different aspects of directed dependence.
    Chatterjee's \(\xi\) measures changes in the entire conditional distribution of the rank \(U = F_Y(Y)\), whereas \(\eta\) only measures changes in its conditional mean.
    Consequently, deviations of \(P(U\leq t\mid X)\) from \(t\) may cancel in the latter quantity as we show in Example \ref{ex_xi_eta}.
    \item Theorem~\ref{cor:xi-rank-sobol} provides an upper bound for \(\eta\) only.
    Indeed, it applies Theorem~\ref{thm:main} to the Markov product \(C\ast C\), but not every copula is of this form.
    Every such Markov product is symmetric and admits the conditional-i.i.d.\ representation in \eqref{eq_uus}; in particular, this class is not dense in \(\CC\), because the symmetric copulas form a closed proper subclass of \(\CC\).
    Sharpness of the \((\varrho,\phi)\)-bounds over all copulas therefore does not imply sharpness under the conditional-i.i.d.\ restriction; see also Example \ref{ex:xi-eta-upper-not-sharp}.
\end{enumerate}
\end{remark}

\begin{example}[\(\xi(X,Y) = \frac 1 4\) while \(\eta(X,Y) =0\)]\label{ex_xi_eta}
For \(U\sim \cU(0,1)\), set \(Y=U\), and take \(X=|2U-1|\).
Then \(X\sim \cU(0,1)\) and, conditionally on \(X=r\), the variable \(U\) is equally likely to be \((1-r)/2\) or \((1+r)/2\).
Hence, \(\E[U\mid X]=1/2\), so that \(\eta(X,Y)=0\).
On the other hand, let \(U'\) be a conditionally independent copy of \(U\) given \(X\).
Then \(U\) and \(U'\) differ by \(r\) with probability \(1/2\).
This gives
\[
    \xi(X,Y)
    =
    1-3\,\E|U-U'| 
    = 1-3\int_0^1 \underbrace{\E[|U-U'|\mid X=r]}_{= r/2}\de r 
    = 1 - \frac 3 4 =
    \frac14.
\]
Thus a positive value of \(\xi\) need not imply a positive copula correlation ratio.
\end{example}

\begin{example}[Non-sharpness of the upper \(\xi\)--\(\eta\) bound]
\label{ex:xi-eta-upper-not-sharp}
The upper bound in \eqref{eq:xi-rank-correlation-ratio} is not sharp.
Indeed, recall the Cauchy--Schwarz estimate
\begin{align*}
    \eta(X,Y)\leq 2\xi(X,Y).
\end{align*}
With \(U=F_Y(Y)\), set \(g_X(t)\coloneqq P(U\leq t\mid X)-t\).
Equality in the above Cauchy--Schwarz inequality requires \(g_X\) to be constant almost everywhere on \([0,1]\), for almost every \(X\).
Since \(t\mapsto t+g_X(t)\) is a distribution function, this is possible only for \(g_X=0\), and hence only at \((\xi,\eta)=(0,0)\).

Further, for \(x = \frac 1 4 \) and \(N=2\), Theorem~\ref{thm:main} gives \(\brho(1/4)=u(1/4)=5/8\).
Thus \eqref{eq:xi-rank-correlation-ratio} yields the upper bound
\begin{align*}
    \eta\leq
    \min\left\{\brho\left(\frac14\right),2\cdot \frac14\right\}
    =\frac12
\end{align*}
at \(\xi=1/4\), but equality cannot be attained, as discussed before.
Moreover, the attainable set of conditional-i.i.d.\ pairs \((\xi,\eta)\) can be shown to be compact, so the maximal value of \(\eta\) at \(\xi=1/4\) is strictly smaller than \(1/2\).
Hence, the upper bound in \eqref{eq:xi-rank-correlation-ratio} is not sharp.
By contrast, Example~\ref{ex_xi_eta} shows that the lower value \(\eta=0\) is attainable at \(\xi=1/4\).
\end{example}

To complement the outer bounds in \eqref{eq:xi-rank-correlation-ratio}, we now determine a large set of attainable \((\xi,\eta)\)-pairs.
Following the above setting, we restrict attention to random vectors \((X,Y)\) with continuous marginal distribution functions.
We denote the corresponding attainable \((\xi,\eta)\)-region by
\begin{align}\label{def:xi-eta-region}
    \Omega_{\xi,\eta}
    :=
    \left\{
        \bigl(\xi(X,Y),\eta(X,Y)\bigr):
        F_X \text{ and } F_Y \text{ are continuous}
    \right\}.
\end{align}
Recall that \(\xi(X,Y)=\phi(C\ast C)\) and \(\eta(X,Y)=\varrho(C\ast C)\), where \(C\) is the copula of \((X,Y)\).
Thus every point in \(\Omega_{\xi,\eta}\) is generated by a copula of the form \(C\ast C\), or equivalently by a conditional-i.i.d.\ pair \((U,U')\) as in \eqref{eq_uus}.
As discussed above, Theorem~\ref{cor:xi-rank-sobol} applies the exact \((\varrho,\phi)\)-region to the larger class of all copulas and provides an outer enclosure of \(\Omega_{\xi,\eta}\) that is not sharp.

To the best of our knowledge, a sharp description of \(\Omega_{\xi,\eta}\) is not known.
We therefore complement the outer enclosure by constructing an explicit inner region.
The natural object for this purpose is the random conditional distribution
\begin{align*}
    \mu_X:=\operatorname{law}(U\mid X),
    \qquad
    U=F_Y(Y).
\end{align*}
Since \(U\sim\cU(0,1)\), the conditional distributions satisfy the barycenter condition \(\E\mu_X=\lam\), where \(\lam\) denotes Lebesgue measure on \([0,1]\).
This means that, for every Borel set \(A\subseteq[0,1]\),
\begin{align*}
    \E[\mu_X(A)]
    =
    P(U\in A)
    =
    \lam(A).
\end{align*} If \(U'\) is
drawn conditionally independently from the same distribution \(\mu_X\), then
\begin{equation}\label{eq:ciid-coordinates}
    \xi(X,Y)
    =
    1-3\E|U-U'|,
    \qquad
    \eta(X,Y)
    =
    12\Var\left(\int_0^1u\,\mu_X(\de u)\right).
\end{equation}
Hence, \(\xi\) reflects the average within-distribution spread of the conditional laws \(\mu_X\), whereas \(\eta\) depends only on the variation of their conditional means.
This distinction suggests two complementary constructions.

For the lower inner curve \(\eta_\ell\) in Figure~\ref{fig:xi-rank-sobol-bounds}, we exploit the fact that \(\eta\) only depends on the variability of the conditional mean.
To keep \(\eta\) small, we therefore keep \(\E[U\mid X]=1/2\) whenever possible by grouping ranks into symmetric pairs \(\{u,1-u\}\).
We then reveal \(U\) exactly on an increasingly large central interval, while retaining these symmetric two-point conditional distributions outside.
This interpolates between the model in Example~\ref{ex_xi_eta} and complete dependence and yields the lower inner curve
\begin{align}\label{def_eta_ell}
    \eta_\ell(x)
    :=
    \begin{cases}
        0, & 0\leq x\leq\frac14,\\[2mm]
        \left(\dfrac{4x-1}{3}\right)^{3/2},
            & \frac14\leq x\leq1.
    \end{cases}
\end{align}
We refer to Figure \ref{fig:eta-lower-construction} for a visualization of the above construction.
Further details are provided in the proof of Proposition \ref{prop:xi-rank-sobol-inner}.

For the upper inner curve \(\eta_u\) in Figure~\ref{fig:xi-rank-sobol-bounds}, we proceed in the opposite direction.
Here the aim is to make the conditional means vary strongly, and hence to make \(\eta\) large, while keeping a controlled amount of variability within the conditional distribution of \(U\) given \(X\).
For the construction of \(\eta_u\) in \eqref{def_eta_u}, we first split the uniform distribution into two conditional laws with different means whose equally weighted mixture recovers the required uniform marginal distribution.
We then place affine copies of this binary model into \(n\) equal subintervals of \([0,1]\).
The interval index determines the location of \(U\) up to an interval of length \(1/n\), while the parameter \(a\) controls the strength of the binary perturbation, and hence the within-cell variability of \(U\).
As \(n\) increases, this increasingly fine localization drives both \(\xi\) and \(\eta\) towards \(1\), whereas \(a\) determines the position of the resulting \((\xi,\eta)\)-pair along each branch. To be precise, denote for a set \(\cS\subseteq[0,1]^2\) by \(\operatorname{conv}(\cS)\) its convex hull.
For \(n\in\N\) and \(a\in[0,\frac12]\), put
\[
    x_n(a)
    :=
    1-\frac{1-2a^2(3-4a)}{n},
    \qquad
    y_n(a)
    :=
    1-\frac{1-12a^2(1-a)^2}{n^2}.
\]
The construction described above generates the set
\begin{align*}
    \mathcal S_{u}
    :=
    \left\{
        \bigl(x_n(a),y_n(a)\bigr):
        n\in\N,\ a\in[0,\tfrac12]
    \right\}
    \cup\{(1,1)\},
\end{align*}
whose upper convexified boundary is the function \(\eta_u\colon[0,1]\to[0,1]\) defined by
\begin{align}\label{def_eta_u}
    \eta_u(x)
    :=
    \max\left\{
        y:(x,y)\in\operatorname{conv}(\mathcal S_{u})
    \right\},
    \qquad 0\leq x\leq1.
\end{align}
For details of the construction, we refer to the proof of the following proposition, which establishes the corresponding inner enclosure.

\begin{figure}[t]
    \centering
    \includegraphics[width=\textwidth]{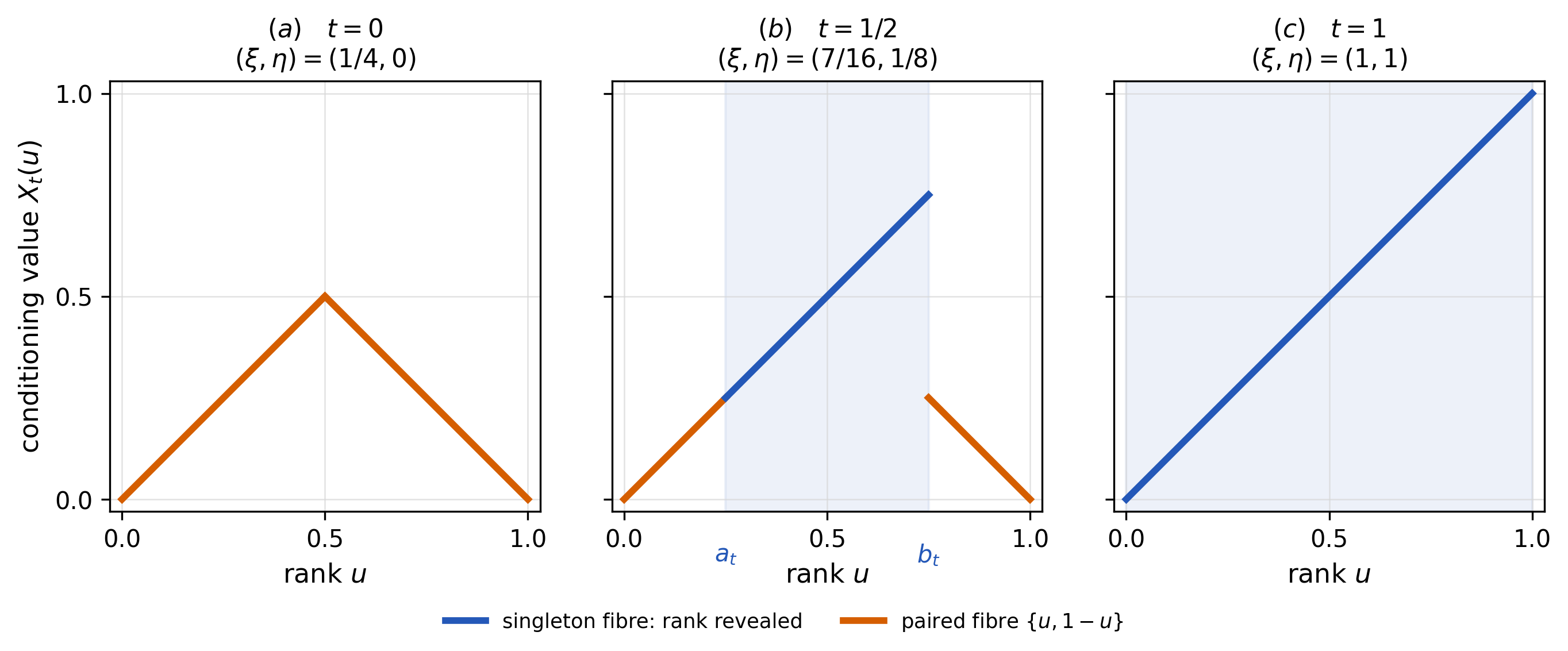}
    \caption{Construction of the lower inner curve \(\eta_\ell\) in \eqref{def_eta_ell}.
    The central interval \([a_t,b_t]\) is revealed exactly, so its conditional fibers are singletons (blue).
    Outside this interval, the symmetric ranks \(u\) and \(1-u\) form a common fiber (orange).
    As \(t\) increases from \(0\) to \(1\), the revealed interval expands and the construction moves from \((\xi,\eta)=(1/4,0)\) to complete dependence at \((1,1)\).}
    \label{fig:eta-lower-construction}
\end{figure}

\begin{figure}[t]
    \centering
    \includegraphics[width=\textwidth]{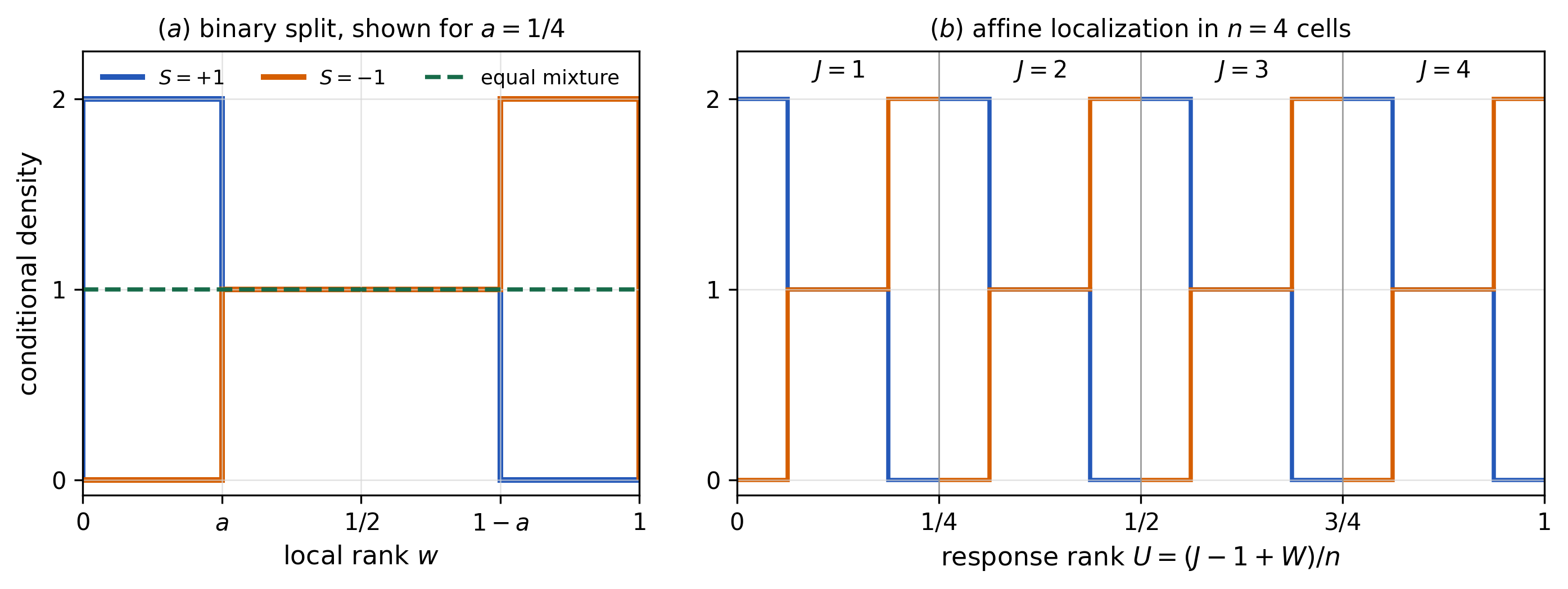}
    \caption{Construction underlying the upper inner curve \(\eta_u\) in \eqref{def_eta_u}.
    Panel~(a) shows, for \(a=1/4\), the two mirror-image conditional densities associated with the binary mixing variable \(S=+1\) and \(S=-1\); their equally weighted mixture is the uniform density; see Eq. \eqref{eq_mix_S}.
    Panel~(b) shows affine copies of this binary model in \(n=4\) equal cells.
    The variable \(J\) selects the cell and \(S\) selects the conditional law within it.
    }
    \label{fig:eta-upper-construction}
\end{figure}

\begin{proposition}[Constructive inner enclosure]
\label{prop:xi-rank-sobol-inner}
The set \(\mathcal S_{u}\) is compact, its projection onto the first coordinate is \([0,1]\), and \(\eta_u\) is a well-defined concave function.
Moreover,
\begin{equation}\label{eq:xi-rank-sobol-inner}
    \left\{
        (x,y):
        0\leq x\leq1,\quad
        \eta_\ell(x)
        \leq y\leq
        \eta_u(x)
    \right\}
    \subseteq
    \Omega_{\xi,\eta}.
\end{equation}
\end{proposition}

\begin{remark}
\begin{enumerate}[label = (\alph*)]
\item The functions \(\eta_\ell\) and \(\eta_u\) provide, respectively, constructive lower and upper inner bounds for the attainable region \(\Omega_{\xi,\eta}\).
Together with the outer bounds in \eqref{eq:xi-rank-correlation-ratio}, they yield explicit inner and outer enclosures of the exact \((\xi,\eta)\)-region; see Figure~\ref{fig:xi-rank-sobol-bounds}.
Numerical evaluation shows that the constructive inner enclosure occupies approximately \(81.4\%\) of the area of the outer enclosure.
Determining the exact region \(\Omega_{\xi,\eta}\), and in particular whether either \(\eta_\ell\) or \(\eta_u\) is sharp, remains open.
\item The lower construction illustrates how \(\xi\) can increase while \(\eta\) remains small.
On the interval \(0\leq x\leq1/4\), the lower bound satisfies \(\eta_\ell(x)=0\); these values are attained by mixing independence with the model in Example~\ref{ex_xi_eta}.
For \(x\in[1/4,1]\), progressively revealing \(U\) on a central interval increases both measures and yields the curve
    \[
        \eta_\ell(x)
        =
        \left(\frac{4x-1}{3}\right)^{3/2},
    \]
    connecting \((1/4,0)\) with complete dependence at \((1,1)\).
    \item The two parameters in the upper construction have distinct roles.
    The integer \(n\) controls the coarse localization of \(U\): the larger \(n\), the smaller the interval of length \(1/n\) in which \(U\) is known to lie, thereby driving both \(\xi\) and \(\eta\) towards \(1\).
    By contrast, the parameter \(a\) controls the conditional structure within each interval through the strength of the binary perturbation.
    Thus \(n\) governs the global localization of \(U\), whereas \(a\) determines the remaining conditional variability within each local cell.
\end{enumerate}
\end{remark}

\section{Copula and optimal-transport preliminaries}\label{sec:prelim}

We first collect the copula and measure-theoretic notation used throughout the proofs.
Let \(\N\coloneqq\{1,2,\ldots\}\), and let \(\lam\) and \(\lam_2\coloneqq\lam\otimes\lam\) denote Lebesgue measure on \([0,1]\) and \([0,1]^2\), respectively.
We write \(\cU(0,1)\) for the uniform distribution on \([0,1]\).
For a Borel set \(A\), the restriction of a measure \(\mu\) to \(A\) is denoted by \(\mu\vert_A\), and \(\cM(A)\) denotes the set of finite Borel measures on \(A\).
If \(T\colon A\to B\) is Borel measurable and \(\mu\in\cM(A)\), its pushforward under \(T\) is the measure \(T_\#\mu\) defined by
\[
    (T_\#\mu)(E)\coloneqq\mu\bigl(T^{-1}(E)\bigr),
    \qquad E\subseteq B\ \text{Borel}.
\]
We write \(\operatorname{proj}_1(u,v)\coloneqq u\) and \(\operatorname{proj}_2(u,v)\coloneqq v\), and denote the marginals of \(\mu\in\cM([0,1]^2)\) by \(\mu_i\coloneqq(\operatorname{proj}_i)_\#\mu\), \(i=1,2\).

A \emph{bivariate copula} is a bivariate distribution function \(C\) on \([0,1]^2\) with uniform margins, that is,
\begin{align}\label{eq_unif_marg}
    C(u,1)=u,
    \qquad
    C(1,v)=v,
    \qquad u,v\in[0,1].
\end{align}    
We denote the class of bivariate copulas by \(\CC\).
The concept of copula is motivated by Sklar's theorem which states that, for every bivariate distribution function \(H\) with marginal distribution functions \(F\) and \(G\), there exists \(C\in\CC\) such that
\begin{align}\label{the_Sklar}
    H(x,y)=C\bigl(F(x),G(y)\bigr),
    \qquad x,y\in\R.
\end{align}
Further, the copula \(C\) is unique on \(\operatorname{Ran}(F)\times\operatorname{Ran}(G)\), and hence unique whenever \(F\) and \(G\) are continuous.
Conversely, for any \(C\in\CC\) and any univariate distribution functions \(F,G\), the function \(H\) in \eqref{the_Sklar} is a bivariate distribution function with marginals \(F\) and \(G\); see \cite[Thm.~2.3.3]{nelsen2006introduction}.
In particular, if \(X\) and \(Y\) have continuous distribution functions, their copula coincides with the joint distribution function of the rank-transformed vector \((F_X(X),F_Y(Y))\).

For \(C\in\CC\), let \(\mu_C\) denote the Borel probability measure determined by
\[
    \mu_C\bigl([0,u]\times[0,v]\bigr)=C(u,v),
    \qquad u,v\in[0,1].
\]
By \eqref{eq_unif_marg}, both marginal distributions of \(\mu_C\) are given by the Lebesgue measure \(\lambda|_{[0,1]} = \cU(0,1)\); such measures are also called \emph{doubly stochastic}.
Thus \(C\mapsto\mu_C\) is a bijection from \(\CC\) onto the set
\[
    \Pi
    \coloneqq
    \bigl\{
        \pi\in\cM([0,1]^2):
        \pi_1=\pi_2=\lam
    \bigr\}
\]
of couplings of two uniform distributions.

The upper Fr\'{e}chet copula \(M(u,v)\coloneqq\min\{u,v\}\) models comonotonicity, i.e., perfect positive dependence.
A copula \(C\) is called a \emph{shuffle of min} if there is a finite interval partition of \([0,1]\) and a measure-preserving bijection \(T\colon[0,1]\to[0,1]\) that is affine with slope \(+1\) or \(-1\) on the interior of each partition interval such that
\[
    \mu_C=(\operatorname{id},T)_\#\lam.
\]
Thus the mass of a shuffle of min is concentrated on finitely many line segments of slope \(+1\) or \(-1\); see \cite{mikusinski1992shuffles}.

\begin{example}[Bertino copulas]
For a copula \(C\), the function \(\delta_C(t)\coloneqq C(t,t)\) is called its diagonal.
Given a copula diagonal \(\delta\), the associated \emph{Bertino copula} is
\begin{equation}\label{def:bertino}
    B_\delta(u,v)
    \coloneqq
    u\wedge v-
    \min_{u\wedge v\leq t\leq u\vee v}\bigl(t-\delta(t)\bigr);
\end{equation}
see \cite[Sec.~5]{fernandezsanchez2016members}.
The Bertino copula family \((B_x)_{x\in [- 1/2,1]}\) that describes the lower bound \(\underline{\varrho}\) due to \cite{kokolbukovsek2024exact} is obtained from the diagonals
\[
    \delta_a(t)\coloneqq
    \begin{cases}
        0, & 0\leq t\leq a,\\
        t-a, & a\leq t\leq1-a,\\
        2t-1, & 1-a\leq t\leq1,
    \end{cases} \qquad a\in[0,\tfrac12],
\]    
by setting
\begin{equation}\label{def:lower-bertino}
    B_x\coloneqq B_{\delta_{a_x}} \qquad \text{for} \quad a_x\coloneqq\frac12\left(1-\sqrt{\frac{1+2x}{3}}\right).
\end{equation}
This is the three-strip shuffle of min from \cite[Example~6]{kokolbukovsek2024exact}, which satisfies \(\phi(B_x)=x\) and \(\varrho(B_x)=\underline{\varrho}(x)\).
Figure~\ref{fig:bertino-lower} illustrates how the support of \(B_x\) changes along the lower boundary.
\end{example}

\begin{figure}[t]
    \centering
    \includegraphics[width=\textwidth]{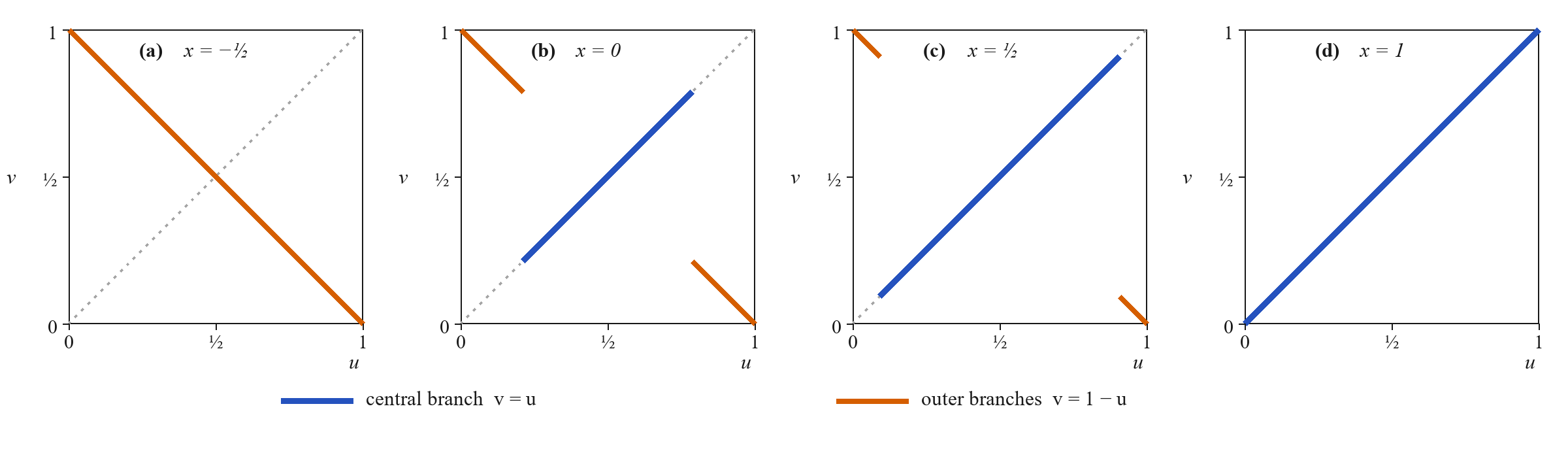}
    \caption{Supports of four lower-bound Bertino copulas \(B_x=B_{\delta_{a_x}}\) from \eqref{def:lower-bertino}.
    The solid blue segment is the central branch \(v=u\), the solid orange segments are the outer branches \(v=1-u\), and the dotted gray diagonal is shown for reference.
    As \(x\) increases, \(a_x\) decreases from \(\tfrac12\) to \(0\), deforming the countermonotone copula \(W(u,v)=\max\{u+v-1,0\}\) into the Min copula \(M\).
    Since \((\varrho(B_x),\phi(B_x))=(\underline{\varrho}(x),x)\), the family traces the lower boundary of \(\Omega_{\varrho,\phi}\).}
    \label{fig:bertino-lower}
\end{figure}

For \(C\in\CC\), let \(C^\top(u,v)\coloneqq C(v,u)\) be its transpose, and call \(C\) \emph{symmetric} if \(C=C^\top\).
With \(S(u,v)\coloneqq(v,u)\), this is equivalent to \(S_\#\mu_C=\mu_C\); accordingly, we denote the set of symmetric couplings by
\[
    \Pi_{\operatorname{sym}}
    \coloneqq
    \{\pi\in\Pi:S_\#\pi=\pi\}.
\]
Since \(\phi\) and \(\varrho\) in \eqref{eq:defs} are affine in \(C\) and invariant under transposition, symmetrization \(C\mapsto\tfrac12(C+C^\top)\) preserves both quantities.
Hence all extremal problems below admit a symmetric optimizer.
Throughout, \((U,V)\sim C\) means that \((U,V)\) has distribution \(\mu_C\).

\begin{lemma}[Moment representation]\label{lem:moments}
	For every $C\in\CC$ the identities \eqref{eq_rep_phi_rho} hold.
\end{lemma}

\begin{proof}
	By Fubini's theorem,
	\[
	\int_{[0,1]^2}C(u,v)\de\lam_2(u,v)
	=\int_{[0,1]^2}P(U\le u, V\le v)\de\lam_2(u,v)
	=\E[(1-U)(1-V)]
	=\E[UV],
	\]
	where the last equality uses the uniform margins.
	Moreover, for \(D:= U-V\), we obtain $2\E[UV]=\E U^2+\E V^2-\E D^2=\tfrac23-\E D^2$, and hence
	\(
	\varrho=12\E[UV]-3=1-6\E D^2.
	\)
	Similarly
	\(
	\int_0^1 C(t,t)\de t=\E[1-\max(U,V)]
	\)
	and
	\(
	\max(U,V)=\tfrac{U+V+|D|}{2},
	\)
	so
	\(
	\phi=6\bigl(1-\tfrac12-\tfrac12\E|D|\bigr)-2=1-3\E|D|.
	\)
\end{proof}

Every \(\pi\in\Pi\) satisfies \(\pi_1+\pi_2=2\lam\).
We next characterize \(\Pi_{\operatorname{sym}}\) through measures on the triangle
\[
    T\coloneqq\{(x,y)\in[0,1]^2:x\leq y\}.
\]
Define
\begin{align}\label{def:Pi_T}
    \Pi_T \coloneqq \{ \nu\in \cM(T)\mid \nu_1 + \nu_2 = \lam\}.
\end{align}
Every \(\nu\in\Pi_T\) has total mass \(\nu(T)=\tfrac12\), obtained by evaluating \(\nu_1+\nu_2=\lam\) on \([0,1]\).

\begin{lemma}[Representation of symmetric couplings]\label{lem:pair-measure-bijection}
    The mapping \(    \Phi:\Pi_T\to\Pi_{\mathrm{sym}}\), \(\Phi(\nu):=\nu+S_\#\nu, \) is a bijection.
\end{lemma}

\begin{proof}
Let \(\nu\in\Pi_T\) and set \(\pi:=\Phi(\nu)=\nu+S_\#\nu.\) Since \(S\circ S=\operatorname{id}\), we obtain
\[
    S_\#\pi
    =
    S_\#\nu+S_\#(S_\#\nu)
    =
    S_\#\nu+\nu
    =
    \pi,
\]
so that \(\pi\) is symmetric.
Moreover, \(\pi_1 = \nu_1 + (S_\#\nu)_1 = \nu_1 + \nu_2 = \lam\).
Similarly, \(\pi_2 = \lam\).
Hence \(\pi\in\Pi_{\mathrm{sym}}\), and thus \(\Phi\) is well defined.

To prove surjectivity, let \(\pi\in\Pi_{\mathrm{sym}}\), and denote by \(\mathsf{Diag}:=\{(x,x):x\in[0,1]\}\) the diagonal.
Define a measure \(\nu\) on \(T\) by
\[
    \nu
    :=
    \pi|_{\{x<y\}}
    +
    \frac12\,\pi|_{\mathsf{Diag}}.
\]
Since \(\pi\) is symmetric, we have \(\pi|_{\{x<y\}} = S_\#(\pi|_{\{x>y\}})\).
Moreover, \(S\) acts as the identity on \(\mathsf{Diag}\).
Therefore, \(\pi=\nu+S_\#\nu\).
Consequently,
\begin{align*}
    \lam
    =
    \pi_1 = \nu_1 + (S_\#\nu)_1 = \nu_1 + \nu_2, 
\end{align*}
which shows that \(\nu\in\Pi_T\).
Thus \(\Phi\) is surjective.

Finally, suppose that \(\nu,\widetilde\nu\in\Pi_T\) satisfy \(\nu+S_\#\nu = \widetilde\nu+S_\#\widetilde\nu. \) Restricting this identity to \(\{x<y\}\) gives \(\nu|_{\{x<y\}} = \widetilde\nu|_{\{x<y\}}, \) because \(S_\#\nu\) and \(S_\#\widetilde\nu\) are supported on \(\{x\ge y\}\).
Restricting it to \(\mathsf{Diag}\) yields
\(
    2\nu|_{\mathsf{Diag}}
    =
    2\widetilde\nu|_{\mathsf{Diag}}.
\)
Hence \(\nu=\widetilde\nu\), and therefore \(\Phi\) is injective.
\end{proof}

Using the representation of symmetric couplings in Lemma \ref{lem:pair-measure-bijection}, the \(\varrho\)-maximal value in dependence on \(\phi(C) = x\) can be written as an optimal transport problem via the measures \(\Pi_T\):

\begin{lemma}[Representation of \(\overline{\varrho}\)]\label{prop:pairlp}
	Let $x\in[-\tfrac12,1]$. For the \(\varrho\)-maximal value in \eqref{defmaxprob}, we have
    \begin{align}\label{eq:lp}
        \overline{\varrho}(x) = 1 - 12\,\min\Bigl\{
			\int_T(b-a)^2\de\nu(a,b)
			\;:\;
			\nu\in\Pi_T,\ \int_T(b-a)\de\nu(a,b) = \tfrac{m}{2}
		\Bigr\},
    \end{align}
    where \(\Pi_T\) is defined in \eqref{def:Pi_T} and \(m\coloneqq\tfrac{1-x}{3}\).
\end{lemma}

\begin{proof}
	Given a symmetric $C\in\CC$, let $\nu\coloneqq\tfrac12\,(\min,\max)_\#\mu_C$ on $T$.
	For integrable  $g:[0,1]\to\R$,
	\begin{align*}
		\int_T [g(a)+g(b)]\de\nu(a,b)
		=& \frac12\int_{[0,1]^2}[g(u\wedge v)+g(u\vee v)]\de\mu_C(u,v)\\
		=& \frac12\int_{[0,1]^2} [g(u)+g(v)]\de\mu_C(u,v)
		= \int_0^1 g(t)\de\lam(t),
	\end{align*}
	so $\nu\in\Pi_T$.
	In particular, taking $g\equiv1$ gives $2\nu(T)=1$.
	Moreover,
	\(
	\int_T h(b-a)\de\nu(a,b)=\tfrac12\E[h(|U-V|)]
	\)
    for \((U,V)\sim C\) and for every bounded Borel function $h:[0,\infty)\to\R$.
    Conversely, given $\nu\in\Pi_T$, the measure $\mu(A)\coloneqq\int_T[\mathbf 1_A(a,b)+\mathbf 1_A(b,a)]\de\nu(a,b)$ is symmetric with uniform margins by the same computation, hence doubly stochastic, and reverses the correspondence.
    Now apply Lemma~\ref{lem:moments} and the symmetrization remark above; the constraint $\phi=x$ becomes $\E|D|=m$, i.e.\
	\(
	\int_T(b-a)\de\nu(a,b)=\tfrac m2.
	\)
\end{proof}

Problem \eqref{eq:lp} is an infinite-dimensional linear program of transportation type: mass on the triangle $T$ with a coupled marginal constraint, cost $(b-a)^2$ and one moment constraint.
Sections~\ref{sec:attain} and \ref{sec:dual} solve it.
We remark that a second reduction is available and was used for independent numerical validation: $\phi$ depends on $C$ only through the diagonal $\delta_C(t)=C(t,t)$, symmetrization preserves the diagonal, and among symmetric copulas with diagonal $\delta$ the diagonal copula $E_\delta(u,v)=\min\{u,v,\tfrac{\delta(u)+\delta(v)}{2}\}$ is pointwise maximal, see \cite{nelsen2004best,ubedaflores2008best,fernandezsanchez2016members}.
Since $\varrho$ is increasing with respect to the pointwise order, the $\varrho$-maximal boundary equals $\max\{\varrho(E_\delta)\}$ over all diagonals with prescribed integral.
The objective is a concave functional of $\delta$.

Turning to the dual problem \eqref{eq:OT_dual2}, we introduce the following notion of feasibility.
A pair \((f,\theta)\) is called \emph{feasible} for the dual problem \eqref{eq:OT_dual2} if it satisfies the inequality constraint
\begin{align}
    f(a) + f(b) + \theta ((b-a)- m) \leq (b-a)^2 \quad \text{for all } (a,b)\in T.
\end{align}
By symmetry, this is equivalent to the full-square constraint in \eqref{eq:OT_dual2}.
Similarly, we call a measure \(\pi\in\Pi\) \emph{feasible} for the primal problem \eqref{eq:OT_primal} if it satisfies the fixed-moment constraint \eqref{eq:OT-moment-constraint} below.
Since the optimization problem \eqref{eq:OT_primal} is symmetric, we may equivalently focus on feasible measures \(\nu\in\Pi_T\) satisfying
\[
    \int_T(b-a)\de\nu(a,b)=\frac m2.
\]

For the next result, recall that \(\mathsf{D}(m)\) is the value of the dual formulation where we optimize with respect to one feasible potential.

\begin{lemma}\label{lem_opt}
    For any feasible potential \((f,\theta)\) and for any feasible \(\nu\in \Pi_T\), we have
    \begin{align}\label{eq_lem_dualineq}
        \int_0^1 f(u) \de u \leq \int_T (b-a)^2 \de \nu(a,b)
    \end{align}
    and thus \(\mathsf{D}(m) \leq \mathsf{P}(m)\).
\end{lemma}

\begin{proof}
    Since \(\nu_1+\nu_2=\lam\), we obtain
\begin{align*}
    \int_0^1 f(u) \de u &= \int_0^1 f(a) \de \nu_1(a) + \int_0^1 f(b) \de \nu_2(b) \\
    &= \int_T (f(a) + f(b)) \de \nu(a,b) \\
    &= \int_T \left[ f(a) + f(b) + \theta\left((b-a) - m\right)\right] \de \nu(a,b)\\
    &\leq \int_T (b-a)^2 \de \nu(a,b),
\end{align*}
because \(\nu(T)=\tfrac12\) and feasibility gives \(\int_T((b-a)-m)\de\nu=\tfrac m2-m\nu(T)=0\).

Now let \(\pi\in\Pi\) be feasible for the primal problem.
Its symmetrization \(\pi^{\mathrm{sym}}\coloneqq\tfrac12(\pi+S_\#\pi)\) has the same cost and moment.
The associated triangular measure
\[
    \nu_\pi\coloneqq
    \frac12(\min,\max)_\#\pi^{\mathrm{sym}}
\]
is feasible by Lemma~\ref{prop:pairlp}.
Applying the inequality just proved yields
\[
    2\int_0^1f(u)\de u
    \leq
    2\int_T(b-a)^2\de\nu_\pi(a,b)
    =
    \int_{[0,1]^2}(b-a)^2\de\pi(a,b).
\]
Taking the supremum over feasible potentials and the minimum over feasible couplings proves \(\mathsf D(m)\leq\mathsf P(m)\).
\end{proof}

The classical Kantorovich duality theorem identifies the primal and dual values under standard regularity assumptions \cite[Theorem~5.10]{villani2009optimal}.
For the present setting, the constrained extension in \cite[Theorem~2.1]{zaev2015monge} can be stated as follows.
For every cost \(c\in C([0,1]^2)\) and every linear subspace \(W\subset C([0,1]^2)\),  we have the strong duality
\begin{equation}\label{eq:zaev-duality}
\inf_{\substack{\pi\in\Pi\\
        \int w\,\de\pi=0\ \text{for all }w\in W}}
    \int_{[0,1]^2}c\,\de\pi \\
=
\sup_{\substack{\varphi,\psi\in C([0,1]),\ w\in W\\
        \varphi\oplus\psi+w\leq c}}
    \left\{
        \int_0^1\varphi\,\de\lam
        +
        \int_0^1\psi\,\de\lam
    \right\}.
\end{equation}
The additional linear constraint in our primal problem \eqref{eq:OT_primal} fixes the first absolute moment of the difference \(U-V\):
\begin{equation}\label{eq:OT-moment-constraint}
    \int_{[0,1]^2}|b-a|\de\pi(a,b)=m.
\end{equation}
The following proposition identifies the left-hand side of \eqref{eq:zaev-duality} with \(\mathsf P(m)\) in \eqref{eq:OT_primal} and its right-hand side, after symmetrization, with \(\mathsf D(m)\) in \eqref{eq:OT_dual2}.
The annihilator constraint on the left-hand side reproduces \eqref{eq:OT-moment-constraint} and the two-potential supremum on the right-hand side equals the one-potential value \(\mathsf D(m)\).
Therefore, \eqref{eq:zaev-duality} reduces to the following strong duality result.

\begin{proposition}[Kantorovich duality with linear constraint]\label{prop_KR_duality}
    For every \(m\in[0,\tfrac12]\), the values of the primal problem \eqref{eq:OT_primal} and the dual problem \eqref{eq:OT_dual2} coincide, i.e., \(\mathsf{P}(m)=\mathsf{D}(m)\).
\end{proposition}

It is important to mention that minimizers of the primal problem exist here by compactness, whereas the supremum in the dual problem need not be attained, even when strong duality holds. 
If, however, there is some feasible pair \((f,\theta)\) and some feasible measure \(\nu\in\Pi_T\) such that equality holds in \eqref{eq_lem_dualineq}, then \(\nu\) is optimal for the triangular problem \eqref{eq:lp}, and \(\nu+S_\#\nu\) is optimal for the primal problem \eqref{eq:OT_primal}.
Furthermore, \((f,\theta)\) is dual optimal in the sense that the supremum in \eqref{eq:OT_dual2} is attained.
This motivates us to consider the \emph{contact set}
\begin{align}\label{def_contactset}
    \Gamma_{f,\theta}\coloneqq\left\{(a,b)\in T \mid f(a)+f(b)+\theta((b-a)-m)=(b-a)^2\right\},
\end{align}
for any feasible \((f,\theta)\).
The following result is a direct consequence of Lemma \ref{lem_opt} and Proposition \ref{prop_KR_duality}.
It states, in particular, that if \(\nu\) is concentrated on the contact set of \((f,\theta)\), then it is optimal.

\begin{corollary}[Optimality via contact set]\label{cor_feas}
    Let \(\nu\in\Pi_T\) be feasible and let \((f,\theta)\) be feasible.
    If \(\nu(T\setminus\Gamma_{f,\theta})=0\), then \(\nu\) is optimal for \eqref{eq:lp}, \(\nu+S_\#\nu\) is optimal for \eqref{eq:OT_primal}, and \((f,\theta)\) is optimal for \eqref{eq:OT_dual2}.
\end{corollary}

To solve the maximization problem \eqref{defmaxprob}, we will apply Corollary \ref{cor_feas}.
To be precise, in Section \ref{sec:attain}, we provide a feasible candidate solution \(\nu\in \Pi_T\) and determine the associated copula \(C_x\).
In Section \ref{sec:dual}, we provide a feasible pair \((f,\theta)\) and then show that \(\nu\) is concentrated on the contact set of \((f,\theta)\).
Consequently, \(C_x\) is optimal and, as we will verify, it satisfies \(\phi(C_x) = x\) and \(\varrho(C_x) = \overline{\varrho}(x)\), so that \(\overline{\varrho}\) determines the upper boundary of \(\Omega_{\varrho,\phi}\).

\section{Construction of optimal copulas}\label{sec:attain}

In this section, we construct a family of distributions \(\pi_x\in \Pi\), \(x\in [-\frac 1 2, 1)\), such that the associated copula family \((C_x)_{x\in [-1/2,1]}\) satisfies \(\phi(C_x)= x\).
In Proposition~\ref{prop:rho-phi-Cx}, we will show that these copulas are optimal in the sense that they maximize Spearman's rho, i.e. \(\varrho(C_x) = \overline{\varrho}(x)\).
The case \(x=1\) is trivial, where we define \(C_x(u,v) := \min\{u,v\}\).
Hence, we assume \(x\in [-\frac 1 2,1)\) in the following.

Recall that, for \((U,V)\sim C\), we have \(\phi(C) = 1 - 3 \E|U-V|\).
Then, for fixed \(\phi(C) = x\), the first moment of \(|U-V|\) is given by
\begin{align}\label{eq:mean-gap}
    m := m(x) = \frac{1-x}{3}.
\end{align}
We aim to minimize the second moment of \(U-V\) over all copulas with fixed first absolute moment \(m\).
In view of \eqref{eq:mean-gap}, choose \(N\in \N\) such that
\begin{align}\label{eq:footrule-interval}
    x\in I_N = \left[1-\frac{3}{2N}, 1-\frac{3}{2N+2}\right).
\end{align}
Note that \(I_1 = [-\frac 1 2 ,\frac 1 4)\), \(I_2 = [\frac 1 4 , \frac 1 2)\), \(I_3 = [\frac 1 2, \frac 5 8)\), ..., and \(\bigcup_{N\in \N} I_N = [-\frac 1 2 , 1)\).
Defining
\begin{align}\label{def:RL}
    R :=\frac{1}{2N+2} \quad \text{and} \quad L := \frac{1}{2N},
\end{align}
the interval in \eqref{eq:footrule-interval} can also be written as \(I_N = [1-3L,1-3R)\).
Now, we divide \(I_N\) into two intervals of equal length, where the left and right part are denoted by
\begin{align}\label{def:I_NL-I_NR}
    I_N^L:= \left[ 1-3L, 1-\frac 3 2 (L+R)\right) \quad \text{and} \quad I_N^R:= \left[1-\frac 3 2 (L+R), 1-3R\right),
\end{align}
respectively.
Define
\begin{align}\label{eq:dN}
    d_N := \frac{1}{L-R} = 2N(N+1)    
\end{align}
as the reciprocal of the interval length of \([R,L]\).
Then we have \(d_N \lam([R,L)) = d_N(L-R) = 1\).
Further, it is \(\lam(I_N) = 3/d_N\).
Note that
\begin{align*}
    x\in I_N^L ~ \Longleftrightarrow ~ m\in \left(\frac{L+R}{2},L\right]\qquad\text{and}\qquad
    x\in I_N^R ~ \Longleftrightarrow ~ m\in \left(R,\frac{L+R}{2}\right].
\end{align*}
To define the candidate couplings \((\pi_x)_x\) via measures \((\nu_x)\) in \(\Pi_T\), we introduce three functions depending on the constraint \(x\) as follows:
\begin{align}\label{def:vptheta}
\begin{split}
    v = v(x) &:= \begin{cases}
        \sqrt{\frac {2 (L-m(x))}{d_N}} , &\text{for } x\in I_N^L,\\
        \sqrt{\frac {2 (m(x)-R)}{d_N}} , &\text{for } x\in I_N^R.
    \end{cases}\\
    p = p(x) &:= 1-d_N v(x),\\
    \theta = \theta(x) &:= \begin{cases}
        2L - v(x), &\text{for } x\in I_N^L,\\
        2R + v(x), &\text{for } x\in I_N^R.
    \end{cases}
\end{split}
\end{align}
The measure \(\nu_x\) is now defined by
\begin{align}\label{def_nu_x}
    \nu_x := \begin{cases}
        \nu_x^L, &\text{for } x\in I_N^L,\\
        \nu_x^R, &\text{for } x\in I_N^R,
    \end{cases}
\end{align}
with
\begin{align}\label{def_nu_xL}
    \nu_x^L&:= \sum_{k=0}^{N-1} (A_k^L)_\#(\lam\vert_{[0,pL]}) + \sum_{k=0}^{N-1} \alpha_k\, (B_k^L)_\#(\lam\vert_{[0,Nv]}) + \sum_{k=0}^{N-1} \beta_k\, (C_k^L)_\#(\lam\vert_{[0,Nv]}),\\
    \label{def_nu_xR} \nu_x^R&:= \sum_{k=0}^{N} (A_k^R)_\#(\lam\vert_{[0,pR]}) + \sum_{k=0}^{N-1} \alpha_k\, (B_k^R)_\#(\lam\vert_{[0,Nv]}) + \sum_{k=0}^{N-1} \beta_k\, (C_k^R)_\#(\lam\vert_{[0,Nv]})
\end{align}
for weights \(\alpha_k := 1 - \frac k N\) and \(\beta_k:= \frac{k+1}{N}\).
Further, for
\begin{align}
    \ell_L(t):= L-v+\frac t N \quad \text{and} \quad \ell_R(t):= R + \frac t N, \qquad  t\in [0,Nv],
\end{align}
the pushforward measures are defined via the mappings
\begin{align}
    \label{def:A_kL} A_k^L(s) &:= \left(k\theta + Nv + s, k\theta + Nv + s + L\right), & & 0\leq s \leq pL,\\
    B_k^L(t) &:= \left(k\theta + t, k\theta + t + \ell_L(t)\right),  &&  0 \leq t \leq Nv,\\
    C_k^L(t) &:= \left(k\theta + t + \ell_L(t), (k+1) \theta + t\right) &&  0 \leq t \leq Nv,\\
    \label{def:A_kR} A_k^R(s)
    &:=
    \left(
        k\theta+s,\,
        k\theta+s+R
    \right),
    && 0\leq s \leq pR,\\
    B_k^R(t)
    &:=
    \left(
        k\theta+pR+t,\,
        k\theta+pR+t+\ell_R(t)
    \right),
    && 0 \leq t \leq Nv,\\
    \label{def:C_kR} C_k^R(t)
    &:=
    \left(
        k\theta+pR+t+\ell_R(t),\,
        (k+1)\theta+pR+t
    \right),
    && 0 \leq t \leq Nv.
\end{align}

In Lemma \ref{lem_nux_in_PiT} below, we show that \(\nu_x\in \Pi_T\).
Hence, by Lemma \ref{lem:pair-measure-bijection}, the following measure is a coupling in \(\Pi_{\operatorname{sym}}\):
\begin{align}\label{def:pi_x}
        \pi_x \coloneqq \nu_x + S_\# \nu_x.
    \end{align}
Since \(\pi_x\) has \(\cU(0,1)\)-marginals, we may define the associated copula \(C_x\) as follows.

\begin{definition}[Optimal copula \(C_x\)]\label{def_Cx}
    For \(x\in [-\frac 1 2, 1)\), the copula associated with \(\pi_x\) is defined by
    \begin{align}\label{def_C_x}
        C_x(u,v) &:= \pi_x([0,u]\times [0,v]), \quad (u,v)\in [0,1]^2.
    \end{align}
    For \(x=1\), we define \(C_1(u,v) := \min\{u,v\}\) as the upper Fr\'{e}chet copula.
\end{definition}

\begin{remark}\leavevmode
    \begin{enumerate}[label = (\alph*)]
        \item Since \(\pi_x\) is symmetric, also \(C_x\) is symmetric.
        In Proposition \ref{prop_C_x}, we will show that \(\phi(C_x) = x\) and we will determine a closed-form expression for \(\varrho(C_x)\).
        In Proposition \ref{prop:rho-phi-Cx}, we verify that \(\varrho(C) \leq \varrho(C_x)\) for every copula \(C\) with \(\phi(C)=x\); hence, \(C_x\) is optimal in the sense that it maximizes Spearman's \(\varrho\).
        \item Figure~\ref{fig:construction} displays the support of \(\nu_x\) for one right-half parameter choice, together with the corresponding dual potential.
        Figure~\ref{fig:copulas} displays the support of \(\pi_x=\nu_x+S_\#\nu_x\), which is the distribution underlying \(C_x\).
        For \(x \in \{1-\frac{3}{2N}\colon N\in \N\} = \{-\frac 1 2,\frac 1 4, \frac 1 2, \frac 5 8,\ldots\}\), \(C_x\) reduces to an equidistant even shuffle of min copula for which optimality has been shown in \cite[Theorem~11 and Example~12]{kokolbukovsek2024exact}. 
    \end{enumerate}
\end{remark}

\begin{figure}[p!]
	\centering
	\includegraphics[width=\textwidth]{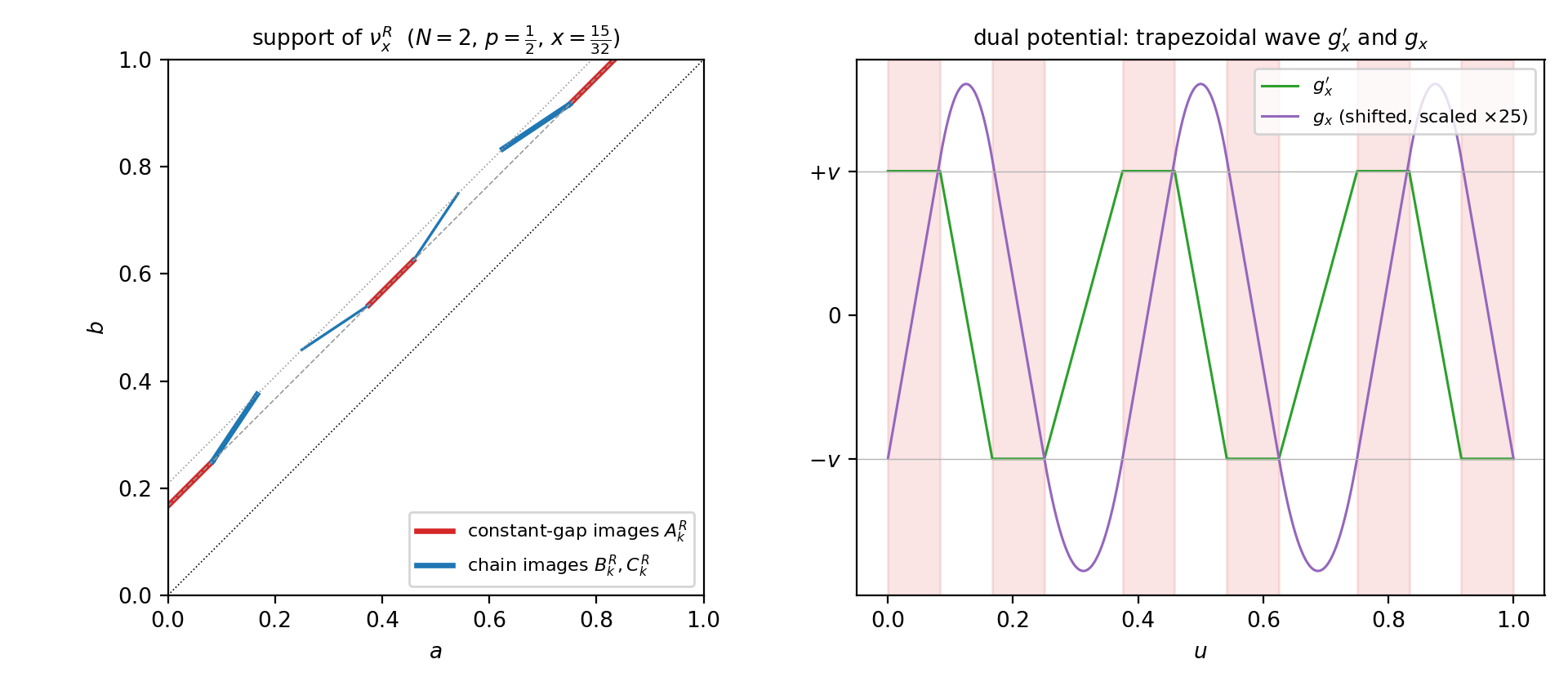}
	\caption{Left: support of the measure \(\nu_x=\nu_x^R\) from \eqref{def_nu_x} and \eqref{def_nu_xR}, for \(N=2\) and \(p=\tfrac12\) (equivalently, \(x=\tfrac{15}{32}\)), inside \(T=\{(a,b):a\le b\}\).
	The red segments are the images of the maps \(A_k^R\) and have constant difference \(b-a=R\).
	The blue segments are the images of \(B_k^R\) and \(C_k^R\); their differences \(b-a\) fill the interval \([R,R+v]\).
	Thus, projection by \(q(a,b)=b-a\) yields the right-half distance law in Lemma~\ref{lem_U-V}.
	Grey lines mark \(b-a=R\) (dashed) and \(b-a=R+v\) (dotted).
	Right: the trapezoidal derivative \(g_x'\) and the \(\theta\)-periodic dual potential \(g_x\) from Definition~\ref{def:dual-potential}; the shaded intervals are the first- and second-coordinate ranges of the constant-difference segments.}
	\label{fig:construction}
\end{figure}

\begin{figure}[p!]
	\centering
	\includegraphics[width=\textwidth]{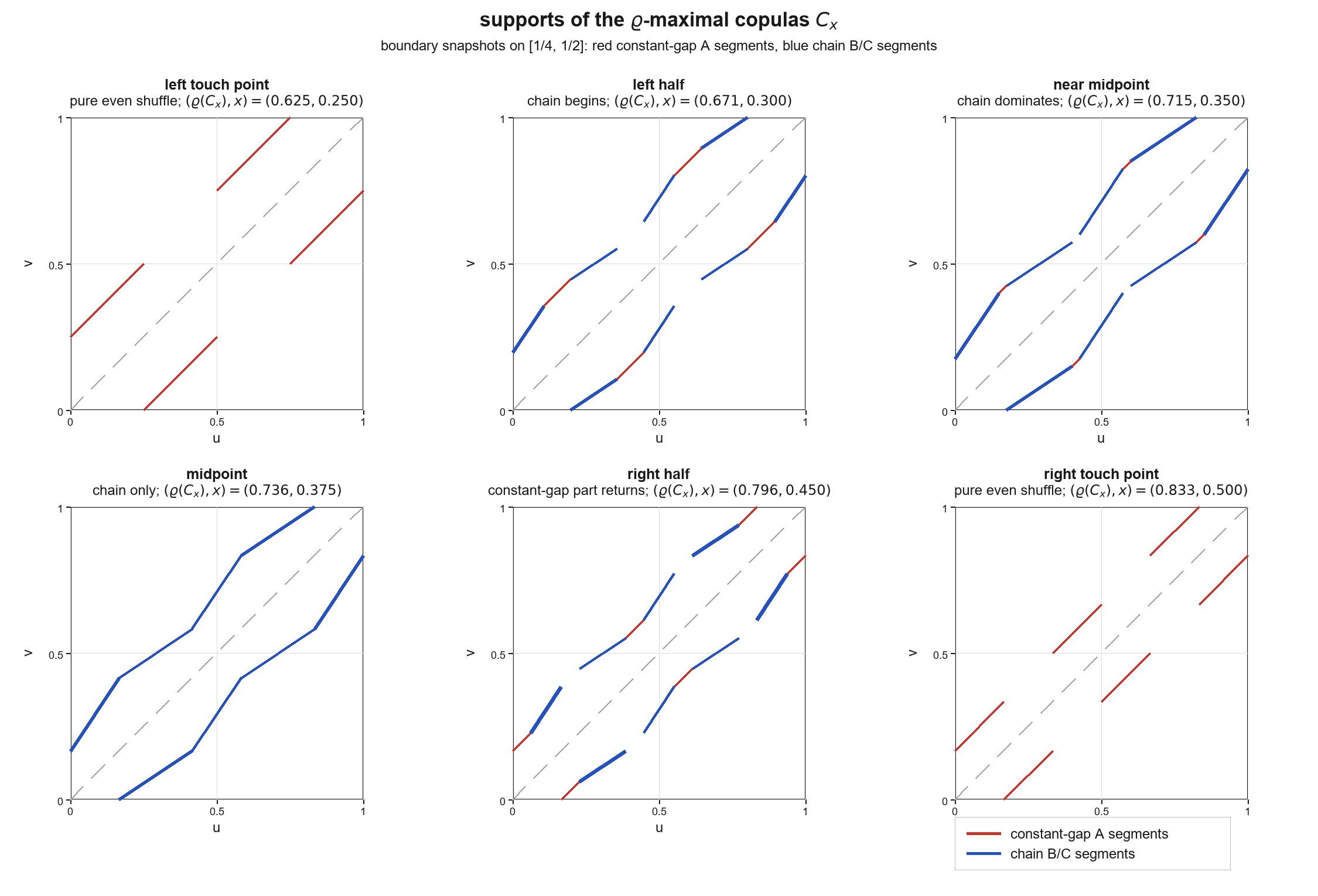}
	\caption{Six supports of \(\pi_x=\nu_x+S_\#\nu_x\), the joint distributions associated with the \(\varrho\)-maximal copulas \(C_x\), along the boundary segment \(x\in[\tfrac14,\tfrac12]\).
	Red denotes the images of the maps \(A_k^L\) or \(A_k^R\), whereas blue denotes the images of \(B_k^L,C_k^L\) or \(B_k^R,C_k^R\), as defined in \eqref{def:A_kL}--\eqref{def:C_kR}, together with their reflections under \(S\).
	On the red segments, \(b-a\) is constant.
	At the two touch points \(x=\tfrac14\) and \(x=\tfrac12\), the supports are equidistant even shuffles.
	Toward the midpoint, the parameter \(v\) increases and the remaining mass \(p=1-d_Nv\) on the constant-difference segments decreases; see \eqref{eq:dN} and \eqref{def:vptheta}.
	At the midpoint \(x=\tfrac38\), we have \(v=L-R=1/d_N\) and hence \(p=0\), so only the blue segments remain.}
	\label{fig:copulas}
\end{figure}

For proving optimality of \(C_x\), we begin by showing that \(\nu_x\) defined in \eqref{def_nu_x} is in \(\Pi_T\).

\begin{lemma}\label{lem_nux_in_PiT}
For every \(x\in [-\frac 1 2 ,1)\), we have \(\nu_x\in \Pi_T\).
\end{lemma}

\begin{proof}
Recall that \(N\in\mathbb N\) is chosen such that \(x\in I_N\).
Further, recall \(v\), \(p\), \(\theta\) from \eqref{def:vptheta}.
By the definitions of \(I_N^L\) and \(I_N^R\) in \eqref{def:I_NL-I_NR}, we have \(0\leq v\leq L-R=\frac1{d_N}\).
Consequently,
\[
    0\leq p=1-d_Nv\leq1.
\]
Moreover, \(\alpha_k=1-\frac{k}{N}\geq 0\) and \(\beta_k=\frac{k+1}{N}\geq 0 \) for \(k=0,\ldots,N-1\).
Thus, \(\nu_x\) is a finite nonnegative Borel measure.

It remains to show that \(\nu_x\) is supported on \(T\) and satisfies
\begin{align}\label{prop_proj_nu_x}
    (\operatorname{proj}_1)_\#\nu_x
    +
    (\operatorname{proj}_2)_\#\nu_x
    =
    \lam.
\end{align}

We treat the left and right parts separately.

\medskip
\noindent

\emph{Case 1 (the right-side part): } Let \(x\in I_N^R\).
Since \(d_NR=N\), we obtain \(pR=(1-d_Nv)R=R-Nv\), and therefore
\begin{equation}\label{eq:right-identities-proof}
    pR+Nv=R.
\end{equation}
Furthermore, using \(\theta=2R+v\), we have
\begin{align}
    N\theta+R+pR
    &=
    N(2R+v)+R+pR \notag\\
    &=
    2NR+Nv+R+pR \notag\\
    &=
    2(N+1)R
    =1.
    \label{eq:right-closure-proof}
\end{align}

The coordinate differences appearing in the definition of \(\nu_x^R\) are given by
\[
\begin{aligned}
    \operatorname{proj}_2(A_k^R(s))
    -
    \operatorname{proj}_1(A_k^R(s))
    &=R,\\
    \operatorname{proj}_2(B_k^R(t))
    -
    \operatorname{proj}_1(B_k^R(t))
    &=\ell_R(t)
      =R+\frac{t}{N},\\
    \operatorname{proj}_2(C_k^R(t))
    -
    \operatorname{proj}_1(C_k^R(t))
    &=\theta-\ell_R(t)
      =R+v-\frac{t}{N}.
\end{aligned}
\]
Hence all these lengths belong to the interval \([R,R+v]\) and are nonnegative.

Using \eqref{eq:right-identities-proof}, the coordinate images are
\[
\begin{aligned}
    (\operatorname{proj}_1\circ A_k^R)([0,pR])
    &=[k\theta,k\theta+pR],\\
    (\operatorname{proj}_2\circ A_k^R)([0,pR])
    &=[k\theta+R,k\theta+R+pR],\\
    (\operatorname{proj}_1\circ B_k^R)([0,Nv])
    &=[k\theta+pR,k\theta+R],\\
    (\operatorname{proj}_2\circ B_k^R)([0,Nv])
    &=[k\theta+R+pR,(k+1)\theta],\\
    (\operatorname{proj}_1\circ C_k^R)([0,Nv])
    &=[k\theta+R+pR,(k+1)\theta],\\
    (\operatorname{proj}_2\circ C_k^R)([0,Nv])
    &=[(k+1)\theta+pR,(k+1)\theta+R].
\end{aligned}
\]
Together with \eqref{eq:right-closure-proof}, this shows that all coordinates lie in \([0,1]\), and thus all mappings take values in \(T\).

To verify \eqref{prop_proj_nu_x}, set
\[
    \gamma_x^R
    :=
    (\operatorname{proj}_1)_\#\nu_x^R
    +
    (\operatorname{proj}_2)_\#\nu_x^R.
\]
We show that \(\gamma_x^R=\lam\).

We repeatedly use the following elementary change-of-variables fact: if \(J\) is an interval and \(g(t)=a+ct\), with \(c>0\), then the push-forward of \(w\,\lam\vert_J\) under \(g\) has density \(w/c\) on \(g(J)\).

For \(k=0,\ldots,N-1\), decompose
\[
    [k\theta,(k+1)\theta]
    =
    J_{k,1}\cup J_{k,2}\cup J_{k,3}\cup J_{k,4},
\]
where
\begin{align*}
    J_{k,1}&=[k\theta,k\theta+pR),\\
    J_{k,2}&=[k\theta+pR,k\theta+R),\\
    J_{k,3}&=[k\theta+R,k\theta+R+pR),\\
    J_{k,4}&=[k\theta+R+pR,(k+1)\theta].
\end{align*}
On \(J_{k,1}\), the first coordinate of \(A_k^R\) in \eqref{def:A_kR} has slope \(1\) and weight \(1\).
Since, on \(J_{k,1}\), only the first coordinate of \(A_k^R\) contributes,  \(\gamma_x^R\) has density \(1\) on \(J_{k,1}\).

On \(J_{k,2}\), the first coordinate of \(B_k^R\) contributes density \(\alpha_k\).
If \(k\geq1\), the second coordinate of \(C_{k-1}^R\) contributes density \(\beta_{k-1}\).
Since
\[
    \alpha_k+\beta_{k-1}
    =
    1-\frac{k}{N}+\frac{k}{N}
    =1,
\]
the total density is \(1\).
For \(k=0\), only \(B_0^R\) contributes, and \(\alpha_0=1\).
Hence, \(\gamma_x^R\) has density \(1\) on \(J_{k,2}\).

On \(J_{k,3}\), only the second coordinate of \(A_k^R\) contributes.
It has slope \(1\) and weight \(1\), and therefore \(\gamma_x^R\) has density \(1\) on \(J_{k,3}\).

Finally, \(J_{k,4}\) is covered by the second coordinate of \(B_k^R\) and the first coordinate of \(C_k^R\).
Both maps have slope
\[
    1+\ell_R'(t)
    =
    1+\frac1N
    =
    \frac{N+1}{N}.
\]
Hence their combined density equals
\[
    \frac{\alpha_k+\beta_k}{(N+1)/N}
    =
    \frac{
        1-\frac{k}{N}+\frac{k+1}{N}
    }{(N+1)/N}
    =1.
\]

It remains to consider the terminal interval \([N\theta,1]\).
By \eqref{eq:right-closure-proof},
\[
    [N\theta,1]
    =
    [N\theta,N\theta+pR]
    \cup
    [N\theta+pR,N\theta+R]
    \cup
    [N\theta+R,1].
\]
The first and third intervals are covered by the two coordinate projections of \(A_N^R\), each with density \(1\).
The middle interval is covered by the second coordinate of \(C_{N-1}^R\), whose weight is
\(
    \beta_{N-1}=1.
\)
Thus, \(\gamma_x^R\) has density \(1\) on all of \([0,1]\), and therefore \((\operatorname{proj}_1)_\#\nu_x^R + (\operatorname{proj}_2)_\#\nu_x^R = \lam. \) follows.

\medskip
\noindent
\emph{Case 2 (the left-side part): } Let \(x\in I_N^L\).
Since \(d_NL=N+1\), we have \(pL=(1-d_Nv)L=L-(N+1)v\), and hence
\begin{align}\label{eq:left-identities-proof}
    Nv+pL=L-v.
\end{align}
Moreover, using \(\theta=2L-v\), we have
\begin{equation}\label{eq:left-closure-proof}
    N\theta+Nv
    =
    N(2L-v)+Nv
    =
    2NL
    =
    1.
\end{equation}

The coordinate differences of the pairs appearing in \(\nu_x^L\) are
\[
\begin{aligned}
    \operatorname{proj}_2(A_k^L(s))
    -
    \operatorname{proj}_1(A_k^L(s))
    &=L,\\
    \operatorname{proj}_2(B_k^L(t))
    -
    \operatorname{proj}_1(B_k^L(t))
    &=\ell_L(t)
      =L-v+\frac{t}{N},\\
    \operatorname{proj}_2(C_k^L(t))
    -
    \operatorname{proj}_1(C_k^L(t))
    &=\theta-\ell_L(t)
      =L-\frac{t}{N}.
\end{aligned}
\]
Thus all differences belong to the interval \([L-v,L]\) and are non-negative.

By \eqref{eq:left-identities-proof}, the coordinate images are
\[
\begin{aligned}
    \operatorname{proj}_1\circ B_k^L([0,Nv])
    &=[k\theta,k\theta+Nv],\\
    \operatorname{proj}_1\circ A_k^L([0,pL])
    &=[k\theta+Nv,k\theta+L-v],\\
    \operatorname{proj}_2\circ B_k^L([0,Nv])
    &=[k\theta+L-v,k\theta+L+Nv],\\
    \operatorname{proj}_1\circ C_k^L([0,Nv])
    &=[k\theta+L-v,k\theta+L+Nv],\\
    \operatorname{proj}_2\circ A_k^L([0,pL])
    &=[k\theta+L+Nv,(k+1)\theta],\\
    \operatorname{proj}_2\circ C_k^L([0,Nv])
    &=[(k+1)\theta,(k+1)\theta+Nv].
\end{aligned}
\]
Together with \eqref{eq:left-closure-proof}, this shows that all coordinates lie in \([0,1]\), and hence all mappings take values in \(T\).

To verify \eqref{prop_proj_nu_x}, define
\[
    \gamma_x^L
    :=
    (\operatorname{proj}_1)_\#\nu_x^L
    +
    (\operatorname{proj}_2)_\#\nu_x^L.
\]
For \(k=0,\ldots,N-1\), decompose
\[
    [k\theta,(k+1)\theta]
    =
    K_{k,1}\cup K_{k,2}\cup K_{k,3}\cup K_{k,4},
\]
where
\[
\begin{aligned}
    K_{k,1}&=[k\theta,k\theta+Nv),\\
    K_{k,2}&=[k\theta+Nv,k\theta+L-v),\\
    K_{k,3}&=[k\theta+L-v,k\theta+L+Nv),\\
    K_{k,4}&=[k\theta+L+Nv,(k+1)\theta].
\end{aligned}
\]

On \(K_{k,1}\), the first coordinate of \(B_k^L\) contributes density \(\alpha_k\).
For \(k\geq1\), the second coordinate of \(C_{k-1}^L\) contributes density \(\beta_{k-1}\).
Hence the total density is
\(
    \alpha_k+\beta_{k-1} = \left(1 - \frac k N \right) + \frac k N =1.
\)
For \(k=0\), only \(B_0^L\) contributes, and \(\alpha_0=1\).

\noindent On \(K_{k,2}\), the first coordinate of \(A_k^L\) has slope \(1\) and
weight \(1\), and therefore contributes density \(1\).

\noindent On \(K_{k,3}\), the second coordinate of \(B_k^L\) and the first
coordinate of \(C_k^L\) both have slope
\[
    1+\ell_L'(t)
    =
    \frac{N+1}{N}.
\]
Their combined density is
\[
    \frac{\alpha_k+\beta_k}{(N+1)/N}=1.
\]

\noindent On \(K_{k,4}\), the second coordinate of \(A_k^L\) has slope \(1\) and
weight \(1\), and hence contributes density \(1\).

\noindent Finally, by \eqref{eq:left-closure-proof}, the remaining terminal
interval is
\[
    [N\theta,1]=[N\theta,N\theta+Nv].
\]
It is covered by the second coordinate of \(C_{N-1}^L\), whose weight is
\(
    \beta_{N-1}=1.
\)
Consequently, \(\gamma_x^L\) has density \(1\) on \([0,1]\), and thus
\(
    (\operatorname{proj}_1)_\#\nu_x^L
    +
    (\operatorname{proj}_2)_\#\nu_x^L
    =
    \lam.
\)
\end{proof}

Next, we show that \(\nu_x\) is feasible.

\begin{lemma}[Feasibility of \(\nu_x\)]
\label{lem:primal-feasibility}
For every \(x\in[-\frac12,1)\), the measure \(\nu_x\) is feasible for \(m=\frac{1-x}{3}\), i.e.,
\[
    \int_T(b-a)\,\de\nu_x(a,b)=\frac{m}{2}.
\]
\end{lemma}

\begin{proof}
Recall the definition of the weights \(\alpha_k = 1 - \frac k N\) and \(\beta_k = \frac{k+1}{N}\) occurring in \eqref{def_nu_xL} and \eqref{def_nu_xR}.
Then we have
\[
    \sum_{k=0}^{N-1}\alpha_k
    =
    \sum_{k=0}^{N-1}\beta_k
    =
    \frac{N+1}{2}.
\]

First, let \(x\in I_N^R\).
The coordinate differences are generated by \(A_k^R\), \(B_k^R\), and \(C_k^R\) are \(R\), \(\ell_R(t)\), and \(\theta-\ell_R(t)\), respectively.
Consequently,
\begin{align*}
    \int_T(b-a)\,\de\nu_x^R(a,b)
    &=
    \sum_{k=0}^{N}\int_0^{pR}R\,\de s
    +
    \sum_{k=0}^{N-1}\alpha_k
    \int_0^{Nv}\ell_R(t)\,\de t
    +\sum_{k=0}^{N-1}\beta_k
    \int_0^{Nv}\bigl(\theta-\ell_R(t)\bigr)\,\de t\\
    &=
    (N+1)pR^2
    +
    \frac{N+1}{2}
    \int_0^{Nv}\theta\,\de t 
    =
    \frac{pR}{2}
    +
    \frac{d_N\theta v}{4},
\end{align*}
where we use \((N+1)R=\frac12\) and \(d_N=2N(N+1)\) by \eqref{def:RL} and \eqref{eq:dN}.
Since, by \eqref{def:vptheta}, \(p=1-d_Nv\) and \(\theta=2R+v\), we obtain
\begin{align*}
    \frac{pR}{2}+\frac{d_N\theta v}{4}
    &=
    \frac12
    \left(
        pR+\frac{d_N\theta v}{2}
    \right)
    =
    \frac12
    \left(
        R-d_NRv+d_NRv+\frac{d_Nv^2}{2}
    \right)
    =
    \frac12
    \left(
        R+\frac{d_Nv^2}{2}
    \right).
\end{align*}
By definition of \(v(x)\) on \(I_N^R\) in \eqref{def:vptheta}, it follows that \(m(x)=R+\frac{d_Nv^2}{2}\).
Hence,
\[
    \int_T(b-a)\,\de\nu_x^R(a,b)
    =
    \frac{m(x)}{2}.
\]

Now let \(x\in I_N^L\).
The differences of the coordinates generated by \(A_k^L\), \(B_k^L\), and \(C_k^L\) are \(L\), \(\ell_L(t)\), and \(\theta-\ell_L(t)\), respectively.
Thus,
\begin{align*}
    \int_T(b-a)\,\de\nu_x^L(a,b)
    &=
    \sum_{k=0}^{N-1}\int_0^{pL}L\,\de s
    +\sum_{k=0}^{N-1}\alpha_k
    \int_0^{Nv}\ell_L(t)\,\de t
    +
    \sum_{k=0}^{N-1}\beta_k
    \int_0^{Nv}\bigl(\theta-\ell_L(t)\bigr)\,\de t\\
    &=
    NpL^2
    +
    \frac{N+1}{2}
    \int_0^{Nv}\theta\,\de t
    =
    \frac{pL}{2}
    +
    \frac{d_N\theta v}{4},
\end{align*}
where we used \(NL=\frac12\).
Since \(p=1-d_Nv\) and \(\theta=2L-v\), it follows that
\begin{align*}
    \frac{pL}{2}+\frac{d_N\theta v}{4}
    &=
    \frac12
    \left(
        pL+\frac{d_N\theta v}{2}
    \right)
    =
    \frac12
    \left(
        L-d_NLv+d_NLv-\frac{d_Nv^2}{2}
    \right)
    =
    \frac12
    \left(
        L-\frac{d_Nv^2}{2}
    \right).
\end{align*}
By definition of \(v(x)\) on \(I_N^L\), it follows that \(m(x)=L-\frac{d_Nv^2}{2}\).
Therefore,
\[
    \int_T(b-a)\,\de\nu_x^L(a,b)
    =
    \frac{m(x)}{2}.
\]
This proves the assertion.
\end{proof}

To calculate Spearman's rho and Spearman's footrule for \(C_x\), we determine the distribution of \(|U-V|\) for \((U,V)\sim C_x\).
We denote by \(\delta_y\) the Dirac measure in \(y\) and write \(\lam\vert_I\) for the restriction of Lebesgue measure to an interval \(I\).

\begin{lemma}\label{lem_U-V}
    Consider the coupling \(\pi_x\) defined in \eqref{def:pi_x}.
    For \((U,V)\sim \pi_x\), we have
    \begin{align}\label{eq:distance-law-right}
    \mathcal L_{\pi_x}(|U-V|)
    = \begin{cases}
        p\delta_R+d_N\lam\vert_{[R,R+v]}, &\text{if } x\in I_N^R,\\
        p\delta_L+d_N\lam\vert_{[L-v,L]}, &\text{if } x\in I_N^L,
    \end{cases}
\end{align}
where \(R\), \(L\), \(I_N^R\), \(I_N^L\), \(d_N\), \(p\), and \(v\) are defined in \eqref{def:RL}--\eqref{def:vptheta}.
\end{lemma}

\begin{proof}
    Define the mappings
\[
    q:T\to[0,1],
    \qquad
    q(a,b):=b-a,
\]
and
\[
    d:[0,1]^2\to[0,1],
    \qquad
    d(a,b):=|b-a|.
\]
Since \(\nu_x\) has support in \(T\), we have \(d=q\) on the support of \(\nu_x\).
Moreover, \(d\circ S=d\).
Hence,
\begin{align}
    d_\#\pi_x
    =
    d_\#\nu_x+d_\#(S_\#\nu_x)
    =
    q_\#\nu_x+(d\circ S)_\#\nu_x
    =
    2q_\#\nu_x.\label{eq:distance-law-pi-x}
\end{align}
Thus, \(2q_\#\nu_x\) is the distribution of \(|U-V|\).
We determine this distribution separately for the two parts of \(I_N\).

\medskip
\noindent
\emph{Case 1: \(x\in I_N^R\).}
Recall that the coordinate differences of the pairs generated by \(A_k^R\), \(B_k^R\), and \(C_k^R\) are \(R\), \(\ell_R(t)=R+\frac{t}{N}\), and \(\theta-\ell_R(t)=R+v-\frac{t}{N}\), respectively.
Therefore, for every bounded Borel function \(h:[0,1]\to\mathbb R\),
\begin{align*}
    \int_{[0,1]}h(r)\,\de(q_\#\nu_x^R)(r)
    &=
    \sum_{k=0}^{N}\int_0^{pR}h(R)\,\de s\\
    &\quad+
    \sum_{k=0}^{N-1}\alpha_k
    \int_0^{Nv}h(R+\tfrac{t}{N})\,\de t\\
    &\quad+
    \sum_{k=0}^{N-1}\beta_k
    \int_0^{Nv}h(R+v-\tfrac{t}{N})\,\de t.
\end{align*}
Using
\(
    (N+1)R=\frac12
\)
and
\(
    \sum_{k=0}^{N-1}\alpha_k
    =
    \sum_{k=0}^{N-1}\beta_k
    =
    \frac{N+1}{2},
\)
together with the changes of variables
\(r=R+\frac{t}{N}\) and \(r=R+v-\frac{t}{N}\),
we obtain
\begin{align*}
    \int_{[0,1]}h(r)\,\de(q_\#\nu_x^R)(r)
    =
    \frac{p}{2}h(R)
    +
    N(N+1)\int_R^{R+v}h(r)\,\de r
    =
    \frac{p}{2}h(R)
    +
    \frac{d_N}{2}\int_R^{R+v}h(r)\,\de r.
\end{align*}
Consequently,
\[
    q_\#\nu_x^R
    =
    \frac{p}{2}\delta_R
    +
    \frac{d_N}{2}\lam\vert_{[R,R+v]}.
\]
By \eqref{eq:distance-law-pi-x}, the distribution of \(|U-V|\) is therefore
\[
    \mathcal L_{\pi_x}(|U-V|)
    = d_\#\pi_x =
    p\delta_R+d_N\lam\vert_{[R,R+v]}.
\]
\medskip
\noindent
\emph{Case 2: \(x\in I_N^L\).}
Recall that the coordinate differences of the pairs generated by \(A_k^L\), \(B_k^L\), and \(C_k^L\) are \(L\), \(\ell_L(t)=L-v+\frac{t}{N}\), and \(\theta-\ell_L(t)=L-\frac{t}{N}\), respectively.
Thus, for every bounded Borel function \(h\),
\begin{align*}
    \int_{[0,1]}h(r)\,\de(q_\#\nu_x^L)(r)
    &=
    \sum_{k=0}^{N-1}\int_0^{pL}h(L)\,\de s\\
    &\quad+
    \sum_{k=0}^{N-1}\alpha_k
    \int_0^{Nv}h(L-v+\tfrac{t}{N})\,\de t\\
    &\quad+
    \sum_{k=0}^{N-1}\beta_k
    \int_0^{Nv}h(L-\tfrac{t}{N})\,\de t.
\end{align*}
Since \(NL=\frac12\), the same change-of-variables argument as before gives \(q_\#\nu_x^L = \frac{p}{2}\delta_L + \frac{d_N}{2}\lam\vert_{[L-v,L]}. \) Therefore,
\begin{align*}
    \mathcal L_{\pi_x}(|U-V|)
    = d_\#\pi_x
    =
    p\delta_L+d_N\lam\vert_{[L-v,L]}.
\end{align*}
\end{proof}

\begin{proposition}[Spearman's \(\varrho\) and Spearman's footrule for \(C_x\)]\label{prop_C_x}
Let \(x\in [-\frac 1 2,1)\).
For the copula \(C_x\) in \eqref{def_C_x}, we have
    \begin{align*}
    \phi(C_x) &= x \qquad \text{and}\\
         \varrho(C_x) &= \begin{cases}
             1-6L(2m-L) - \frac{4(L-m)^{3/2}}{\sqrt{N (N+1)}}, & \text{for } x\in I_N^L,\\
             1-6R(2m-R) - \frac{4(m-R)^{3/2}}{\sqrt{N (N+1)}}, & \text{for } x\in I_N^R.
         \end{cases}
    \end{align*}
    At \(x=1\), one has \(\phi(C_1)=\varrho(C_1)=1\).
\end{proposition}

\begin{proof}
For \(x=1\), Definition~\ref{def_Cx} gives the comonotonicity copula, under which \(U=V\) almost surely; in this case, the assertion follows from Lemma~\ref{lem:moments}.
Let now \(x\in[-\tfrac12,1)\).
The joint distribution of \((U,V)\sim C_x\) is by definition of the copula \(C_x\) in \eqref{def_C_x} the coupling \(\pi_x=\nu_x+S_\#\nu_x\).
To determine \(\phi(C_x)\) and \(\varrho(C_x)\), we need to determine the first and second moment of \(|U-V|\).

\medskip
\noindent
\emph{Case 1: \(x\in I_N^R\).}
For the first moment, we obtain from Lemma \ref{lem_U-V} that
\begin{align*}
    \E|U-V|
    &=
    pR+d_N\int_R^{R+v}r\,\de r
    =
    (1-d_Nv)R
    +
    d_N\left(Rv+\frac{v^2}{2}\right)
    =
    R+\frac{d_Nv^2}{2}
    =
    m,
\end{align*}
where the last equality follows from the definition of \(v(x)\) on \(I_N^R\).
Hence,
\[
    \phi(C_x)
    =
    1-3\E|U-V|
    =
    1-3m
    =
    x.
\]
For the second moment, we obtain from Lemma \ref{lem_U-V} that
\begin{align*}
    \E(U-V)^2
    &=
    pR^2+d_N\int_R^{R+v}r^2\,\de r\\
    &=
    (1-d_Nv)R^2
    +
    d_N\left(
        R^2v+Rv^2+\frac{v^3}{3}
    \right)\\
    &=
    R^2+d_NR v^2+\frac{d_Nv^3}{3}
    =
    R(2m-R)+\frac{d_Nv^3}{3}.
\end{align*}
Since \(v=\sqrt{\frac{2(m-R)}{d_N}} \) and \(d_N=2N(N+1)\) (see \eqref{def:vptheta} and \eqref{eq:dN}), we have
\[
    \frac{d_Nv^3}{3}
    =
    \frac{2(m-R)^{3/2}}
         {3\sqrt{N(N+1)}}.
\]
Using
\(
    \varrho(C_x)=1-6\,\E(U-V)^2,
\)
we conclude that
\[
    \varrho(C_x)
    =
    1-6R(2m-R)
    -
    \frac{4(m-R)^{3/2}}{\sqrt{N(N+1)}}.
\]

\medskip
\noindent
\emph{Case 2: \(x\in I_N^L\).}
For the first moment, we obtain from Lemma \ref{lem_U-V} that
\begin{align*}
    \E|U-V|
    &=
    pL+d_N\int_{L-v}^{L}r\,\de r
    =
    (1-d_Nv)L
    +
    d_N\left(Lv-\frac{v^2}{2}\right)
    =
    L-\frac{d_Nv^2}{2}
    =
    m.
\end{align*}
Hence,
\[
    \phi(C_x)=1-3m=x.
\]
For the second moment, we again use Lemma \ref{lem_U-V} and obtain
\begin{align*}
    \E(U-V)^2
    &=
    pL^2+d_N\int_{L-v}^{L}r^2\,\de r\\
    &=
    (1-d_Nv)L^2
    +
    d_N\left(
        L^2v-Lv^2+\frac{v^3}{3}
    \right)\\
    &=
    L^2-d_NL v^2+\frac{d_Nv^3}{3}
    =
    L(2m-L)+\frac{d_Nv^3}{3}.
\end{align*}
Since \(v=\sqrt{\frac{2(L-m)}{d_N}}, \) it follows that
\[
    \frac{d_Nv^3}{3}
    =
    \frac{2(L-m)^{3/2}}
         {3\sqrt{N(N+1)}}.
\]
This yields
\[
    \varrho(C_x)
    =
    1-6L(2m-L)
    -
    \frac{4(L-m)^{3/2}}{\sqrt{N(N+1)}},
\]
which proves the assertion.

\end{proof}

\section{Dual potential and optimality}\label{sec:dual}

Up to this point, for every \(x\in[-\frac12,1)\), we have constructed a copula \(C_x\) that satisfies \(\phi(C_x)=x\), and we have derived an explicit expression for \(\varrho(C_x)\).
By definition of \(\overline{\varrho}(x)\), we necessarily have
\[
\varrho(C_x)\leq\overline{\varrho}(x).
\]
To prove equality, it remains to show that every copula \(C\) satisfying \(\phi(C)=x\) also fulfills
\[
\varrho(C)\leq\varrho(C_x).
\]
In this section, we establish the upper bound by constructing a dual feasible potential whose contact set contains the support of \(\nu_x\).
In Corollary \ref{cor_feas}, we established the contact-set optimality criterion for the centered dual potential.
We now construct, for each \(x\), an explicit dual potential matching the candidate measure \(\nu_x\).
It is more convenient to work with the non-centered potential \(g_x\), satisfying
\begin{align}\label{eq:noncentered-feasibility}
    g_x(a) + g_x(b) + \theta(b-a) \leq (b-a)^2 \quad \text{for all } (a,b)\in T;
\end{align}
recall \(T=\{(a,b)\in[0,1]^2:a\leq b\}\).
The corresponding centered potential is
\begin{align*}
    f_x = g_x + \theta m/2.
\end{align*}
Hence, \((f_x,\theta)\) is feasible for the centered dual problem if and only if \(g_x\) satisfies \eqref{eq:noncentered-feasibility}.

Now, fix \(x\in[-\frac12,1)\), choose \(N\in\mathbb N\) such that \(x\in I_N\), and recall the parameters
\[
    m=m(x),\qquad
    v=v(x),\qquad
    p=p(x),\qquad
    \theta=\theta(x)
\]
defined in \eqref{eq:mean-gap} and \eqref{def:vptheta}.
The definitions imply \(0\leq v\leq L-R\).
If \(x\in I_N^R\), then \(\theta=2R+v>v\).
If \(x\in I_N^L\), then \(\theta-v=2(L-v)\geq2R>0\).
Thus, in both cases, \(\theta>v\geq0\), and in particular \(\theta>0\).
Define the modified cost function
\begin{align}\label{def:k_x}
    k_x(r):=r^2-\theta r,
    \qquad r\in \R.
\end{align}
Then \eqref{eq:noncentered-feasibility} is equivalent to
\begin{align}\label{eq:gx-kx-feasibility}
    g_x(a)+g_x(b)
    \leq k_x(b-a)
    \qquad\text{for all }(a,b)\in T.
\end{align}
The choice of \(\theta\) is adapted to the construction of \(\nu_x\).
Specifically, we have
\begin{align}\label{prop_kr}
    k_x(\theta-r)=k_x(r)
    \qquad\text{for all }r\in\mathbb R.
\end{align}
We shall construct \(g_x\) such that
\begin{align}\label{eq:contact-support-goal}
    g_x(a)+g_x(b)=k_x(b-a)
    \qquad
    \nu_x\text{-almost surely},
\end{align}
while \eqref{eq:gx-kx-feasibility} holds globally on \(T\).
Then, by Corollary~\ref{cor_feas}, the two properties \eqref{eq:gx-kx-feasibility} and \eqref{eq:contact-support-goal} imply that \(\nu_x\) is optimal.
Recall \(v= v(x)\), \(p=p(x)\), and \(\theta = \theta(x)\) in \eqref{def:vptheta}.

\begin{definition}[Optimal dual potential]
\label{def:dual-potential}
For \(x\in[-\frac12,1)\), define \(g_x\) first on one period \([0,\theta]\) and then extend it \(\theta\)-periodically to \(\mathbb R\); we use the same symbol \(g_x\) for its restriction to \([0,1]\).
If \(x\in I_N^R\), put
\(
    c_x^R\coloneqq\frac{k_x(R)-vpR}{2}
\)
and define
\begin{align}\label{def:dual-potential-value-right}
    g_x(u)
    :=
    \begin{cases}
        c_x^R+vu,
        &0\leq u<pR,\\[1mm]
        \displaystyle
        c_x^R+vpR+v(u-pR)-\frac{(u-pR)^2}{N},
        &pR\leq u<R,\\[3mm]
        c_x^R+vpR-v(u-R),
        &R\leq u<R+pR,\\[1mm]
        \displaystyle
        c_x^R-v(u-R-pR)+\frac{(u-R-pR)^2}{N+1},
        &R+pR\leq u\leq\theta.
    \end{cases}
\end{align}
If \(x\in I_N^L\), put
\(
    c_x^L\coloneqq\frac{k_x(L)+vpL}{2}
\)
and define
\begin{align}\label{def:dual-potential-value-left}
    g_x(u)
    :=
    \begin{cases}
        \displaystyle
        c_x^L+vu-\frac{u^2}{N},
        &0\leq u\leq Nv,\\[3mm]
        c_x^L-v(u-Nv),
        &Nv\leq u\leq L-v,\\[1mm]
        \displaystyle
        c_x^L-vpL-v(u-L+v)+\frac{(u-L+v)^2}{N+1},
        &L-v\leq u\leq L+Nv,\\[3mm]
        c_x^L-vpL+v(u-L-Nv),
        &L+Nv\leq u\leq\theta.
    \end{cases}
\end{align}
\end{definition}

\begin{remark}
    The choice of \(c_x^R\) is equivalently characterized by
\begin{align}\label{eq:normalization-potential-right}
    g_x(0)+g_x(R)=k_x(R).
\end{align}
and \(c_x^L\) is determined by
\begin{align}\label{eq:normalization-potential-left}
    g_x(Nv)+g_x(Nv+L)=k_x(L).
\end{align}
\end{remark}

Some properties of \(g_x\) are given in the following Lemma; see also Figure \ref{fig:construction}.

\begin{lemma}
\label{lem:dual-potential-well-defined}
The function \(g_x\) in Definition~\ref{def:dual-potential} is \(\theta\)-periodic and continuously differentiable with derivative
\begin{align}\label{def:dual-potential-right}
    g_x'(u)
    =
    \begin{cases}
        v,
        & 0\leq u< pR,\\[1mm]
        \displaystyle
        v-\frac{2}{N}(u-pR),
        & pR\leq u< R,\\[3mm]
        -v,
        & R\leq u< R+pR,\\[1mm]
        \displaystyle
        -v+\frac{2}{N+1}(u-R-pR),
        & R+pR\leq u\leq\theta.
    \end{cases}
\end{align}
on \([0,\theta]\) for \(x\in I_N^R\), and
\begin{align}\label{def:dual-potential-left}
    g_x'(u)
    =
    \begin{cases}
        \displaystyle
        v-\frac{2}{N}u,
        & 0\leq u< Nv,\\[3mm]
        -v,
        & Nv\leq u < L-v,\\[1mm]
        \displaystyle
        -v+\frac{2}{N+1}(u-L+v),
        & L-v\leq u < L+Nv,\\[3mm]
        v,
        & L+Nv\leq u\leq\theta.
    \end{cases}
\end{align}
on \([0,\theta]\) for \(x\in I_N^L\).
\end{lemma}

\begin{proof}
For \(x\in I_N^R\), recall that \(pR+Nv=R\) and \(\theta=2R+v\) by \eqref{eq:right-identities-proof} and \eqref{def:vptheta}.
Straightforward calculations show that \(g_x'\) is continuous and piecewise linear with the expressions in \eqref{def:dual-potential-right} and \eqref{def:dual-potential-left}.
Thus, \(g_x\) is continuously differentiable.
To show that \(g_x\) is \(\theta\)-periodic, consider
\begin{align*}
    \int_0^\theta g_x'(u)\,\de u
    &=
    \int_0^{pR}v\,\de u
    +
    \int_{pR}^{R}
    \left(
        v-\frac{2}{N}(u-pR)
    \right)\de u\\
    &\quad
    -
    \int_R^{R+pR}v\,\de u
    +
    \int_{R+pR}^{\theta}
    \left(
        -v+\frac{2}{N+1}(u-R-pR)
    \right)\de u.
\end{align*}
The second integral vanishes because its integrand decreases linearly from \(v\) to \(-v\) on an interval of length \(Nv\).
Similarly, the fourth integral vanishes because its integrand increases linearly from \(-v\) to \(v\) on an interval of length \((N+1)v\).
The first and third integrals cancel.
Hence,
\(
    \int_0^\theta g_x'(u)\de u=0.
\)

For \(x\in I_N^L\), recall that \(Nv+pL=L-v\) and \(\theta=2L-v\) from \eqref{eq:left-identities-proof} and \eqref{def:vptheta}.
Again, straightforward calculations show that \(g_x'\) is continuous and piecewise linear so that \(g_x\) is continuously differentiable.
To show that \(g_x\) is \(\theta\)-periodic, consider
\begin{align*}
    \int_0^\theta g_x'(u)\,\de u
    &=
    \int_0^{Nv}
    \left(
        v-\frac{2}{N}u
    \right)\de u
    -
    \int_{Nv}^{L-v}v\,\de u\\
    &\quad+
    \int_{L-v}^{L+Nv}
    \left(
        -v+\frac{2}{N+1}(u-L+v)
    \right)\de u
    +
    \int_{L+Nv}^{\theta}v\,\de u.
\end{align*}
The first and third integrals vanish because their integrands are linear and have endpoint values \(v,-v\) and \(-v,v\), respectively.
The second and fourth intervals both have length
\[
    L-v-Nv
    =
    \theta-L-Nv
    =
    pL.
\]
Hence, the second and fourth integrals cancel, and therefore
\(
    \int_0^\theta g_x'(u)\de u=0.
\)
\end{proof}

In the following result, we verify that \(\nu_x\) in \eqref{def_nu_x} is concentrated on the set \[\{(a,b)\in T:g_x(a)+g_x(b)=k_x(b-a)\}.\] For \(f_x=g_x+\theta m/2\), this is equivalent to concentration on the centered contact set \(\Gamma_{f_x,\theta}\).

\begin{lemma}[Equality on the support of \(\nu_x\)]
\label{lem:equality-on-support}
For every \(x\in[-\frac12,1)\), we have
\begin{align}\label{eq:equality-on-support}
    g_x(a)+g_x(b)
    =
    k_x(b-a)
    \qquad
    \text{for }\nu_x\text{-almost every }(a,b)\in T.
\end{align}
\end{lemma}

\begin{proof}
Recall the definitions of the mappings \(A_k^{L/R}\), \(B_k^{L/R}\), and \(C_k^{L/R}\).
By definition, the measure \(\nu_x\) is concentrated on the images of these mappings.
We verify \eqref{eq:equality-on-support} separately on each of the corresponding families of line segments.

By Lemma~\ref{lem:dual-potential-well-defined}, the function \(g_x\) is \(\theta\)-periodic.
Moreover, every mapping with index \(k\) is obtained from the corresponding mapping with index \(0\) by translating both coordinates by \(k\theta\).
For instance,
\[
    A_k^{L/R}(s)
    =
    A_0^{L/R}(s)+(k\theta,k\theta),
\]
and analogously for \(B_k^{L/R}\) and \(C_k^{L/R}\).
Since a common translation does not change the coordinate difference and
\[
    g_x(u+k\theta)=g_x(u)
    \qquad\text{for all }u\in\mathbb R,
\]
it is sufficient to consider the case \(k=0\).

\medskip
\noindent
\emph{Case 1: Let \(x\in I_N^R\).}
For the \(A_0^R\)-segments, define
\[
    H_A^R(s)
    :=
    g_x(s)+g_x(s+R)-k_x(R),
    \qquad
    0\leq s\leq pR.
\]
Since \(s\in[0,pR]\) and \(s+R\in[R,R+pR]\), it follows from \eqref{def:dual-potential-right} that \(g_x'(s)=v\) and \(g_x'(s+R)=-v\).
Consequently,
\[
    (H_A^R)'(s)
    =
    g_x'(s)+g_x'(s+R)
    =
    0.
\]
Thus, \(H_A^R\) is constant.
By the normalization condition \eqref{eq:normalization-potential-right}, \(H_A^R(0) = g_x(0)+g_x(R)-k_x(R) = 0. \) Hence,
\begin{align}\label{eq_H_A00}
    g_x(s)+g_x(s+R)=k_x(R)
\end{align}
for every \(s\in[0,pR]\).
Since \(A_0^R(s)=(s,s+R)\) has coordinate difference \(R\), this proves \eqref{eq:equality-on-support} on the \(A_k^R\)-segments.

Next, consider the \(B_0^R\)-segments.
Set
\[
    a(t):=pR+t,
    \qquad
    r(t):=\ell_R(t)=R+\frac{t}{N},
    \qquad
    0\leq t\leq Nv,
\]
and define
\[
    H_B^R(t)
    :=
    g_x(a(t))
    +
    g_x(a(t)+r(t))
    -
    k_x(r(t)).
\]
Since \(\theta=2R+v\) by \eqref{def:vptheta}, we have
\[
    k_x'(r(t))
    =
    2r(t)-\theta
    =
    -v+\frac{2t}{N}.
\]
Moreover, by \eqref{def:dual-potential-right},
\[
    g_x'(a(t))
    =
    v-\frac{2t}{N}
    =
    -k_x'(r(t)).
\]
Furthermore, \(a(t)+r(t) = R+pR+\frac{N+1}{N}t, \) and hence
\[
    g_x'(a(t)+r(t))
    =
    -v+\frac{2t}{N}
    =
    k_x'(r(t)).
\]
Using \(a'(t)=1\) and \(r'(t)=\frac1N\), we obtain
\begin{align*}
    (H_B^R)'(t)
    &=
    g_x'(a(t))
    +
    g_x'(a(t)+r(t))
    \left(1+\frac1N\right)
    -
    k_x'(r(t))\frac1N
    \\
    &=
    -k_x'(r(t))
    +
    k_x'(r(t))
    \left(1+\frac1N\right)
    -
    k_x'(r(t))\frac1N
    =0.
\end{align*}
Thus, \(H_B^R\) is constant.
At \(t=0\), we have \(a(0)=pR\) and \(r(0)=R\), and therefore
\[
    B_0^R(0)
    =
    (pR,pR+R)
    =
    A_0^R(pR).
\]
The equality \eqref{eq_H_A00} already proved on the \(A_0^R\)-segment consequently gives
\[
    H_B^R(0)
    =
    g_x(pR)+g_x(pR+R)-k_x(R)
    =
    0.
\]
Hence,
\begin{align}\label{eq_B_0R}
    g_x(a(t))+g_x(a(t)+r(t))
    =
    k_x(r(t))
\end{align}
for every \(t\in[0,Nv]\).
Since
\(
    B_0^R(t)
    =
    \bigl(a(t),a(t)+r(t)\bigr)
\)
and the coordinate difference of this pair is \(r(t)\), this proves \eqref{eq:equality-on-support} on the \(B_k^R\)-segments.

Finally, the pair generated by \(C_0^R\) is
\(
    C_0^R(t)
    =
    \bigl(a(t)+r(t),\,\theta+a(t)\bigr).
\)
By the \(\theta\)-periodicity of \(g_x\),
\[
    g_x(\theta+a(t))=g_x(a(t)).
\]
Using the equality \eqref{eq_B_0R} already established on the \(B_0^R\)-segment, we therefore obtain
\begin{align*}
    g_x(a(t)+r(t))+g_x(\theta+a(t))
    &=
    g_x(a(t)+r(t))+g_x(a(t))
    =
    k_x(r(t)).
\end{align*}
The coordinate difference of the \(C_0^R\)-pair is
\[
    \bigl(\theta+a(t)\bigr)
    -
    \bigl(a(t)+r(t)\bigr)
    =
    \theta-r(t).
\]
Since
\(
    k_x(\theta-r)=k_x(r),
\) by \eqref{prop_kr},
it follows that
\[
    g_x(a(t)+r(t))+g_x(\theta+a(t))
    =
    k_x(\theta-r(t)).
\]
This proves \eqref{eq:equality-on-support} on the
\(C_k^R\)-segments.

\medskip
\noindent
\emph{Case 2: Let \(x\in I_N^L\).}
For the \(A_0^L\)-segments, define
\[
    H_A^L(s)
    :=
    g_x(Nv+s)+g_x(Nv+s+L)-k_x(L),
    \qquad
    0\leq s\leq pL.
\]
Since \(Nv+pL=L-v\) by \eqref{eq:left-identities-proof},
we have \(Nv+s\in[Nv,L-v]\) and \(Nv+s+L\in[L+Nv,\theta]\).
It follows from \eqref{def:dual-potential-left} that
\(g_x'(Nv+s)=-v\) and \(g_x'(Nv+s+L)=v\).
Consequently,
\[
    (H_A^L)'(s)
    =
    g_x'(Nv+s)+g_x'(Nv+s+L)
    =
    0.
\]
Thus, \(H_A^L\) is constant. By the normalization condition
\eqref{eq:normalization-potential-left},
\[
    H_A^L(0)
    =
    g_x(Nv)+g_x(Nv+L)-k_x(L)
    =
    0.
\]
Hence,
\begin{align}\label{eq_H_A0}
    g_x(Nv+s)+g_x(Nv+s+L)
    =
    k_x(L)
\end{align}
for every \(s\in[0,pL]\). Since
\(A_0^L(s)
    =
    \bigl(Nv+s,Nv+s+L\bigr)
\)
has coordinate difference \(L\), this proves \eqref{eq:equality-on-support} on the \(A_k^L\)-segments.

Next, consider the \(B_0^L\)-segments.
Set
\[
    a(t):=t,
    \qquad
    r(t):=\ell_L(t)=L-v+\frac{t}{N},
    \qquad
    0\leq t\leq Nv,
\]
and define
\[
    H_B^L(t)
    :=
    g_x(a(t))
    +
    g_x(a(t)+r(t))
    -
    k_x(r(t)).
\]
Since \(\theta=2L-v\), we have
\[
    k_x'(r(t))
    =
    2r(t)-\theta
    =
    -v+\frac{2t}{N}.
\]
Moreover, by \eqref{def:dual-potential-left},
\[
    g_x'(a(t))
    =
    v-\frac{2t}{N}
    =
    -k_x'(r(t)).
\]
Furthermore, \(a(t)+r(t) = L-v+\frac{N+1}{N}t, \) and hence
\[
    g_x'(a(t)+r(t))
    =
    -v+\frac{2t}{N}
    =
    k_x'(r(t)).
\]
Using \(a'(t)=1\) and \(r'(t)=\frac1N\), we obtain
\begin{align*}
    (H_B^L)'(t)
    &=
    g_x'(a(t))
    +
    g_x'(a(t)+r(t))
    \left(1+\frac1N\right)
    -
    k_x'(r(t))\frac1N
    \\
    &=
    -k_x'(r(t))
    +
    k_x'(r(t))
    \left(1+\frac1N\right)
    -
    k_x'(r(t))\frac1N
    =0.
\end{align*}
Thus, \(H_B^L\) is constant.
At \(t=Nv\), we have \(a(Nv)=Nv\) and \(r(Nv)=L\) and therefore
\[
    B_0^L(Nv)
    =
    (Nv,Nv+L)
    =
    A_0^L(0).
\]
The equality \eqref{eq_H_A0} already proved on the \(A_0^L\)-segment consequently gives
\[
    H_B^L(Nv)
    =
    g_x(Nv)+g_x(Nv+L)-k_x(L)
    =
    0.
\]
Hence,
\begin{align}\label{eq:H_BL}
    g_x(a(t))+g_x(a(t)+r(t))
    =
    k_x(r(t))
\end{align}
for every \(t\in[0,Nv]\).
Since \(B_0^L(t) = \bigl(a(t),a(t)+r(t)\bigr) \) and the coordinate difference of this pair is \(r(t)\), this proves \eqref{eq:equality-on-support} on the \(B_k^L\)-segments.

Finally, the pair generated by \(C_0^L\) is
\[
    C_0^L(t)
    =
    \bigl(a(t)+r(t),\,\theta+a(t)\bigr).
\]
By the \(\theta\)-periodicity of \(g_x\),
\[
    g_x(\theta+a(t))=g_x(a(t)).
\]
Using the equality in \eqref{eq:H_BL} already established on the \(B_0^L\)-segment, we therefore obtain
\begin{align*}
    g_x(a(t)+r(t))+g_x(\theta+a(t))
    &=
    g_x(a(t)+r(t))+g_x(a(t))
    =
    k_x(r(t)).
\end{align*}
The coordinate difference of the \(C_0^L\)-pair is
\[
    \bigl(\theta+a(t)\bigr)
    -
    \bigl(a(t)+r(t)\bigr)
    =
    \theta-r(t).
\]
Since \(k_x(\theta-r)=k_x(r)\) by \eqref{prop_kr}, it follows that
\[
    g_x(a(t)+r(t))+g_x(\theta+a(t))
    =
    k_x(\theta-r(t)),
\]
which proves \eqref{eq:equality-on-support} on the \(C_k^L\)-segments.

Since \(\nu_x\) is concentrated on the images of the mappings \(A_k^{L/R}\), \(B_k^{L/R}\), and \(C_k^{L/R}\), the equality \eqref{eq:equality-on-support} holds \(\nu_x\)-almost surely.
\end{proof}

Next we show dual feasibility.
\begin{lemma}[Global dual feasibility]
\label{lem:global-dual-feasibility}
For every \(x\in[-\frac12,1)\), the candidate potential \(g_x\) satisfies
\begin{align}\label{eq:global-dual-feasibility}
    g_x(a)+g_x(b)
    \leq
    k_x(b-a)
    \qquad
    \text{for all }(a,b)\in T.
\end{align}
Consequently, the centered potential
\begin{align}\label{def_fxvgx}
    f_x(u):=g_x(u)+\frac{\theta m}{2},
    \qquad u\in[0,1],
\end{align}
satisfies \(f_x(a)+f_x(b) +\theta\bigl((b-a)-m\bigr) \leq (b-a)^2 \) for all \((a,b)\in T\), and hence \((f_x,\theta)\) is feasible.
\end{lemma}

\begin{proof}
Extend \(g_x\) and \(g_x'\) \(\theta\)-periodically to \(\mathbb R\).
For \(u\in\mathbb R\) and \(r\geq0\), define
\begin{align}\label{def_Gxar}
    G_x(u,r)
    :=
    k_x(r)-g_x(u)-g_x(u+r).
\end{align} 
It suffices to prove that
\begin{align}\label{eq_assG}
    G_x(u,r)\geq0
    \qquad
    \text{for all }u\in\mathbb R,\ r\geq0.
\end{align}
We first reduce the problem to \(r\in[0,\theta]\).
Write \(r=r_0+n\theta\) and \(r_0\in[0,\theta)\) for \(n\in\mathbb N_0\).
By the \(\theta\)-periodicity of \(g_x\) (Lemma \ref{lem:dual-potential-well-defined}), we have \(g_x(u+r)=g_x(u+r_0)\).
Moreover, by definition of \(k_x\) in \eqref{def:k_x},
\begin{align}\label{eq:G-periodic-increment}
G_x(u,r_0+n\theta)-G_x(u,r_0) 
    & =k_x(r_0+n\theta)-k_x(r_0)\\
    &=
    (r_0+n\theta)^2
    -\theta(r_0+n\theta)
    -r_0^2+\theta r_0
    =
    n\theta
    \bigl(2r_0+(n-1)\theta\bigr)
    \geq0. \notag
\end{align}
Consequently, \(G_x(u,r)\geq G_x(u,r_0)\) and it is enough to establish \eqref{eq_assG} for \(0\leq r\leq\theta\).

Fix \(u\in\mathbb R\).
Since \(g_x'\) is continuous and piecewise linear, the function \(r\longmapsto G_x(u,r)\) is continuously differentiable and piecewise twice differentiable.
Its first derivative is
\begin{align}\label{eq_delGr}
    \partial_rG_x(u,r)
    =
    2r-\theta-g_x'(u+r).
\end{align}
The slopes of \(g_x'\) belong to the set \(\left\{ 0,-\frac{2}{N},\frac{2}{N+1} \right\}\).
Hence, wherever the second derivative exists,
\[
    \partial_r^2G_x(u,r)
    =
    2-g_x''(u+r)
    \geq
    2-\frac{2}{N+1}
    >0.
\]
It follows that \(r\mapsto G_x(u,r)\) is strictly convex on \([0,\theta]\).
Define
\begin{align}\label{eq:def-r-star}
    r_*(u)
    :=
    \frac{\theta-g_x'(u)}{2}.
\end{align}
Since, as a consequence of the representation of \(g_x'\) in \eqref{def:dual-potential-right} and \eqref{def:dual-potential-left}, using the definition of \(\theta\) in \eqref{def:vptheta}, one can show that
\begin{align}\label{ineq_gp}
    |g_x'(u)|\leq v<\theta
\end{align}
which yields
\(
    r_*(u)\in(0,\theta).
\)

We next prove the identity
\begin{align}\label{eq:opposite-slope-identity}
    g_x'\bigl(u+r_*(u)\bigr)
    =
    -g_x'(u),
\end{align}
Both \(g_x'\) and, by \eqref{eq:def-r-star}, \(r_*\) are \(\theta\)-periodic.
Given \(u\in\mathbb R\), choose \(u_0\in[0,\theta)\) with \(u-u_0\in\theta\mathbb Z\).
Then \(g_x'(u)=g_x'(u_0)\), \(r_*(u)=r_*(u_0)\), and
\[
    g_x'\bigl(u+r_*(u)\bigr)
    =g_x'\bigl(u_0+r_*(u_0)\bigr).
\]
It therefore suffices to prove \eqref{eq:opposite-slope-identity} for \(u\in[0,\theta)\), which is done by the following case split.

\medskip
\noindent
\emph{Case 1: \(x\in I_N^R\).}
Recall from \eqref{def:vptheta} and \eqref{eq:right-identities-proof} that
\begin{align*}
    \theta=2R+v \quad \text{and} \quad R-pR=Nv, \quad \text{so} \quad \theta-R-pR=(N+1)v.
\end{align*}
If \(u\in[0,pR]\), then \(g_x'(u)=v\), and hence \(r_*(u)=R\).
By \eqref{def:dual-potential-right}, this gives
\begin{align*}
    g_x'(u + r_*(u)) = -v = -g_x'(u).
\end{align*}
If \(u=pR+s\), \(0\leq s\leq Nv\), then by\eqref{def:dual-potential-right} and \eqref{eq:def-r-star} \(g_x'(u)=v-\frac{2s}{N}\) and \(r_*(u) = R+\frac{s}{N}. \) Thus,
\[
    u+r_*(u)
    =
    R+pR+\frac{N+1}{N}s.
\]
This point belongs to \([R+pR,\theta]\), and therefore
\begin{align*}
    g_x'(u+r_*(u))
    &=
    -v
    +
    \frac{2}{N+1}
    \left(
        \frac{N+1}{N}s
    \right)
    =
    -v+\frac{2s}{N}
    =
    -g_x'(u).
\end{align*}
If \(u=R+s\), \(0\leq s\leq pR\), then \(g_x'(u)=-v\), and hence \(r_*(u)=R+v\).
Since \(u+r_*(u) = \theta+s, \) periodicity gives
\[
    g_x'(u+r_*(u))
    =
    g_x'(s)
    =
    v
    =
    -g_x'(u).
\]
Finally, if \(u=R+pR+s\), \(0\leq s\leq(N+1)v\).
Then \(g_x'(u) = -v+\frac{2s}{N+1} \) and \(r_*(u) = R+v-\frac{s}{N+1} \).
Consequently,
\[
    u+r_*(u)
    =
    \theta+pR+\frac{N}{N+1}s.
\]
By periodicity of \(g_x'\),
\[
    g_x'(u+r_*(u))
    =
    g_x'\left(
        pR+\frac{N}{N+1}s
    \right).
\]
Since \(0 \leq \frac{N}{N+1}s \leq Nv, \) the argument above belongs to \([pR,R]\), and thus
\begin{align*}
    g_x'(u+r_*(u))
    &=
    v
    -
    \frac{2}{N}
    \left(
        \frac{N}{N+1}s
    \right)
    =
    v-\frac{2s}{N+1}
    =
    -g_x'(u).
\end{align*}
This proves \eqref{eq:opposite-slope-identity} in the right-hand case.

\medskip
\noindent
\emph{Case 2: \(x\in I_N^L\).}
Recall that
\begin{align*}
    \theta=2L-v\ \quad \text{and} \quad L-v-Nv=pL, \quad \text{so}\quad \theta-L-Nv=pL.
\end{align*}
If \(u=s,\) \(0\leq s\leq Nv\), then, by \eqref{def:dual-potential-left} and  \eqref{eq:def-r-star}, \(g_x'(u) = v-\frac{2s}{N} \) and \(r_*(u) = L-v+\frac{s}{N} \).
Hence,
\[
    u+r_*(u)
    =
    L-v+\frac{N+1}{N}s,
\]
which belongs to \([L-v,L+Nv]\).
Therefore,
\begin{align*}
    g_x'(u+r_*(u))
    =
    -v
    +
    \frac{2}{N+1}
    \left(
        \frac{N+1}{N}s
    \right)
    =
    -v+\frac{2s}{N}
    =
    -g_x'(u).
\end{align*}
If \(u=Nv+s\), \(0\leq s\leq pL\), then \(g_x'(u)=-v\), and hence \(r_*(u)=L\).
Since
\[
    u+r_*(u)
    =
    L+Nv+s
    \in[L+Nv,\theta],
\]
we obtain
\[
    g_x'(u+r_*(u))
    =
    v
    =
    -g_x'(u).
\]
If \(u=L-v+s\), \(0\leq s\leq(N+1)v\), then \(g_x'(u) = -v+\frac{2s}{N+1} \) and \(r_*(u) = L-\frac{s}{N+1} \).
Thus,
\[
    u+r_*(u)
    =
    \theta+\frac{N}{N+1}s.
\]
By periodicity of \(g_x'\),
\[
    g_x'(u+r_*(u))
    =
    g_x'\left(
        \frac{N}{N+1}s
    \right).
\]
Since \(0 \leq \frac{N}{N+1}s \leq Nv, \) we obtain
\begin{align*}
    g_x'(u+r_*(u))
    =
    v
    -
    \frac{2}{N}
    \left(
        \frac{N}{N+1}s
    \right)
    =
    v-\frac{2s}{N+1}
    =
    -g_x'(u).
\end{align*}
Finally, if \(u=L+Nv+s\), \(0\leq s\leq pL\), then \(g_x'(u)=v\), and hence
\(
    r_*(u)=L-v
\).
Since
\[
    u+r_*(u)
    =
    \theta+Nv+s,
\]
periodicity gives
\[
    g_x'(u+r_*(u))
    =
    g_x'(Nv+s)
    =
    -v
    =
    -g_x'(u).
\]
Thus, \eqref{eq:opposite-slope-identity} also holds in the left-hand
case.

Now, using \eqref{eq_delGr}, \eqref{eq:def-r-star}, and \eqref{eq:opposite-slope-identity}, we obtain
\begin{align*}
    \partial_rG_x(u,r_*(u))
    =
    2r_*(u)-\theta
    -
    g_x'\bigl(u+r_*(u)\bigr)
    =
    -g_x'(u)+g_x'(u)
    =0.
\end{align*}
By strict convexity, \(r_*(u)\) is the unique minimizer of
\(r\mapsto G_x(u,r)\) on \([0,\theta]\).

It remains to determine the minimum value. Define
\[
    H_x(u)
    :=
    G_x(u,r_*(u)).
\]
Since \(g_x'\) is continuous and piecewise linear, the function
\(r_*\) is Lipschitz continuous. Consequently, \(H_x\) is locally
Lipschitz and hence absolutely continuous on compact intervals.
At every point at which the relevant derivatives exist, the chain
rule gives
\begin{align*}
    H_x'(u)
    =
    k_x'(r_*(u))r_*'(u)
    -
    g_x'(u)
    -
    g_x'\bigl(u+r_*(u)\bigr)
    \bigl(1+r_*'(u)\bigr).
\end{align*}
By the definition of \(k_x\) and \(r_*\) in \eqref{def:k_x} and \eqref{eq:def-r-star},
\[
    k_x'(r_*(u))
    =
    2r_*(u)-\theta
    =
    -g_x'(u).
\]
Together with \eqref{eq:opposite-slope-identity}, this yields
\begin{align*}
    H_x'(u)
    &=
    -g_x'(u)r_*'(u)
    -
    g_x'(u)
    +
    g_x'(u)\bigl(1+r_*'(u)\bigr)
    =0
\end{align*}
almost everywhere. Hence, \(H_x\) is constant.

If \(x\in I_N^R\), then
\(g_x'(0)=v\) and \(r_*(0)
    =
    \frac{\theta-v}{2}
    =
    R\).
By the normalization condition
\eqref{eq:normalization-potential-right},
\[H_x(0)
    =
    k_x(R)-g_x(0)-g_x(R)
    =
    0.
\]
If \(x\in I_N^L\), then
\(g_x'(Nv)=-v\) and \(r_*(Nv)
    =
    \frac{\theta+v}{2}
    =
    L\).
By the normalization condition
\eqref{eq:normalization-potential-left},
\[
    H_x(Nv)
    =
    k_x(L)-g_x(Nv)-g_x(Nv+L)
    =
    0.
\]
Thus, in both cases,
\(H_x(u)=0\) for all \(u\in\mathbb R\).
Since \(r_*(u)\) is the minimizer of \(r\mapsto G_x(u,r)\), we obtain
\begin{align}\label{eq:G-unique-zero}
    G_x(u,r)
    \geq
    G_x(u,r_*(u))
    =
    0
\end{align}
for every \(u\in\mathbb R\) and every \(r\in[0,\theta]\). By the
initial reduction, the same inequality holds for every \(r\geq0\).

Finally, let \((a,b)\in T\). Setting
\(u=a\) and \(r=b-a\),
we obtain
\[
    0
    \leq
    G_x(a,b-a)
    =
    k_x(b-a)-g_x(a)-g_x(b),
\]
which proves \eqref{eq:global-dual-feasibility}.

Since \(k_x(b-a)
    =
    (b-a)^2-\theta(b-a),
\)
we have
\[
    g_x(a)+g_x(b)+\theta(b-a)
    \leq
    (b-a)^2.
\]
Since \(f_x(u)=g_x(u)+\frac{\theta m}{2}\) by \eqref{def_fxvgx}, this is equivalent to
\[
    f_x(a)+f_x(b)
    +\theta\bigl((b-a)-m\bigr)
    \leq
    (b-a)^2.
\]
Hence, \((f_x,\theta)\) is feasible.
\end{proof}

For proving uniqueness of the optimal copula \(C_x\) in Proposition~\ref{prop:rho-phi-Cx}, we use the following elementary lemma.
It states that a finite signed measure cannot be invariant, up to sign, under a map that shifts every point by a strictly positive amount, unless the measure is identically zero.

\begin{lemma}
\label{lem:no-upward-circulation}
Let \(A\subseteq[0,1]\) be Borel and let \(\tau:A\to[0,1]\) be Borel.
Suppose that, for some \(q>0\), \(\tau(a)\geq a+q\) for every \(a\in A\).
If a finite signed Borel measure \(\mu\) on \(A\) satisfies
\[
    \mu=\varepsilon\,\tau_\#\mu
    \qquad\text{as measures on }[0,1]
\]
for some \(\varepsilon\in\{-1,1\}\), where \(\mu\) is extended by zero outside \(A\), then \(\mu=0\).
\end{lemma}

\begin{proof}
Extend \(\mu\) by zero on \(A^c\) to a finite signed Borel measure on \([0,1]\).
The identity
\(
    \mu=\varepsilon\,\tau_\#\mu
\)
means that, for every Borel set \(B\subseteq[0,1]\),
\begin{align}\label{eq:pushforward-identity}
    \mu(B)
    =
    \varepsilon\,\mu\bigl(\tau^{-1}(B)\bigr).
\end{align}
Choose \(N\in\N\) minimal such that \(Nq>1\), and define
\begin{align*}
    E_j
    &:=
    [jq,(j+1)q),
    \qquad j=0,\ldots,N-2,\\
    E_{N-1}
    &:=
    [(N-1)q,1].
\end{align*}
Then \((E_j)_{j=0}^{N-1}\) is a partition of \([0,1]\).

We show inductively that the restriction of \(\mu\) to each \(E_j\) vanishes.
For \(j=0\), let \(B\subseteq E_0\) be Borel.
If \(a\in A\), then
\[
    \tau(a)\geq a+q\geq q,
\]
and hence \(\tau(a)\notin E_0\).
Thus \(\tau^{-1}(B)=\varnothing\), and \eqref{eq:pushforward-identity} yields \(\mu(B)=0\).
Consequently, \(\mu|_{E_0}=0\).

Now let \(j\in\{1,\ldots,N-1\}\) and suppose that \(\mu\big|_{E_0\cup\cdots\cup E_{j-1}}=0\).
For a Borel set \(B\subseteq E_j\) and \(a\in\tau^{-1}(B)\), we have
\[
    \tau(a)\in E_j
    \qquad\text{and}\qquad
    a\leq\tau(a)-q.
\]
If \(j\leq N-2\), then \(\tau(a)<(j+1)q\), and therefore \(a<jq\).\\
For \(j=N-1\), the minimality of \(N\) gives \(Nq>1\), and hence
\[
    a\leq 1-q<(N-1)q.
\]
Thus, in either case,
\[
    \tau^{-1}(B)
    \subseteq
    E_0\cup\cdots\cup E_{j-1}.
\]
By the induction hypothesis,
\(
    \mu(\tau^{-1}(B))=0
\),
and \eqref{eq:pushforward-identity} implies
\(
    \mu(B)=0
\).
Hence \(\mu|_{E_j}=0\).

We conclude, that by induction, \(\mu\) vanishes on every \(E_j\), and since these sets
partition \([0,1]\), we we must have \(\mu=0\).
\end{proof}

We make use of the following characterization of the contact set for the centered dual potentials constructed in Lemma \ref{lem:global-dual-feasibility}. Recall \(\Gamma_{f_x,\theta}
    =
    \bigl\{
        (a,b)\in T:
        f_x(a)+f_x(b)
        +\theta\bigl((b-a)-m\bigr)
        =(b-a)^2
    \bigr\} \).

\begin{lemma}[Characterization of the contact set of \((f_x,\theta)\)]
\label{lem:contact-set-characterization}
Let \(x<1\), and let \((f_x,\theta)\) be the feasible dual
potential in \eqref{def_fxvgx}. Define
\begin{align}\label{def_rx_taux_ax}
    r_x(a)
    &:=
    \frac{\theta-g_x'(a)}{2}, \qquad
    \tau_x(a)
    :=
    a+r_x(a),
    \qquad \text{and} \qquad
    A_x:=\{a\in[0,1]:\tau_x(a)\leq1\}.
\end{align}
Then the contact set of the dual potential \((f_x,\theta)\) is given by
\begin{align*}
    \Gamma_{f_x,\theta}
    =
    \bigl\{
        (a,\tau_x(a)):a\in A_x
    \bigr\}.
\end{align*}
In particular, for every \(a\in[0,1]\), there exists at most one
\(b\geq a\) such that \((a,b)\in\Gamma_{f_x,\theta}\).
\end{lemma}

\begin{proof}
Recall \eqref{def_fxvgx} that
\(
    f_x=g_x+\frac{\theta m}{2}.
\)
Hence, writing \(r=b-a\), the slack in the constraint of the dual problem \eqref{eq:OT_dual2} is
\begin{align*}
    &(b-a)^2
    -f_x(a)-f_x(b)
    -\theta\bigl((b-a)-m\bigr) 
    =
    k_x(r)-g_x(a)-g_x(a+r)
    =
    G_x(a,r);
\end{align*}
where the first equality follows from definition of \(f_x\) and \(k_x\) in \eqref{def_fxvgx} and \eqref{def:k_x}.
The last equality is by definition of \(G_x\) in \eqref{def_Gxar}.
Consequently,
\begin{align}
    (a,b)\in\Gamma_{f_x,\theta}
    \quad\Longleftrightarrow\quad
    G_x(a,b-a)=0.
    \label{eq:contact-G-zero}
\end{align}

Notice that \(r_x\) in \eqref{def_rx_taux_ax} coincides with the function \(r_*\) introduced in \eqref{eq:def-r-star}.
By the strict convexity argument following \eqref{eq_delGr}, for every fixed \(a\), the map
\(
    r\longmapsto G_x(a,r)
\)
is strictly convex on \([0,\theta]\).
Moreover, \eqref{eq_delGr}, \eqref{eq:def-r-star}, and \eqref{eq:opposite-slope-identity} imply that \(r_x(a)\in(0,\theta)\) is its unique minimizer, while \eqref{eq:G-unique-zero} gives
\begin{align*}
    G_x(a,r)=0
    \quad\Longleftrightarrow\quad
    r=r_x(a),
    \qquad r\in[0,\theta].
\end{align*}
It remains to exclude further zeros for \(r>\theta\).
Write
\[
    r=r_0+n\theta,
    \qquad
    r_0\in[0,\theta),\quad n\in\mathbb N_0.
\]
By \eqref{eq:G-periodic-increment},
\begin{align*}
    G_x(a,r)-G_x(a,r_0)
    =
    n\theta\bigl(2r_0+(n-1)\theta\bigr).
\end{align*}
For \(n\geq1\), the right-hand side of the above equation is strictly positive except when \((n,r_0)=(1,0)\).
In this exceptional case,
\[
    G_x(a,\theta)=G_x(a,0)>0,
\]
because the unique zero of \(G_x(a,\cdot)\) on \([0,\theta]\) is the interior point \(r_x(a)\).
Therefore,
\begin{align*}
    G_x(a,r)=0
    \quad\Longleftrightarrow\quad
    r=r_x(a)
\end{align*}
throughout the admissible range \(0\leq r\leq1-a\).

Combining this with \eqref{eq:contact-G-zero}, equality in the dual constraint holds precisely when
\[
    b=a+r_x(a)=\tau_x(a).
\]
The condition \(b\leq1\) is equivalent to \(a\in A_x\), and hence
\(
    \Gamma_{f_x,\theta}
    =
    \{(a,\tau_x(a)):a\in A_x\}.
\)
\end{proof}

\begin{proposition}[Optimality and uniqueness of the copulas \(C_x\)]
\label{prop:rho-phi-Cx}
For every \(x\in[-\frac12,1]\), the copula \(C_x\) in \eqref{def_C_x} uniquely maximizes Spearman's \(\varrho\) among all copulas \(C\) satisfying \(\phi(C)=x\).
In particular,
\[
    \varrho(C_x)=\overline{\varrho}(x).
\]
\end{proposition}

\begin{proof}
We first show optimality of \(C_x\) for \(x\in[-\frac12,1)\).
Therefore, recall \(m=m(x) = \frac{1-x}{3}\) and \(\theta=\theta(x)\) due to \eqref{eq:mean-gap} and \eqref{def:vptheta}.
By Lemma \ref{lem_nux_in_PiT} and Lemma~\ref{lem:primal-feasibility}, the measure \(\nu_x\) defined in \eqref{def_nu_x} satisfies
\[
    \nu_x\in\Pi_T
    \qquad\text{and}\qquad
    \int_T(b-a)\,\de\nu_x(a,b)=\frac{m}{2}.
\]

Consider the centered potential \(f_x(u)=g_x(u)+\frac{\theta m}{2}\) for \(u\in[0,1]\).
Then, by Lemma~\ref{lem:global-dual-feasibility}, the pair \((f_x,\theta)\) is feasible, i.e.,
\[
    f_x(a)+f_x(b)
    +\theta\bigl((b-a)-m\bigr)
    \leq
    (b-a)^2 \quad \text{for all } (a,b)\in T.
\]
Moreover, Lemma~\ref{lem:equality-on-support} yields
\[
    g_x(a)+g_x(b)
    =
    k_x(b-a)
    =
    (b-a)^2-\theta(b-a)
\]
for \(\nu_x\)-almost every \((a,b)\in T\).
Consequently,
\begin{align*}
    &f_x(a)+f_x(b)
    +\theta\bigl((b-a)-m\bigr)
    =
    g_x(a)+g_x(b)+\theta(b-a)
    =
    (b-a)^2
\end{align*}
for \(\nu_x\)-almost every \((a,b)\in T\).
Thus, \(\nu_x\) is concentrated on the contact set of the feasible pair \((f_x,\theta)\).

Corollary~\ref{cor_feas} therefore implies that \(\nu_x\) is optimal for the triangular problem, i.e.,
\[
    \min_{\substack{\nu\in\Pi_T\\
    \int_T(b-a)\,\de\nu=m/2}}
    \int_T(b-a)^2\,\de\nu
    =
    \int_T(b-a)^2\,\de\nu_x.
\]
By the representation of \(\overline{\varrho}\) in \eqref{eq:lp}, it follows that
\[
    \overline{\varrho}(x)
    =
    1-12\int_T(b-a)^2\,\de\nu_x(a,b).
\]
On the other hand, since \(\pi_x=\nu_x+S_\#\nu_x\), we have
\[
    \varrho(C_x)
    =
    1-6\int_{[0,1]^2}(b-a)^2\,\de\pi_x(a,b)
    =
    1-12\int_T(b-a)^2\,\de\nu_x(a,b).
\]
Therefore, \(\varrho(C_x)=\overline{\varrho}(x)\).

We next prove uniqueness of the optimizer \(C_x\) for \(x\in[-\frac12,1)\).
The argument proceeds in two steps.
First, we prove uniqueness of the optimizer of the triangular problem.
By Lemma~\ref{lem:contact-set-characterization}, the contact set of the dual pair \((f_x,\theta)\) is the graph of the map \(\tau_x\).
Since this dual pair is optimal, complementary slackness forces every triangular optimizer to be concentrated on this graph.
The difference of two such optimizers is therefore determined by a finite signed measure \(\mu\) on the first coordinate.
The common marginal constraint then implies \(\mu=-(\tau_x)_\#\mu.\) Since \(\tau_x\) shifts every point by a strictly positive amount, Lemma~\ref{lem:no-upward-circulation} yields \(\mu=0\).

In the second step, we exclude nonsymmetric maximizers of the original copula problem.
Any maximizer has the same symmetrization as the candidate optimizer.
Hence, its difference from the candidate is an antisymmetric signed measure concentrated on the contact graph and its transpose.
Restricting this measure to the upper triangle yields an invariance relation under \(\tau_x\), and a second application of Lemma~\ref{lem:no-upward-circulation} completes the proof.

For the first step, recall from \eqref{def_rx_taux_ax} that \(\tau_x(a) = a+\frac{\theta-g_x'(a)}{2}. \) Since \(\lvert g_x'\rvert\leq v<\theta\) by \eqref{ineq_gp}, we have
\begin{align}\label{eq:tau-positive-shift}
    \tau_x(a)
    \geq
    a+q_x \qquad \text{for}
    \quad
    q_x:=\frac{\theta-v}{2}>0.
\end{align}

We first note that the particular dual pair \((f_x,\theta)\) is optimal.
Indeed, \((f_x,\theta)\) is dual feasible by Lemma~\ref{lem:global-dual-feasibility}, while \(\nu_x\) is primal feasible.
Moreover, Lemma~\ref{lem:equality-on-support} shows that equality in the dual constraint holds \(\nu_x\)-almost surely.
Consequently, the primal value attained by \(\nu_x\) coincides with the dual value attained by \((f_x,\theta)\).
Thus, \(\nu_x\) is primal optimal and \((f_x,\theta)\) is dual optimal.

Let now \(\widetilde\nu\) be any optimizer of the triangular problem in \eqref{eq:lp}.
Since \((f_x,\theta)\) is dual optimal, equality of the primal and dual values yields
\begin{align*}
    0
    =
    \int_T
    \Bigl[
        (b-a)^2
        -f_x(a)-f_x(b)
        -\theta\bigl((b-a)-m\bigr)
    \Bigr]
    \,\widetilde\nu(\de a,\de b).
\end{align*}
The integrand is nonnegative by dual feasibility.
Hence it must vanish \(\widetilde\nu\)-almost surely, so \(\widetilde\nu\) is concentrated on the contact set of \((f_x,\theta)\).
By Lemma~\ref{lem:contact-set-characterization},
\begin{align*}
    \Gamma_{f_x,\theta}
    =
    \{(a,\tau_x(a)):a\in A_x\}.
\end{align*}
The same is true for \(\nu_x\).
Now, set
\begin{align*}
    \zeta:=\widetilde\nu-\nu_x,
    \qquad
    \mu:=(\operatorname{proj}_1)_\#\zeta.
\end{align*}
Since both \(\widetilde\nu\) and \(\nu_x\) are concentrated on \(\Gamma_{f_x,\theta}\), so is the signed measure \(\zeta\).
Moreover, the projection onto the first coordinate is a bijection from \(\Gamma_{f_x,\theta} = \{(a,\tau_x(a)):a\in A_x\} \) onto \(A_x\), with inverse \(a\mapsto(a,\tau_x(a))\).
Consequently,
\begin{align}\label{proof_uniqueness2}
    \zeta
    &=
    (\operatorname{id},\tau_x)_\#\mu,
    &
    (\operatorname{proj}_2)_\#\zeta
    &=
    (\tau_x)_\#\mu.
\end{align}
Since \(\widetilde\nu\) and \(\nu_x\) are in \(\Pi_T\) and thus satisfy the same coupled marginal constraint in \eqref{def:Pi_T}, their difference satisfies
\begin{align}\label{proof_uniqueness3}
    (\operatorname{proj}_1)_\#\zeta
    +
    (\operatorname{proj}_2)_\#\zeta
    =0.
\end{align}
Therefore, equations \eqref{proof_uniqueness2} and \eqref{proof_uniqueness3} yield
\begin{align*}
    \mu=-(\tau_x)_\#\mu.
\end{align*}
By \eqref{eq:tau-positive-shift}, the map \(\tau_x\) satisfies the assumptions of Lemma~\ref{lem:no-upward-circulation}.
Applying the lemma with \(\varepsilon=-1\) gives \(\mu=0\), and hence \(\zeta=0\).
Consequently, \(\widetilde\nu=\nu_x.\) Thus the optimizer of the triangular problem is unique.

In the second step, we exclude nonsymmetric maximizers of the original copula problem.
Let \(\widetilde\pi\) be the coupling associated with any maximizing copula at the prescribed value \(x\), and define its symmetrization by
\begin{align*}
    \widetilde\pi^{\operatorname{sym}}
    :=
    \frac12
    \bigl(
        \widetilde\pi+S_\#\widetilde\pi
    \bigr).
\end{align*}
Since both the objective and the constraint are symmetric, \(\widetilde\pi^{\operatorname{sym}}\) is again optimal.
Its triangular representative is therefore \(\nu_x\) by the uniqueness proved above.
Hence, by Lemma~\ref{lem:pair-measure-bijection}, \(\widetilde\pi^{\operatorname{sym}}=\pi_x.\)

Now, set
\begin{align*}
    \delta:=\widetilde\pi-\pi_x.
\end{align*}
Since \(\pi_x\) is symmetric, the preceding identity implies \(S_\#\delta=-\delta.\) Moreover, both marginals of \(\delta\) vanish, and
\begin{align}\label{eq:sym-delta-support}
    \widetilde\pi+S_\#\widetilde\pi=2\pi_x.
\end{align}

Consider the symmetrized contact set \(K := \Gamma_{f_x,\theta} \cup S_\#\Gamma_{f_x,\theta}. \) Since the triangular representative \(\nu_x\) of \(\pi_x\) is concentrated on \(\Gamma_{f_x,\theta}\), the symmetric coupling \(\pi_x\) is concentrated on \(K\).
Hence \(\pi_x(K^c)=0\).
Evaluating \eqref{eq:sym-delta-support} on \(K^c\) gives
\begin{align*}
    \widetilde\pi(K^c)
    +
    S_\#\widetilde\pi(K^c)
    =0.
\end{align*}
Both terms are nonnegative, and hence \(\widetilde\pi(K^c)=0\).
Thus also \(\delta\) is concentrated on \(K\).

By \eqref{eq:tau-positive-shift}, \(\Gamma_{f_x,\theta}\) lies strictly above the diagonal, whereas \(S(\Gamma_{f_x,\theta})\) lies strictly below it.
In particular, \(K\) does not intersect the diagonal.
Define
\begin{align*}
    \eta:=\delta|_{\{a<b\}}.
\end{align*}
Since \(S_\#\delta=-\delta\) and \(\delta\) has no mass on the diagonal, we obtain
\begin{align*}
    \delta=\eta-S_\#\eta.
\end{align*}
Since the first marginal of \(\delta\) vanishes, it follows that \(    (\operatorname{proj}_1)_\#\eta = (\operatorname{proj}_2)_\#\eta. \) By the concentration of \(\delta\) on \(K\) established above and by \eqref{eq:tau-positive-shift}, the restriction \(\eta=\delta|_{\{a<b\}}\) is concentrated on \(\Gamma_{f_x,\theta}\), which, by Lemma~\ref{lem:contact-set-characterization}, is the graph of \(\tau_x\).
Setting
\[
    \mu:=(\operatorname{proj}_1)_\#\eta,
\]
we therefore have
\[
    \eta=(\operatorname{id},\tau_x)_\#\mu
    \qquad\text{and}\qquad
    (\tau_x)_\#\mu
    =
    (\operatorname{proj}_2)_\#\eta
    =
    (\operatorname{proj}_1)_\#\eta
    =
    \mu.
\]
Using \eqref{eq:tau-positive-shift}, we may apply Lemma~\ref{lem:no-upward-circulation} once more, now with \(\varepsilon=1\).
This gives \(\mu=0\), and hence \(\eta=0\).
Therefore \(\delta=0\), so \(\widetilde\pi=\pi_x.\) Thus the maximizing coupling, and hence the maximizing copula, is uniquely given by \(C_x\).

Finally, let \(x=1\).
Then \(\phi(C)=1\) implies \(\E|U-V|=0\), and hence \(U=V\) almost surely.
Thus, the only possible copula is the comonotonic copula \(C_1(u,v)=\min\{u,v\}\) for which \(\varrho(C_1)=1=\overline{\varrho}(1).\) This completes the proof.
\end{proof}

\section{Proofs of Section \ref{sec:intro}}\label{sec:proof1}

\begin{proof}[Proof of Theorem \ref{thm:main}.]
Fix \(x\in[-\frac12,1)\), and let \(N\in\mathbb N\) be such that \(x\in I_N = \left[1-\frac{3}{2N}, 1-\frac{3}{2N+2}\right)\).
Recall \(m=\frac{1-x}{3}\), \(L=\frac{1}{2N}\), and \(R=\frac{1}{2N+2}\) as well as \(I_N^L = \left[ 1-3L, 1-\frac 3 2 (L+R)\right)\) and \(I_N^R= \left[1-\frac 3 2 (L+R), 1-3R\right)\).

If \(x\in I_N^L\), then
\[
    m\in\left(\frac{L+R}{2},L\right].
\]
Hence, in this case, it is \(\ell=L\) and \(\Delta=L-m\).
By Proposition~\ref{prop_C_x}, we then have
\[
    \varrho(C_x)
    =
    1-6L(2m-L)
    -
    \frac{4(L-m)^{3/2}}{\sqrt{N(N+1)}} = 1-6\ell(2m-\ell)
    -
    \frac{4\Delta^{3/2}}{\sqrt{N(N+1)}}.
\]

If \(x\in I_N^R\), then
\[
    m\in\left(R,\frac{L+R}{2}\right].
\]
In this case, it is \(\ell=R\) and \(\Delta=m-R\).
Again by Proposition~\ref{prop_C_x}, it follows that
\[
    \varrho(C_x)
    =
    1-6R(2m-R)
    -
    \frac{4(m-R)^{3/2}}{\sqrt{N(N+1)}}
    =
    1-6\ell(2m-\ell)
    -
    \frac{4\Delta^{3/2}}{\sqrt{N(N+1)}}.
\]
At the common boundary point of \(I_N^L\) and \(I_N^R\), we have \( m=\frac{L+R}{2}\).
In this case, both \(L\) and \(R\) are equally close to \(m\), and \(L(2m-L)=LR=R(2m-R)\), while \(L-m=m-R=\frac{L-R}{2}\).
Hence, the two expressions coincide.

By Proposition~\ref{prop:rho-phi-Cx}, the copula \(C_x\) uniquely maximizes Spearman's rho among all copulas with Spearman's footrule equal to \(x\).
Consequently,
\[
    \overline{\varrho}(x)
    =
    \varrho(C_x)
    =
    1-6\ell(2m-\ell)
    -
    \frac{4\Delta^{3/2}}{\sqrt{N(N+1)}}.
\]
In particular, the maximum is uniquely attained by the symmetric copula \(C_x\).

It remains to verify the second representation in \eqref{eq:boundary}.
Recall that
\[
    u(x)
    =
    1-\frac{2}{3}(1-x)^2
    =
    1-6m^2.
\]
Since \(m=\ell\pm\Delta\), we have in either case \(m^2 = \ell(2m-\ell)+\Delta^2 \).
Therefore,
\begin{align*}
    \overline{\varrho}(x)
    &=
    1-6\ell(2m-\ell)
    -
    \frac{4\Delta^{3/2}}{\sqrt{N(N+1)}}\\
    &=
    1-6m^2+6\Delta^2
    -
    \frac{4\Delta^{3/2}}{\sqrt{N(N+1)}}\\
    &=
    u(x)
    -
    \left(
        \frac{4\Delta^{3/2}}{\sqrt{N(N+1)}}
        -
        6\Delta^2
    \right).
\end{align*}
This proves \eqref{eq:boundary}.
The assertion at \(x=1\) follows from Proposition~\ref{prop:rho-phi-Cx}.
\end{proof}

\begin{proof}[Proof of Corollary \ref{cor:main}.]
The sharp lower bound for Spearman's rho in terms of Spearman's footrule gives
\[
    \varrho(C)
    \geq
    \underline{\varrho}(x)
    =
    \frac{2\sqrt{3}}{9}(1+2x)^{3/2}-1
\]
for every copula \(C\) satisfying \(\phi(C)=x\), and this lower bound is attained for every \(x\in[-\frac12,1]\); see \cite[Theorem 8]{kokolbukovsek2024exact}.
Together with the upper bound established in Theorem \ref{thm:main}, every copula \(C\) satisfies
\[
    \underline{\varrho}(\phi(C))
    \leq
    \varrho(C)
    \leq
    \overline{\varrho}(\phi(C)).
\]
It follows that
\[
    \Omega_{\varrho,\phi}
    \subseteq
    \left\{
        (\varrho,x):
        x\in[-\tfrac12,1],\
        \underline{\varrho}(x)
        \leq \varrho\leq
        \overline{\varrho}(x)
    \right\}.
\]

To prove the reverse inclusion, fix \(x\in[-\frac12,1]\).
Let \(C_x^+\coloneqq C_x\) be the copula in \eqref{def_C_x} attaining the upper boundary, and let \(B_x\) be the Bertino copula in \eqref{def:lower-bertino} attaining the lower boundary.
Thus
\[
    \phi(C_x^+)=\phi(B_x)=x, \qquad 
    \varrho(C_x^+)=\overline{\varrho}(x),
    \qquad \text{and} \qquad
    \varrho(B_x)=\underline{\varrho}(x).
\]
For \(t\in[0,1]\), define
\[
    C_{x,t}
    :=
    tC_x^++(1-t)B_x.
\]
Since the class of copulas is convex, \(C_{x,t}\) is again a copula.
Moreover, both \(\phi\) and \(\varrho\) are affine in the copula.
Hence,
\[
    \phi(C_{x,t})
    =
    t\phi(C_x^+)+(1-t)\phi(B_x)
    =
    x,
\qquad \text{and} \qquad
    \varrho(C_{x,t})
    =
    t\overline{\varrho}(x)
    +(1-t)\underline{\varrho}(x).
\]
As \(t\) ranges over \([0,1]\), the latter expression ranges over the entire interval \([\underline{\varrho}(x),\overline{\varrho}(x)]\).
Thus, every point between the lower and upper boundary is attained.

Finally, for \(x=1\), the identity
\[
    \phi(C)=1
    \quad\Longrightarrow\quad
    \E|U-V|=0
\]
implies \(U=V\) almost surely.
Hence, the only possible copula is \(M\), and \(\varrho(M)=\phi(M)=1\).
Therefore,
\[
    \Omega_{\varrho,\phi}
    =
    \left\{
        (\varrho,x):
        x\in[-\tfrac12,1],\
        \frac{2\sqrt{3}}{9}(1+2x)^{3/2}-1
        \leq \varrho\leq
        \overline{\varrho}(x)
    \right\},
\]
as claimed.
\end{proof}

\begin{proof}[Proof of Proposition \ref{prop:sharp-CS}]
For \(m=0\), we have \(V_{\operatorname{min}}(0)=0\), and \eqref{eq:interpretation-sharp-cs} reduces to \(\E(U-V)^2\geq0\).
The bound is attained by the comonotone coupling \(U=V\).

Now let \(m>0\) and put \(x=1-3m\).
By \eqref{eq_rep_phi_rho}, maximizing \(\varrho(C)\) under the constraint \(\phi(C)=x\) is equivalent to minimizing \(\E(U-V)^2\) over all couplings \(U,V\sim\cU(0,1)\) satisfying \(\E|U-V|=m\).
Hence, Theorem~\ref{thm:main} yields
\begin{align*}
    \min_{\substack{U,V\sim\cU(0,1)\\ \E|U-V|=m}}
    \E(U-V)^2
    =
    \ell(2m-\ell)
    +
    \frac{2\Delta_m^{3/2}}{3\sqrt{N(N+1)}}.
\end{align*}
Since
\(
    \ell(2m-\ell)=m^2-\Delta_m^2,
\)
we obtain
\begin{align*}
    \min_{\substack{U,V\sim\cU(0,1)\\ \E|U-V|=m}}
    \E(U-V)^2
    =
    m^2+
    \frac{2\Delta_m^{3/2}}{3\sqrt{N(N+1)}}-\Delta_m^2
    =
    m^2+V_{\operatorname{min}}(m),
\end{align*}
which proves \eqref{eq:interpretation-sharp-cs}.
By Theorem~\ref{thm:main}, the minimum is attained by \(C_x\), proving sharpness.

It remains to characterize the zeros of the correction term.
Since \(\ell\) is an endpoint nearest to \(m\), we have \(0\leq\Delta_m\leq\frac1{4N(N+1)}\).
Consequently,
\begin{align*}
    V_{\operatorname{min}}(m)
    &=
    \Delta_m^{3/2}
    \left(
        \frac2{3\sqrt{N(N+1)}}-\sqrt{\Delta_m}
    \right),
\end{align*}
and
\begin{align*}
    \frac2{3\sqrt{N(N+1)}}-\sqrt{\Delta_m}
    \geq
    \frac1{6\sqrt{N(N+1)}}
    >0.
\end{align*}
Thus \(V_{\operatorname{min}}(m)=0\) if and only if \(\Delta_m=0\).
By the choice
\(
    m\in(\frac1{2N+2},\frac1{2N}],
\)
this is equivalent to \(m=1/(2N)\).
Together with the case \(m=0\), this proves \(V_{\operatorname{min}}(m)=0\) if and only if \(m\in\mathfrak C\).
\end{proof}

\begin{proof}[Proof of Corollary \ref{cor:rank-gap-variance}.]
Put \(Z\coloneqq|U-V|\), \(m\coloneqq\E Z\), and \(x\coloneqq1-3m\).
By Lemma~\ref{lem:moments},
\[
    \Var(Z)=\frac{1-\varrho(C)}{6}-m^2.
\]
Substituting the upper boundary \(\brho(x)\) from Theorem~\ref{thm:main} gives \(V_{\operatorname{min}}(m)\), while substituting
\[
    \underline{\varrho}(1-3m)
    =
    2(1-2m)^{3/2}-1
\]
gives \(V_{\operatorname{max}}(m)\).
Both bounds are attained by \(C_x\) and \(B_x\), respectively.

It remains to maximize \(V_{\operatorname{max}}\) on \([0,\tfrac12]\).
Its derivative is
\[
    V_{\operatorname{max}}'(m)
    =
    \sqrt{1-2m}-2m,
\]
so its unique maximizer is \(m_0=(\sqrt5-1)/4\).
A direct substitution gives
\[
    V_{\operatorname{max}}(m_0)
    =
    \frac{5(3-\sqrt5)}{24},
\]
and \(x_0=1-3m_0=(7-3\sqrt5)/4\).
This proves \eqref{eq_cor:rank-gap-variance} and the equality statements.
\end{proof}

\section{Proofs of Section \ref{sec:appl_main_thm}}\label{sec:proof2}

\begin{proof}[Proof of Theorem \ref{cor:permutation-costs}.]
	We embed the permutation into a copula.
	Partition \((0,1]\) into the rank cells \(I_i\coloneqq((i-1)/n,i/n]\), and define the measure-preserving map
	\[
		Q_\pi(0)\coloneqq0,
		\qquad
		Q_\pi(u)\coloneqq u+\frac{\pi(i)-i}{n},
		\qquad u\in I_i.
	\]
	If \(U\) is uniform on \([0,1]\) and \(V=Q_\pi(U)\), then \(V\) is also uniform, so the copula of \((U,V)\) is a shuffle of \(M\) \cite{mikusinski1992shuffles}.
	Conditional on \(U\in I_i\), the difference \(V-U\) is the constant \((\pi(i)-i)/n\).
	Therefore
	\[
		\E|U-V|=m_\pi,
		\qquad
		\E(U-V)^2=q_\pi.
	\]
	In other words, the normalized finite-ranking costs in \eqref{eq:permutation-costs} are exactly the two copula moments considered above.
	Since
	\[
		q_\pi
		=\E|U-V|^2
		=m_\pi^2+\Var(|U-V|),
	\]
	Corollary~\ref{cor:rank-gap-variance} gives \eqref{eq:permutation-envelope}.
\end{proof}

\begin{proof}[Proof of Lemma \ref{lem:psi-mixability-centers}.]
Let \(U:=U'+\tfrac12\) and \(V:=\tfrac12-V'\).
Then \(U,V\sim\cU(0,1)\) and, by \eqref{eq:centered-distance},
\begin{align*}
    \Psi(U',V')
    =
    |U-V|.
\end{align*}
Hence, the assertion is equivalent to the existence of a coupling \(U,V\sim\cU(0,1)\) such that \(|U-V|=m\) almost surely.
This is equivalent to equality in the Cauchy--Schwarz inequality \eqref{eq:CSimp}.
By Proposition~\ref{prop:sharp-CS}, such equality is attainable if and only if \(m\in\mathfrak C\).
\end{proof}

\begin{proof}[Proof of Theorem \ref{cor:distance-mixability}.]
For \(U',V'\sim\cU(-\tfrac12,\tfrac12)\), define \(U:=U'+\tfrac12\) and \(V:=\tfrac12-V'\).
Then \(U,V\sim\cU(0,1)\) and, by \eqref{eq:centered-distance}, \(\Psi(U',V') = |U-V|. \) Hence,
\begin{align*}
    \mathsf V(m)
    =
    \min\left\{
        \Var(|U-V|):
        U,V\sim\cU(0,1),\
        \E|U-V|=m
    \right\}.
\end{align*}
Since \(\Var(|U-V|) = \E(U-V)^2-m^2\), Proposition~\ref{prop:sharp-CS} implies \(\mathsf V(m) = V_{\operatorname{min}}(m).\) The final assertion follows from Proposition~\ref{prop:sharp-CS}(iii).
\end{proof}

\begin{proof}[Proof of Theorem \ref{cor:xi-rank-sobol}.]
Let \(C\) be the copula of \((X,Y)\).
Then Corollary~\ref{cor:main} implies for \(x = \xi(C) = \phi(C\ast C)\) that
\begin{align*}
    \underline{\varrho}(x) \leq \varrho(C\ast C) \leq \brho(x).
\end{align*}
By non-negativity of the copula correlation ratio and its representation in \eqref{eq_rep_ccorr}, we obtain
\begin{align}\label{eq_cor21eq3}
    \max\!\left\{
        0,\underline{\varrho}\bigl(\xi(X,Y)\bigr)
    \right\}
    \leq
    \eta(X,Y)
    \leq   \brho(\xi(X,Y)).
\end{align}
For \(U = F_Y(Y)\), the statement now follows from
\begin{align*}
    \eta(X,Y) = 12 \E\left(\int_0^1 \left(P(U\leq t \mid X) - t\right) \de t\right)^2 \leq 12 \E \int_0^1 \left(P(U\leq t \mid X) - t \right)^2 \de t = 2 \xi(X,Y),
\end{align*}
applying Cauchy--Schwarz inequality.
\end{proof}

\begin{proof}[Proof of Proposition \ref{prop:xi-rank-sobol-inner}.]

We first show that \(\Omega_{\xi,\eta}\) is convex.
Let \((X_1,Y_1)\) and \((X_2,Y_2)\), with continuous marginal distribution functions, attain \((\xi_1,\eta_1)\) and \((\xi_2,\eta_2)\), respectively.
Set
\[
    U_j:=F_{Y_j}(Y_j),
    \qquad
    T_j:=F_{X_j}(X_j),
    \qquad j=1,2.
\]
Then \(U_j,T_j\sim\cU(0,1)\).
Moreover, since \(F_{X_j}\) is continuous, conditioning on \(T_j\) is equivalent to conditioning on \(X_j\), up to null sets.

Let \(B\sim\operatorname{Bernoulli}(\alpha)\), \(\alpha\in[0,1]\), be independent of both models, with \(P(B=1)=\alpha\), and define
\begin{align*}
    \widetilde X
    &:=
    \begin{cases}
        \alpha T_1, & B=1,\\
        \alpha+(1-\alpha)T_2, & B=0,
    \end{cases}
    &
    \widetilde U
    &:=
    \begin{cases}
        U_1, & B=1,\\
        U_2, & B=0.
    \end{cases}
\end{align*}
Conditionally on \(B=1\), \(\widetilde X\) is uniform on \([0,\alpha]\), whereas conditionally on \(B=0\), it is uniform on \([\alpha,1]\).
Hence \(\widetilde X\sim\cU(0,1)\).
Moreover, since \(U_1,U_2\sim\cU(0,1)\), also \(\widetilde U\sim\cU(0,1)\).
Thus, setting \(\widetilde Y:=\widetilde U\), both marginal distributions of \((\widetilde X,\widetilde Y)\) are continuous.

On \(\{B=1\}\), we have \((\widetilde{U}\mid \widetilde{X}) \eqd (U_1\mid X_1)\), where as on \(\{B=0\}\), it is \((\widetilde{U}\mid \widetilde{X}) \eqd (U_2\mid X_2)\).
Let \(\widetilde U'\) denote a conditionally independent copy of \(\widetilde U\) given \(\widetilde X\).
Then
\begin{align*}
    \E|\widetilde U-\widetilde U'|
    =
    \alpha\E|U_1-U_1'|
    +(1-\alpha)\E|U_2-U_2'|,
\end{align*}
and therefore
\begin{align}\label{eq_convxi12}
    \xi(\widetilde X,\widetilde Y)
    =
    \alpha\,\xi_1(X_1,Y_1)+(1-\alpha)\,\xi_2(X_2,Y_2).
\end{align}
Further, for \(M_j:=\E[U_j\mid X_j]\), we have \(\E M_j=1/2\) and thus
\begin{align*}
    \Var\bigl(\E[\widetilde U\mid\widetilde X]\bigr)
    &=
    \alpha\Var(M_1)
    +(1-\alpha)\Var(M_2).
\end{align*}
This gives
\begin{align}\label{eq_conveta12}
    \eta(\widetilde X,\widetilde Y)
    =
    \alpha\eta_1+(1-\alpha)\eta_2.
\end{align}
Combining \eqref{eq_convxi12} and \eqref{eq_conveta12} yields \(\alpha(\xi_1,\eta_1) +(1-\alpha)(\xi_2,\eta_2) \in\Omega_{\xi,\eta}\), which proves that \(\Omega_{\xi,\eta}\) is convex.

For the lower inner curve \(\eta_\ell\), let \(t\in[0,1]\), let \(U\sim\cU(0,1)\), and put \(a_t:=(1-t)/2\) and \(b_t:=(1+t)/2\).
Define
\begin{align*}
    X_t
    :=
    \begin{cases}
        \min\{U,1-U\}, & U\notin[a_t,b_t],\\
        U, & U\in[a_t,b_t].
    \end{cases}
\end{align*}
Outside the interval \([a_t,b_t]\), the level sets
\(\{u\colon X_t(u)=x\}\) of \(X_t\) are the symmetric pairs
\(\{x,1-x\}\), whereas inside the interval they are singletons.
The distribution function of \(X_t\) is continuous.
Consequently, for \(x<a_t\), conditionally on \(X_t=x\), the rank
\(U\) is equally likely to be \(x\) or \(1-x\), whereas for
\(x\in[a_t,b_t]\), it is uniquely determined by \(U=x\).
Direct integration yields
\begin{align*}
    \xi(X_t,U)
    &=
    \frac{1+3t^2}{4} \quad \text{and} \quad 
    \eta(X_t,U)
    =
    12\Var\bigl(\E[U\mid X_t]\bigr)
    =
    t^3.
\end{align*}
Eliminating \(t\) gives
\begin{align*}
    \eta
    =
    \left(\frac{4\xi-1}{3}\right)^{3/2},
    \qquad
    \frac14\leq\xi\leq1,
\end{align*}
which is the second branch of \(\eta_\ell\).
The endpoint \(t=0\) gives \((1/4,0)\), while independence gives \((0,0)\).
By convexity, the entire segment \(\{(x,0):0\leq x\leq1/4\}\) is attainable by models with continuous marginals, proving the first branch of \(\eta_\ell\).

For the upper inner curve \(\eta_u\), fix \(a\in[0,1/2]\) and set
\begin{align*}
    h_a(s):=\min\{s,a,1-s\},
    \qquad 0\leq s\leq1.
\end{align*}
For \(\sigma\in\{-1,1\}\), define
\begin{align}\label{def_G_sigma}
    G_\sigma(s):=s+\sigma h_a(s).
\end{align}
The functions \(G_{-1}\) and \(G_1\) are distribution functions and satisfy
\begin{align}\label{prop_G_sigma}
    \frac{G_{-1}(s)+G_1(s)}{2}=s.
\end{align}
Let \(S\) be uniformly distributed on \(\{-1,1\}\) and, conditionally on \(S=\sigma\), let \(U\) have distribution function \(G_\sigma\).
Then, by \eqref{prop_G_sigma}, \(U\sim\cU(0,1)\).

To obtain a continuous conditioning variable, let \(R\sim\cU(0,1)\) be independent of \((S,U)\) and define
\begin{align*}
    X
    :=
    \begin{cases}
        R/2, & S=-1,\\
        (1+R)/2, & S=1.
    \end{cases}
\end{align*}
Then \(X\sim\cU(0,1)\), and \(S\) is measurable with respect to \(X\).
Moreover, the additional variable \(R\) carries no information about \(U\) given \(S\).
Hence
\begin{align}\label{eq_mix_S}
    P(U\leq s\mid X)
    =
    P(U\leq s\mid S)
    =
    G_S(s).
\end{align}
Thus, two conditionally independent draws given \(X\) define the same conditional-i.i.d.\ model as conditioning directly on \(S\).
Since
\begin{align*}
    \int_0^1h_a(s)^2\,\de s
    &=
    a^2-\frac43a^3 \qquad \text{and} \qquad
    \int_0^1h_a(s)\,\de s
    =
    a(1-a),
\end{align*}
substitution into \eqref{eq:ciid-coordinates} gives
\begin{align*}
    \xi_0(a)
    &:=
    2a^2(3-4a)\qquad \text{and} \qquad
    \eta_0(a)
    :=
    12a^2(1-a)^2.
\end{align*}

Now let \(J\) be uniformly distributed on \(\{1,\ldots,n\}\), independently of \(S\), and let \(W\) denote the response rank in the preceding binary model, i.e.,
\[
    \operatorname{law}(W\mid S=\sigma)=G_\sigma,
    \qquad \sigma\in\{-1,1\}.
\]
Define
\[
    U:=\frac{J-1+W}{n}.
\]
Thus, conditionally on \(J=j\), the binary model is mapped affinely from \([0,1]\) onto the interval
\(
[(j-1)/n,j/n]
\).
The conditioning state is therefore the pair \((J,S)\): the variable \(J\)
determines the interval containing \(U\), while \(S\) determines the
conditional distribution of \(U\) within that interval.

To encode this discrete state by a continuous conditioning variable, set
\begin{align*}
    K:=2(J-1)+1_{\{S=1\}},
\end{align*}
and let \(R\sim\cU(0,1)\) be independent of \((J,S,U)\). Then
\begin{align*}
    X:=\frac{K+R}{2n}
\end{align*}
is uniformly distributed on \([0,1]\) and determines the pair \((J,S)\).
Hence the resulting model has continuous marginals and belongs to
\(\Omega_{\xi,\eta}\).

Conditionally on \((J,S)\), let \(W'\) be an independent copy of \(W\)
and put
\(
U'=(J-1+W')/n.
\)
Then
\begin{align*}
    |U-U'|
    =
    \frac{1}{n}|W-W'|,
\end{align*}
and hence
\begin{align*}
    1-\xi
    =
    3\E|U-U'|
    =
    \frac{1-\xi_0(a)}{n}.
\end{align*}
Therefore,
\begin{align*}
    \xi
    =
    1-\frac{1-\xi_0(a)}{n}
    =
    x_n(a).
\end{align*}

For the copula correlation ratio, let
\(
\mu_S:=\E[W\mid S].
\)
Then
\begin{align*}
    \E[U\mid J,S]
    =
    \frac{J-1+\mu_S}{n}
    =
    \frac{J-\frac12}{n}
    +
    \frac{\mu_S-\frac12}{n}.
\end{align*}
Thus the conditional mean consists of the midpoint of the \(J\)-th interval and the binary perturbation inherited from the base model.
Since \(J\) and \(S\) are independent,
\begin{align*}
    \Var\bigl(\E[U\mid J,S]\bigr)
    &=
    \frac{1}{n^2}
    \left(
        \Var(J-1)+\Var(\mu_S)
    \right)
    =
    \frac{1}{12n^2}
    \left(
        n^2-1+\eta_0(a)
    \right),
\end{align*}
where we used
\(
\Var(J-1)=(n^2-1)/12
\)
and
\(
\eta_0(a)=12\Var(\mu_S).
\)
Consequently,
\begin{align*}
    \eta
    =
    12\Var\bigl(\E[U\mid J,S]\bigr)
    =
    1-\frac{1-\eta_0(a)}{n^2}
    =
    y_n(a).
\end{align*}
For every fixed \(n\in\N\), the curve
\(   \left\{
        (x_n(a),y_n(a)):a\in[0,\tfrac12]
    \right\}
\)
is compact, and these curves converge uniformly to \((1,1)\) as \(n\to\infty\).
Hence \(\mathcal S_{u}\) is compact.
Since \(\Omega_{\xi,\eta}\) is convex,
\(\operatorname{conv}(\mathcal S_{u})
    \subseteq
    \Omega_{\xi,\eta}.
\)
Moreover, the projection of \(\mathcal S_{u}\) onto its first coordinate is \([0,1]\).
Indeed, for \(x\in [0,1)\), write \(q=1-x\) and choose \(n=\lfloor\tfrac1q\rfloor\).
Then \(1-nq\in[0,1/2]\), which belongs to the range of \(a\mapsto2a^2(3-4a)\) for \(a\in [0,1/2]\).
Thus there exists \(a\in[0,1/2]\) such that \(x_n(a)=x\).
Since \(\operatorname{conv}(\mathcal S_u)\) is compact, the maximum defining \(\eta_u(x)\) is attained.
Its concavity follows from the convexity of \(\operatorname{conv}(\mathcal S_{u})\).

It remains to verify that the lower construction does not exceed the upper one.
Fix \(x\) and choose \((n,a)\) such that \(x_n(a)=x\).
Put \(r:=1-\xi_0(a)\), so that \(1-x=r/n\).
For \(x\geq1/4\),
\begin{align*}
    \eta_\ell(x)
    &=
    \left(1-\frac{4r}{3n}\right)^{3/2}
    \leq
    1-\frac{4r}{3n}
    \leq
    1-\frac{1-\eta_0(a)}{n^2}
    =
    y_n(a),
\end{align*}
where the second inequality follows from
\begin{align*}
    \frac43\bigl(1-\xi_0(a)\bigr)
    -
    \bigl(1-\eta_0(a)\bigr)
    =
    \frac{1+12a^2-40a^3+36a^4}{3}
    \geq0.
\end{align*}
For \(x\leq1/4\), the same ordering follows from nonnegativity.
Finally, convexity of \(\Omega_{\xi,\eta}\) fills the vertical segment between the attained lower and upper points for every fixed \(x\), proving \eqref{eq:xi-rank-sobol-inner}.
\end{proof}

\section*{Acknowledgement}

The first author was funded in whole by the Austrian Science Fund (FWF) {[10.55776/PAT1669224]} project \emph{Stochastic orders for functional dependence}.

\bibliographystyle{plainnat}
\bibliography{Spearman_arXiv}

\end{document}